\documentclass[11pt]{amsart}

\RequirePackage{etex}
\usepackage[utf8]{inputenc}
\usepackage[T1]{fontenc}
\usepackage[english]{babel}

\usepackage[mathscr]{euscript}
\usepackage{caption}
\usepackage{subcaption}
\usepackage{datetime}
\usepackage{url}
\usepackage{verbatim}
\usepackage{float}
\usepackage{lipsum}
\usepackage{setspace}

\usepackage{titletoc}
\usepackage{textcomp}
\usepackage{anysize}
\usepackage{fancyhdr}
\usepackage{calrsfs}

\usepackage{amssymb}
\usepackage{amsmath,mathtools}
\usepackage{amsfonts}
\usepackage{amscd}
\usepackage{amsthm}

\usepackage{latexsym}
\usepackage{enumerate}
\usepackage[shortlabels]{enumitem}
\usepackage{array}
\usepackage{stackrel}
\usepackage{stmaryrd}
\usepackage{stackengine}

\usepackage{graphicx}
\graphicspath{ {Images/} }
\usepackage{color,graphics}
\usepackage{tikz}
	\usetikzlibrary{arrows,automata,patterns, decorations.pathmorphing, decorations.pathreplacing, decorations.markings,tikzmark} 
\usepackage{tikz-cd}
	\usetikzlibrary{matrix,arrows,calc,automata,trees}
\usepackage[all,cmtip,color]{xy}
\usepackage{pgfplots,caption}
\pgfplotsset{compat=1.9}
\usepackage{ftnxtra}

\usepackage{fnpos}

\usepackage{lscape}
\usepackage[active]{srcltx}
\usepackage{rotating}

\usepackage[hypertexnames=false]{hyperref}

\makeatletter
\renewcommand{\tocsection}[3]{%
  \indentlabel{\@ifnotempty{#2}{\bfseries\ignorespaces#1 #2\quad}}\bfseries#3}
\renewcommand{\tocsubsection}[3]{%
  \indentlabel{\@ifnotempty{#2}{\ignorespaces#1 #2\quad}}#3}
\newcommand\@dotsep{4.5}
\def\@tocline#1#2#3#4#5#6#7{\relax
  \ifnum #1>\c@tocdepth % then omit
  \else
    \par \addpenalty\@secpenalty\addvspace{#2}%
    \begingroup \hyphenpenalty\@M
    \@ifempty{#4}{%
      \@tempdima\csname r@tocindent\number#1\endcsname\relax
    }{%
      \@tempdima#4\relax
    }%
    \parindent\z@ \leftskip#3\relax \advance\leftskip\@tempdima\relax
    \rightskip\@pnumwidth plus1em \parfillskip-\@pnumwidth
    #5\leavevmode\hskip-\@tempdima{#6}\nobreak
    \leaders\hbox{$\m@th\mkern \@dotsep mu\hbox{.}\mkern \@dotsep mu$}\hfill
    \nobreak
    \hbox to\@pnumwidth{\@tocpagenum{\ifnum#1=1\bfseries\fi#7}}\par% <-- \bfseries for \section page
    \nobreak
    \endgroup
  \fi}
\AtBeginDocument{%
\expandafter\renewcommand\csname r@tocindent0\endcsname{0pt}
}
\def\l@subsection{\@tocline{2}{0pt}{2.5pc}{5pc}{}}
\makeatother

\fancypagestyle{my_own_style}{
		\fancyhf{}
		\fancyfoot[C]{\thepage}
}

\newcommand{\N}{{\mathbb N}}
\newcommand{\Z}{{\mathbb Z}}
\newcommand{\Q}{{\mathbb Q}}
\newcommand{\C}{{\mathbb C}}
\newcommand{\R}{{\mathbb R}}

\newcommand{\V}{{\mathbb V}}

\DeclareMathAlphabet{\pazocal}{OMS}{zplm}{m}{n}

\newcommand{\calA}{{\pazocal A}}
\newcommand{\calB}{{\pazocal B}}
\newcommand{\calC}{{\pazocal C}}

\newcommand{\calG}{{\pazocal G}}
\newcommand{\calH}{{\pazocal H}}

\newcommand{\calM}{{\pazocal M}}
\newcommand{\calN}{{\pazocal N}}
\newcommand{\calO}{{\pazocal O}}
\newcommand{\calP}{{\pazocal P}}

\newcommand{\calR}{{\pazocal R}}
\newcommand{\calS}{{\pazocal S}}

\newcommand{\calU}{{\pazocal U}}

\newcommand{\pazB}{{\mathcal B}}

\newcommand{\pazF}{{\mathcal F}}

\newcommand{\pazX}{{\mathcal X}}

\newcommand{\gotL}{{\mathfrak L}}

\newcommand{\gotR}{{\mathfrak R}}

\newcommand{\ra}{\rightarrow}

\newcommand{\act}{\curvearrowright}

\newcommand{\ol}{\overline}
\newcommand{\ul}{\underline}
\newcommand{\wh}{\widehat}

\newcommand{\myRule}[3][white]{\textcolor{#1}{\rule{#2}{#3}}}

\newcommand{\bigzero}{\mbox{\normalfont\Large\bfseries 0}}

\newcommand{\tr}{\operatorname{tr}}
\newcommand{\Tr}{\operatorname{Tr}}
\newcommand{\rk}{\operatorname{rk}}
\newcommand{\Rk}{\operatorname{Rk}}

\newcommand{\xddots}{%
  \raise 4pt \hbox {.}
  \mkern 6mu
  \raise 1pt \hbox {.}
  \mkern 6mu
  \raise -2pt \hbox {.}
}

\numberwithin{equation}{section}

\theoremstyle{plain}

\newtheorem{theorem}{Theorem}[section]
\newtheorem*{theorem*}{Theorem}
\newtheorem{lemma}[theorem]{Lemma}
\newtheorem{proposition}[theorem]{Proposition}

\theoremstyle{definition}

\newtheorem{definition}[theorem]{Definition}
\newtheorem*{definition*}{Definition}

\newtheorem{remark}[theorem]{Remark}

\newtheorem*{remark*}{Remark}
\newtheorem{question}[theorem]{Question}
\newtheorem*{question*}{Question}

\newtheorem*{assumption*}{Assumption}

\newtheorem*{sac*}{Strong Atiyah Conjecture (\textbf{SAC})}
\newtheorem*{bsac*}{Strong Atiyah Conjecture, bounded case (\textbf{BSAC})}

\title[$L^2$-Betti numbers of the lamplighter]{$L^2$-Betti numbers arising from the lamplighter group}
\author{Pere Ara}
\address[P. Ara]{Departament de Matem\`atiques, Edifici Cc, Universitat Aut\`onoma de Barcelona, 08193 Cerdanyola del Vall\`es (Barcelona), Spain.}
\address[]{Centre de Recerca Matem\`atica, Edifici Cc, Campus de Bellaterra, 08193 Cerdanyola del Vall\`es (Barcelona), Spain.}
\email{para@mat.uab.cat}
\author{Joan Claramunt}
\address[J. Claramunt]{Department of Mathematics and Statistics, Lancaster University, LA1 4YW Lancaster, United Kingdom.}
\address[]{Departamento de Matem\'atica, Universidade Federal de Santa Catarina, 88040-900 Florian\'opolis SC, Brazil.}
\email{j.claramunt@lancaster.ac.uk}


\thanks{Both authors were partially supported by DGI-MINECO-FEDER through the grants MTM2014-53644-P and MTM2017-83487-P, and by the Generalitat de Catalunya through the grant 2017-SGR-1725. The second named author was also partially supported by DGI-MINECO-FEDER through the grant BES-2015-071439, by the research funding Brazilian agency CAPES and by the ERC Starting Grant “Limits of Structures in Algebra and Combinatorics” No. 805495}

\date{\today}

\begin{document}

\pagestyle{plain}
 
\begin{abstract}
We apply a construction developed in a previous paper by the authors in order to obtain a formula which enables us to compute $\ell^2$-Betti numbers coming from a family of group algebras representable as crossed product algebras. As an application, we obtain a whole family of irrational $\ell^2$-Betti numbers arising from the lamplighter group algebra $\Q[\Z_2 \wr \Z]$. This procedure is constructive, in the sense that one has an explicit description of the elements realizing such irrational numbers. This extends the work made by Grabowski, who first computed irrational $\ell^2$-Betti numbers from the algebras $\Q[\Z_n \wr \Z]$, where $n \geq 2$ is a natural number. We also apply the techniques developed to the generalized odometer algebra $\calO(\overline{n})$, where $\overline{n}$ is a supernatural number. We compute its $*$-regular closure, and this allows us to fully characterize the set of $\calO(\overline{n})$-Betti numbers.
\end{abstract}

\maketitle

\tableofcontents

\normalsize

\section{Introduction and the Atiyah problem}\label{section-introduction.Atiyah}

%\textcolor{red}{comentar que, en el paper "Sylvester matrix rank..." es va dir que en aquest paper "Approximating the group algebra..." es faria servir l'aproximacio de les An per estudiar aproximacions de la clausura $*$-regular del lamplighter, pero aixo es fa en l'altre paper! Potser canviar els titols, o referenciar l'altre en aquest?...}

%In 1963, Atiyah and Singer \cite{AS} established a very nice relation between %the analytical index of an elliptic differential operator $D$, closely related to the dimension of the finite-dimensional space of solutions of the homogeneous equation $Df = 0$, and the topological index of $D$, defined in terms of topological data relating the \textit{compact} manifold $M$ over which $D$ is defined. This relation is nowadays called the \textit{Atiyah-Singer Index Theorem}.

In \cite{Ati} Atiyah (in collaboration with Singer) extended the Atiyah-Singer Index Theorem \cite{AS} to the non-compact setting; more concretely, to the arena of non-compact manifolds $M$ equipped with a free cocompact action of a discrete group $G \act M$ (i.e. with $M/G$ compact). In this situation, the space of solutions of the equation $Df = 0$, where $D$ is an elliptic differential operator, becomes an infinite-dimensional Hilbert space $\calH_D$, and the concept of analytic index is extended by using the theory of von Neumann algebras. This theory provides a well-defined notion of dimension for $\calH_D$, called the \textit{von Neumann dimension} of $\calH_D$. This new framework led Atiyah to define the so-called \textit{$\ell^2$-Betti numbers} as von Neumann dimensions of $\ell^2$-cohomology Hilbert modules \cite{Luck}.

%M. Atiyah \cite{Ati} introduced the notion of $\ell^2$-Betti numbers when generalizing the Atiyah-Singer Index Theorem to (not necessarily compact) Riemannian manifolds $M$ endowed with a free cocompact action of a countable discrete group $G \act M$ (i.e.with $M/G$ compact). In order to generalize his Index Theorem, Atiyah provided a well-defined notion of dimension for the space of solutions of the equation $Df = 0$, being $D$ a differential operator. This new framework led to define

%His motivation was to generalize the Atiyah-Singer Index Theorem to the non-compact setting. Although they were defined analytically, $\ell^2$-Betti numbers can be defined purely in terms of the group $G$. In this line, a real positive number $r$ is said to be an $\ell^2$-Betti number arising from $G$, with coefficients in a fixed field $K$, whenever there exists a matrix operator $T \in M_n(K[G])$ such that the von Neumann dimension of $\ker(T)$ is equal to $r$.

In his paper \cite{Ati} Atiyah observed that, in case $G$ is a finite group, the $\ell^2$-Betti numbers coincide, \emph{modulo rescaling by $|G|$}, with the previously-known Betti numbers. Thus in this situation, they are in fact \textit{rational} numbers, and indeed Atiyah posed the following question, which is nowadays commonly known as the \textit{Atiyah Conjecture}.

\begin{question*}[Atiyah]\label{conjecture-atiyah}
Is it possible to obtain irrational values of $\ell^2$-Betti numbers?
\end{question*}

Although they were defined analytically, $\ell^2$-Betti numbers can be defined in an algebraic setting, purely in terms of the group $G$ and without explicit mention of the manifold $M$ (see Section \ref{subsection-Betti.numbers} for a detailed explanation). In this line, a real positive number $r$ is said to be an $\ell^2$-Betti number arising from $G$, with coefficients in a fixed subfield $K \subseteq \C$ closed under complex conjugation, whenever there exists a matrix operator $T \in M_n(KG)$ such that the von Neumann dimension of $\ker \, T$ is equal to $r$. Following this algebraic point of view, stronger questions --in relation with the original Atiyah question-- were formulated over the years. One of its strongest versions is the so-called \textit{Strong Atiyah Conjecture}, considered by Schick in \cite[Definition 1]{Schick}.

\begin{sac*}\label{conjecture-strong.atiyah}
The set of $\ell^2$-Betti numbers arising from $G$ with coefficients in $K$ is contained in the subgroup $\sum_H \frac{1}{|H|}\Z$ of $\mathbb Q$, where $H$ ranges over the finite subgroups of $G$.
\end{sac*}

The lamplighter is precisely the first counterexample to the SAC, as proved by R. I. Grigorchuk and A. \.{Z}uk \cite{GZ,GLSZ}, see also \cite{DiSc} for a more elementary proof. More recently, the original Atiyah's question has been solved in the negative, and some authors, including Austin \cite{Aus}, Grabowski \cite{Gra14,Gra16} and Pichot, Schick and Żuk \cite{PSZ} have found examples of groups having irrational values of $\ell^2$-Betti numbers. In particular, Grabowski shows in \cite{Gra16} that there are \textit{transcendental} numbers that appear as $\ell^2$-Betti numbers of the lamplighter group. Nevertheless, the following version of the SAC is still open.

\begin{bsac*}\label{conjecture-bounded.strong.atiyah}
Suppose that there exists an upper bound for the orders of the finite subgroups of $G$. Then the SAC holds for $G$.
\end{bsac*}

In particular it is open for torsion-free groups, which gives a generalization of the famous Kaplansky's zero-divisor Conjecture. For an extensive study of Atiyah's original question and strong versions of it, see \cite{DLMSY,Jaik-survey,Jaik,JL2,Lin91,Lin93,Lin08,LLS,LS07,LS12,Luck,LL18}. 

During the last 40 years, the importance of both the SAC and BSAC has increased due to the wide variety of their consequences in several branches of mathematics. For instance, in differential geometry and topology the SAC has connections with the Hopf Conjecture on the possible sign of the Euler characteristic of a Riemannian manifold (\cite{Davis}, see also \cite[Chapters 10 and 11]{Luck}). In group theory, the SAC has implications relating group-theoretic properties of a group and its homological dimension. In particular, it is known that if a group has homological dimension one and satisfies the SAC, then it must be locally free \cite{KLL}. 

One of the main problems (following these lines) is to actually compute the whole set $\calC(G,K)$ of $\ell^2$-Betti numbers arising from $G$ with coefficients in $K$. In this paper, we uncover a portion of this set for the lamplighter group $\Gamma$ with coefficients in $\Q$. The group $\Gamma$ is defined to be the semidirect product of $\Z$ copies of the finite group $\Z_2$ by $\Z$, i.e.
$$\Gamma = \Big( \bigoplus_{i \in \Z} \Z_2 \Big) \rtimes_{\rho} \Z,$$
whose automorphism $\rho$ implementing the semidirect product is the well-known Bernoulli shift. A precursor of our work here can be found in the paper \cite{AG} by Ara and Goodearl. In that article, the authors attack this problem algebraically, by trying to uncover the structure of what is called the \textit{$*$-regular closure} $\calR_{K \Gamma}$ of the lamplighter group algebra $K \Gamma$ inside $\calU(\Gamma)$, the algebra of unbounded operators affiliated to the group von Neumann algebra $\calN(\Gamma)$ or, more algebraically, the classical ring of quotients of $\calN(\Gamma)$. The precise connection between the $*$-regular closure $\calR_{K\Gamma}$ and the set $\calC(\Gamma,K)$ has been provided recently by a result of Jaikin-Zapirain \cite{Jaik}, which states that the rank function on $\calR_{K\Gamma}$, obtained by restricting the canonical rank function on $\calU(\Gamma)$, is completely determined by its values on matrices over $K\Gamma$. Since $\calR_{K\Gamma}$ is $*$-regular, this can be rephrased in the form
$$\phi(K_0(\calR_{K\Gamma})) = \calG(\Gamma,K),$$
where $\phi$ is the state on $K_0(\calR_{K\Gamma})$ induced by the restriction of the rank function on $K\Gamma$, and $\calG(\Gamma,K)$ is the subgroup of $\R$ generated by $\calC(\Gamma,K)$ \cite[Proposition 4.1]{AC2} (See e.g. \cite{Ros} for the definition of the $K_0$-group of a ring). The algebra $\calR_{K\Gamma}$, together with tight connections with $\calC(\Gamma,K)$, has been recently studied by the authors in \cite{AC2}.

A key observation in light of the development of the work presented in this paper is to realize the lamplighter group algebra as a $\Z$-crossed product $*$-algebra
$$K\Gamma \cong C_K(X) \rtimes_T \Z$$
through the Fourier transform. This method was first (somewhat implicitly) used in \cite{DiSc}, and very explicitly in \cite{Aus}. Here $X = \{0,1\}^{\Z}$ is the Pontryagin dual of the group $\bigoplus_{i \in \Z} \Z_2$, topologically identified with the Cantor set, $C_K(X)$ is the set of locally constant functions $f : X \ra K$, and $T : X \ra X$ is the homeomorphism of $X$ implemented by the Bernoulli shift. There is a natural measure $\mu$ on $X$, namely the usual product measure, having taken the $\big(\frac{1}{2},\frac{1}{2}\big)$-measure on each component $\{0,1\}$. This measure is ergodic, full and $T$-invariant, so we can apply the techniques developed in \cite{AC} to study the $\Z$-crossed product algebra $\calA := C_K(X) \rtimes_T \Z$ by giving `$\mu$-approximations' of the space $X$, which at the level of the algebra $\calA$ correspond to certain `approximating' $*$-subalgebras $\calA_n \subseteq \calA$ (see \cite[Section 4.1]{AC}, also \cite[Section 6]{AC2}). By using \cite[Theorem 4.7 and Proposition 4.8]{AC}, we obtain a canonical faithful Sylvester matrix rank function $\rk_{\calA}$ on $\calA$ which coincides, in case $K$ is a subfield of $\C$ closed under complex conjugation, with the rank function $\rk_{K\Gamma}$ on the group algebra $K\Gamma$ naturally inherited from the canonical rank function in the $*$-regular ring $\calU(\Gamma)$ \cite[Proposition 5.10]{AC2}. In light of this, one can define `generalized' $\ell^2$-Betti numbers in this more general setting, that is, arising from the $\Z$-crossed product algebra $\calA = C_K(X) \rtimes_T \Z$, for $K$ an \textit{arbitrary} field and $T$ an \textit{arbitrary} homeomorphism on a Cantor set $X$.

Turning back to the lamplighter group algebra $K\Gamma$, Graboswki has shown in a recent paper \cite{Gra16} the existence of irrational (in fact transcendental) $\ell^2$-Betti numbers arising from $\Gamma$, exhibiting a concrete example in \cite[Theorem 2]{Gra16}. With this result, the lamplighter group has become the simplest known example which gives rise to irrational $\ell^2$-Betti numbers. Very roughly, his idea is to compute $\ell^2$-Betti numbers by means of decomposing them as an infinite sum of (normalized) dimensions of kernels of finite-dimensional operators (i.e. matrices). He then realizes these matrices as adjacency-labeled matrices of certain graphs in order to determine the global behavior of the dimensions of their kernels. We use these ideas in this work, but applied to our construction. Our main result is the following:

\begin{theorem}[Theorem \ref{theorem-irrational.dimensions}]
Let $\Gamma$ be the lamplighter group. %, and $K$ any subfield of $\C$ closed under complex conjugation.
Then for any family of natural numbers $\{d_i \mid 1 \leq i \leq n\}$, and any family of polynomials $\{p_i(x) \mid 0 \leq i \leq n\}$ of positive degrees and non-negative integer coefficients, such that the linear coefficient of $p_0(x)$ is non-zero, there exists a matrix operator $A$ over $\Q \Gamma$ such that
$$b^{(2)}(A) = q_0 + q_1 \sum_{k \geq 1} \frac{1}{2^{p_0(k) + p_1(k)d_1^k + \cdots + p_n(k)d_n^k}}$$
for some non-zero rational numbers $q_0,q_1$. If moreover the constant coefficient of $p_0(x)$ is $0$, then
$$\sum_{k \geq 1} \frac{1}{2^{p_0(k) + p_1(k)d_1^k + \cdots + p_n(k)d_n^k}} \in \calG(\Gamma,\Q).$$
\end{theorem}
Here $b^{(2)}(A)$ stands for the \textit{$\ell^2$-Betti number of $A$} (see Definition \ref{definition-l2bettinumber.grouprings}). Thus we obtain a whole family of irrational and even transcendental $\ell^2$-Betti numbers arising from $\Gamma$. Another source of irrational $\ell^2$-Betti numbers arising from $\Gamma$ comes from the fact that the $*$-regular closure $\calR_{K\Gamma}$ contains a copy of the algebra of non-commutative rational power series $K_{\text{rat}}\langle \pazX \rangle$ in infinitely many indeterminates, see \cite[Subsection 6.2]{AC2} and Subsection \ref{subsection-rationalseries}. Using this, we show that $\calG (\Gamma,\Q)$ contains for instance the irrational algebraic number $\frac{1}{4}\sqrt{\frac{2}{7}}$ (Theorem \ref{thm:irrat-algl2betti}).

We also apply our machinery to study a particular crossed product algebra known as the odometer algebra. It is defined as the $\Z$-crossed product algebra $\calO := C_K(X) \rtimes_T \Z$ where $X = \{0,1\}^{\N}$ is the one-sided shift space, and the automorphism $T : X \ra X$ implementing the crossed product is given by addition of the element $(1,0,0,...)$ with carry over. Although it is not possible to realize the odometer algebra as a group algebra, this example is interesting in its own right because we are able to fully determine the structure of its $*$-regular closure $\calR_{\calO}$ (see Theorem \ref{theorem-reg.closure.odometer}), thus giving a complete description of the set of $\calO$-Betti numbers (Theorem \ref{theorem-betti.numbers.odometer}). The algebra $\calO$ has also been studied by Elek in \cite{Elek16}, although the author does not exactly compute the $*$-regular closure $\calR_{\calO}$; instead, the author computes its rank-completion, showing that it must be isomorphic to the von Neumann continuous factor $\calM_K$, which is by definition the rank-completion of $\varinjlim_i M_{2^i}(K)$ with respect to its unique rank metric (cf. \cite[Proposition 4.2]{AC2}). In fact we study a general version of the classical odometer algebra, namely the dynamical system generated by ``addition of $1$'' for arbitrary profinite completions of $\Z$.

This work is structured as follows. In Section \ref{section-preliminaries} we provide essential background and preliminary concepts about Sylvester matrix rank functions and $\ell^2$-Betti numbers for general group algebras. We summarize, in Section \ref{section-approx.crossed.product}, a general approximation construction for $\Z$-crossed product algebras using measure-theoretic tools, which enables us to construct a canonical Sylvester matrix rank function over the crossed product algebra. Using this construction, we derive a formula for computing generalized $\ell^2$-Betti numbers (see Definition \ref{definition-vN.dim.cross.prod}) over this large class of algebras (Formula \eqref{equation-Betti.no.final}). In Section \ref{section-lamplighter.algebra} we focus on the particular case of the lamplighter group algebra. We explicitly realize it as a $\Z$-crossed product algebra, thus enabling us to apply the whole theory developed in \cite{AC} and \cite{AC2}. Our main result (Theorem \ref{theorem-irrational.dimensions}) provides a vast family of irrational, and even transcendental, $\ell^2$-Betti numbers arising from the lamplighter group.

We also explore another method to find $\ell^2$-Betti numbers arising from the lamplighter group, using the algebra of non-commutative rational series. In this direction, we  obtain in Theorem \ref{thm:irrat-algl2betti} that the irrational algebraic number $\frac{1}{4}\sqrt{\frac{2}{7}}$ belongs to $\calG (\Gamma, \Q)$.

Finally we study, in Section \ref{section-odometer.alg}, the generalized odometer algebra $\calO(\ol{n})$ in great detail, and we are able to completely determine the algebraic structure of its $*$-regular closure (Theorem \ref{theorem-reg.closure.odometer}). We use this characterization in Theorem \ref{theorem-betti.numbers.odometer}, where we explicitly compute the whole set of $\calO(\ol{n})$-Betti numbers.

\section{Preliminaries}\label{section-preliminaries}

\subsection{\texorpdfstring{$*$}{}-regular rings and rank functions}\label{subsection-regular.rings.*.rank.functions}

 %Regularity is closed under taking extensions, ideals, direct products, matrices, direct limits, among others.

A \textit{$*$-regular ring} is a regular ring $R$ endowed with a \textit{proper} involution $*$, that is, $x^*x = 0$ if and only if $x = 0$. %The involution is called \textit{positive definite} in case, for each $n \geq 1$, the equation $\sum_{i=1}^n x_i^*x_i = 0$ implies $x_i = 0$ for all $1 \leq i \leq n$. If $R$ is a $*$-regular ring with a positive definite involution, then $M_n(R)$ is also a $*$-regular ring when endowed with the $*$-transpose involution.
In a $*$-regular ring $R$, for every $x \in R$ there exist unique projections $e, f \in R$ such that $xR = eR$ and $Rx = Rf$. It is common to denote them by $e = \mathrm{LP}(x)$ and $f = \mathrm{RP}(x)$, and are termed the \textit{left} and \textit{right projections} of $x$, respectively. %There also exists a unique element $\ol{x} \in fRe$,  termed the \textit{relative inverse} of $x$, so that $x\ol{x} = e$ and $\ol{x}x = f$.
We refer the reader to \cite{Ara87, Berb} for further information on $*$-regular rings.

%We say that two projections $e,f$ in a $*$-ring $R$ are \textit{equivalent}, written $e \overset{*}{\sim} f$, in case there is $x \in eRf$ such that $e = xx^*$ and $f = x^*x$.

For any subset $S \subseteq R$ of a unital $*$-regular ring, there exists a smallest unital $*$-regular subring of $R$ containing $S$ (\cite[Proposition 6.2]{AG}, see also \cite[Proposition 3.1]{LS12} and \cite[Proposition 3.4]{Jaik}). This $*$-regular ring is denoted by $\calR(S,R)$, and called the \textit{$*$-regular closure} of $S$ in $R$. %In fact,
%$$\calR(S,R) = \bigcup_{n \geq 0} \calR_n(S,R),$$
%where $\calR_0(S,R)$ is the $*$-subring of $R$ generated by the set $S$, and $\calR_{n+1}(S,R)$ is generated by $\calR_n(S,R)$ and the relative inverses in $R$ of the elements of $\calR_n(S,R)$. It was observed in \cite{Jaik} that $\calR_{n+1}(S,R)$ can be described as the subring of $R$ generated by the elements of $\calR_n(S,R)$ and the relative inverses of the elements of the form $x^*x$ for $x \in \calR_n(S,R)$.

%There is a whole theory in development concerning the study of the $*$-regular closure, initiated by Jaikin-Zapirain in \cite{Jaik}. 

\hspace{0.1cm}

%A \textit{pseudo-rank function} on a regular ring $R$ is a map $\rk : R \ra [0,1]$ that satisfies the following properties:
%\begin{enumerate}[a),leftmargin=1cm]
%\item $\rk(0) = 0$, $\rk(1) = 1$;
%\item $\rk(xy) \leq \min\{\rk(x),\rk(y)\}$ for every $x,y \in R$;
%\item if $e, f \in R$ are orthogonal idempotents, then $\rk(e+f) = \rk(e) + \rk(f)$.
%\end{enumerate}
%If $\rk$ satisfies the additional property
%\begin{enumerate}[d),leftmargin=1cm]
%\item $\rk(x) = 0$ only if $x = 0$,
%\end{enumerate}
%then $\rk$ is called a \textit{rank function} on $R$. 

%Every (pseudo-)rank function $\rk$ on a regular ring $R$ defines a (pseudo-) metric $d$ on $R$ by the rule $d(x,y) = \rk(x-y)$. Since the ring operations are continuous with respect to this (pseudo-)metric, one can consider the completion $\ol{R}$ of $R$ with respect to $d$. It is also a regular ring, and $\rk$ can be uniquely extended continuously to a rank function $\ol{\rk}$ on $\ol{R}$ such that $\ol{R}$ is also complete with respect to the metric induced by $\ol{\rk}$. It turns out that the completion $\ol{R}$ is also a right and left self-injective ring (see Theorems 19.6 and 19.7 of \cite{Goo91}).

%The space of pseudo-rank functions $\Ps(R)$ on a regular ring $R$ is a Choquet simplex (\cite[Theorem 17.5]{Goo91}), and the completion $\ol{R}$ of $R$ with respect to $\rk \in \Ps(R)$ is a simple ring if and only if $\rk$ is an extreme point in $\Ps(R)$ (\cite[Theorem 19.14]{Goo91}).

Let us denote by $M(R)$ the set of finite matrices over $R$ of arbitrary size, i.e. $\bigcup_{n \geq 1} M_n(R)$.

\begin{definition}\label{definition-sylvester.rank}
A \textit{Sylvester matrix rank function} on a unital ring $R$ is a map $\rk : M(R) \ra \R^+$ satisfying the following conditions:
\begin{enumerate}[a),leftmargin=1cm]
\item ${\rm rk} (0)= 0$ and ${\rm rk}(1)= 1$;
\item ${\rm rk} (M_1M_2) \leq \min\{{\rm rk}(M_1), {\rm rk}(M_2)\}$ for any matrices $M_1$ and $M_2$ of appropriate sizes;
\item ${\rm rk} \begin{pmatrix} M_1 & 0 \\ 0 & M_2 \end{pmatrix} = {\rm rk} (M_1) + {\rm rk}(M_2) $ for any matrices $M_1$ and $M_2$;
\item ${\rm rk} \begin{pmatrix} M_1 & M_3 \\ 0 & M_2  \end{pmatrix} \ge {\rm rk}(M_1) + {\rm rk}(M_2)$ for any matrices $M_1$, $M_2$ and $M_3$ of appropriate sizes.
\end{enumerate}
\end{definition}

The notion of Sylvester matrix rank function was first introduced by Malcolmson in \cite{Mal} on a question of characterizing homomorphisms from a fixed ring to division rings. Equivalent definitions exist, introduced by Malcolmson itself and Schofield \cite{Sch}. For more theory and properties about Sylvester matrix rank functions we refer the reader to \cite{Jaik} and \cite[Part I, Chapter 7]{Sch}.

In a regular ring $R$, any Sylvester matrix rank function $\rk$ is uniquely determined by its values on elements of $R$, see e.g. \cite[Corollary 16.10]{Goo91}. This is no longer true if $R$ is not regular.

%We denote by $\Ps(R)$ the compact convex set of Sylvester matrix rank functions on $R$. By \cite[Proposition 16.20]{Goo91} this space coincides with the space of pseudo-rank functions on $R$ when $R$ is a regular ring.

Any Sylvester matrix rank function $\rk$ on a unital ring $R$ defines a pseudo-metric by the rule $d(x,y) = \rk(x-y)$. The rank function is called \textit{faithful} if the only element with zero rank is the zero element. In this case, $d$ becomes a metric on $R$. %We can always achieve this situation by passing, if necessary, to the quotient $R/\ker(\rk)$, where $\ker(\rk) = \{x \in R \mid \rk(x) = 0\}$.
The ring operations are continuous with respect to $d$, and $\rk$ extends uniquely to a Sylvester matrix rank function $\ol{\rk}$ on the completion $\ol{R}$ of $R$ with respect to $d$.

For a $*$-subring $S$ of a $*$-regular ring $R$, there are tight connections between the structure of projections of the $*$-regular closure $\calR(S,R)$ and possible values of a Sylvester matrix rank function defined on $R$, see for instance \cite{Jaik, AC2}.

%\textcolor{red}{cal posar la proposicio que ve ara? potser millor posar-la a l'altre article i ja esta}

%A very useful result connecting the $*$-regular closure $\calR(S,R)$ and possible values of a Sylvester matrix rank function defined on $R$ is given in the following proposition, which can be thought of as an analogue of the classical Cramer's rule.

%\begin{proposition}[Corollary 6.2 of \cite{Jaik}]\label{proposition-cramer.rule.reg.closure}
%Let $S$ be a unital $*$-subring of a $*$-regular ring $R$, and let $\calR = \calR(S,R)$ be the $*$-regular closure of $S$ in $R$.

%Then for any matrices $r_1,...,r_k \in M_{n \times m}(\calR)$, there exists a matrix $M \in M_{a \times b}(S)$ and matrices $A_1,...,A_k \in M_{n \times b}(S)$ such that, for any other square-matrices $t_1,...,t_k \in M_n(S)$ and any Sylvester matrix rank function $\rk$ defined on $\calR$,
%$$\rk(t_1 r_1 + \cdots + t_k r_k) = \rk \begin{pmatrix}
%M \\
%t_1 A_1 + \cdots + t_k A_k
%\end{pmatrix} - \rk(M).$$
%In particular, any Sylvester matrix rank function on $\calR$ is completely determined by its values on matrices over $S$.
%\end{proposition}

\subsection{\texorpdfstring{$\ell^2$}{}-Betti numbers for group algebras}\label{subsection-Betti.numbers}

In this subsection we define $\ell^2$-Betti numbers arising from a group $G$ with coefficients in a subfield $K \subseteq \C$ closed under complex conjugation.

Let $G$ be a discrete, countable group. For any subring $R \subseteq \C$ closed under complex conjugation, let $RG$ denote the \textit{group $*$-algebra} of $G$ with coefficients in $R$, consisting of formal finite sums $\sum_{\gamma \in G} a_{\gamma} \gamma$ with $a_{\gamma} \in R$. The sum operation is defined pointwise, the product is induced by the group product and the $*$-operation is defined by linearity and according to the rule $(a_{\gamma}\gamma)^* = \ol{a_{\gamma}} \gamma^{-1}$, for $a_{\gamma} \in R$ and $\gamma \in G$. Let also $\ell^2(G)$ denote the Hilbert space of all square-summable functions $f : G \ra \C$ with obvious addition and scalar multiplication, and inner product defined by
$$\langle f,g \rangle_{\ell^2(G)} = \sum_{\gamma \in G} f(\gamma) \ol{g(\gamma)} \quad \text{ for } f,g \in \ell^2(G).$$
The space $\ell^2(G)$ has an orthonormal basis, naturally identified with $G$, consisting of indicator functions $\xi_{\gamma} \in \ell^2(G)$ for each $\gamma \in G$. Here $\xi_{\gamma}$ is defined to be $1$ over the element $\gamma$ and $0$ otherwise.

Observe that $G$ acts faithfully on $\ell^2(G)$ by left (resp. right) multiplication $\lambda : G \ra \calB(\ell^2(G))$ (resp. $\rho : G \ra \calB(\ell^2(G))$), defined by
$$\lambda_{\gamma}(f)(\delta) = f(\gamma^{-1}\delta) \quad \text{ and } \quad \rho_{\gamma}(f)(\delta) = f(\delta \gamma)$$
for $f \in \ell^2(G)$ and $\gamma,\delta \in G$. Either $\lambda$ or $\rho$ extend $R$-linearly to actions $RG \act \ell^2(G)$ by bounded operators, preserving the $*$-operation. We will identify $RG$ with the image of $\lambda$ inside $\calB(\ell^2(G))$.

We denote by $\calN(G)$ the weak-completion of $\C G \subseteq \calB(\ell^2(G))$, which is commonly known as the \textit{group von Neumann algebra of $G$}. An equivalent algebraic definition can be given: it consists exactly of those bounded operators $T : \ell^2(G) \ra \ell^2(G)$ that are $G$-equivariant, i.e. the relation $\rho_{\gamma} \circ T = T \circ \rho_{\gamma}$ is satisfied for every $\gamma \in G$. The algebra $\calN(G)$ is endowed with a normal, positive and faithful trace, defined as
$$\tr_{\calN(G)}(T) := \langle T(\xi_e),\xi_e \rangle_{\ell^2(G)} \quad \text{for } T \in \calN(G).$$
Note that for an element $T = \sum_{\gamma \in G} a_{\gamma} \gamma \in \C G$, its trace is simply the coefficient $a_e$.

All the above constructions can be extended to $k \times k$ matrices: the ring $M_k(RG)$ acts faithfully on $\ell^2(G)^k$ by left (resp. right) multiplication. We denote the extended actions by $\lambda_k$ and $\rho_k$, respectively. We identify $M_k(RG)$ with its image $M_k(RG) \subseteq \calB (\ell^2(G)^k)$ under $\lambda_k$. We denote by $\calN_k (G)$ the weak-completion of $M_k(\C G)$ inside $\calB (\ell^2(G)^k)$, which is easily seen to be equal to $M_k(\calN(G))$. The previous trace can be extended to an unnormalized trace over $M_k(\calN(G))$ by setting, for a matrix $T = (T_{ij}) \in M_k(\calN(G))$,
$$\Tr_{\calN_k (G)}(T) := \sum_{i=1}^k \tr_{\calN (G)}(T_{ii}).$$

A \textit{finitely generated Hilbert (right) $G$-module} is any closed subspace $V$ of $\ell^2(G)^k$, invariant with respect to the \textit{right} action $\rho^{\oplus k} := \rho \oplus \stackrel{k}{\cdots} \oplus \rho$. For $V \leq \ell^2(G)^k$ a finitely generated Hilbert $G$-module, the corresponding orthogonal projection operator $p_V : \ell^2(G)^k \ra \ell^2(G)^k$ onto $V$ belongs to $\calN_k(G)$. One then defines the \textit{von Neumann dimension of $V$} as the trace of $p_V$:
$$\dim_{\calN(G)}(V) := \Tr_{\calN_k (G)}(p_V).$$

In the particular case of matrix group rings $M_k(KG)$, being $K \subseteq \C$ a subfield closed under complex conjugation, every matrix operator $A \in M_k(KG)$ gives rise to an \textit{$\ell^2$-Betti number}, in the following way. Consider $A$ as an operator $A : \ell^2(G)^k \ra \ell^2(G)^k$ acting on the left, and take $p_A \in \calN_k (G)$ to be the projection onto $\ker \, A$, which is a finitely generated Hilbert $G$-module. One can then consider the von Neumann dimension of $\ker \, A$, which is simply the trace of the projection $p_A$.

\begin{definition}\label{definition-l2bettinumber.grouprings}
Let $A$ be a matrix operator in $M_k(KG)$ for some integer $k \geq 1$. We define the \textit{$\ell^2$-Betti number of $A$} by
$$b^{(2)}(A) := \dim_{\calN(G)}(\ker \, A) = \Tr_{\calN_k (G)}(p_A).$$
The set of all $\ell^2$-Betti numbers of operators $A \in M_k(KG)$ will be denoted by $\calC(G,K)$, and will be referred to as the set of all \textit{$\ell^2$-Betti numbers arising from $G$ with coefficients in $K$}. It should be noted that this set is always a subsemigroup of $(\R^+,+)$. We also write $\calG(G,K)$ for the subgroup of $(\R,+)$ generated by $\calC(G,K)$.
\end{definition}

It is also possible to define the von Neumann dimension by means of a \textit{Sylvester matrix rank function}, as follows. Let $\calU(G)$ be the algebra of unbounded operators affiliated to $\calN(G)$; equivalently, the classical ring of quotients of $\calN(G)$. It is a $*$-regular ring possessing a Sylvester matrix rank function $\rk_{\calU(G)}$ defined by
$$\rk_{\calU(G)}(U) := \Tr_{\calN_k(G)}(\text{LP}(U)) = \Tr_{\calN_k(G)}(\text{RP}(U))$$
for any matrix $U \in M_k(\calU(G))$, where $\text{LP}(U)$ and $\text{RP}(U)$ are the left and right projections of $U$ inside the $*$-regular algebra $M_k(\calU(G))$, respectively.
Notice that these projections actually belong to $\calN_k(G)$. In particular, we obtain by restriction a Sylvester matrix rank function $\rk_{KG}$ over $KG$. So for a matrix operator $A \in M_k(KG)$ we have $p_A=  1_k-\text{RP}(A)$ and we get the equality
\begin{equation}\label{equation-vN.dim.rank.UG}
b^{(2)}(A) = k - \rk_{KG}(A).
\end{equation}
Here $1_k$ stands for the identity matrix in $k$ dimensions.

%We will write $\rk_{K G}$ for $\Rk_{K G}$, i.e. when considering the case $k=1$.

\section{Approximating crossed product algebras through a dynamical perspective}\label{section-approx.crossed.product}

We recall the general construction used in \cite{AC} on approximating $\Z$-crossed product algebras.

Let $T : X \ra X$ be a homeomorphism of a totally disconnected, compact metrizable space $X$, which we also assume to be infinite (e.g. one can take $X$ to be the Cantor space). Let also $K$ be an arbitrary field endowed with a positive definite involution, that is, an involution $*$ such that for all $n\ge 1$ and  $a_1,\dots, a_n\in K$, we have $\sum_{i=1}^n a_i^*a_i= 0 \implies a_i=0$ for each $i=1,\dots ,n$.

The algebras of interest are $\Z$-crossed product algebras of the form
$$\calA := C_K(X) \rtimes_T \Z,$$
where $C_K(X)$ denotes the algebra of locally constant functions $f : X \ra K$; equivalently, the algebra of continuous functions $f : X \ra K$ when $K$ is endowed with the discrete topology. For the approximation process, we choose a $T$-invariant, ergodic and full probability measure $\mu$ on $X$. We refer the reader to \cite[Section 3]{AC} for a detailed exposition of the construction.

For any clopen subset $\emptyset \neq E \subseteq X$ and any (finite) partition $\calP$ of the complement $X \backslash E$ into clopen subsets, let $\calB$ be the unital $*$-subalgebra of $\calA$ generated by the partial isometries $\{\chi_Z t \mid Z \in \calP\}$. Here $t$ denotes the generator of the copy of $\Z$ inside $\calA$, and $\chi_A$ denotes the characteristic function of the set $A$. %By \cite[Lemma 3.4]{AC}, the $*$-algebra $\calB_0 = C_K(X) \cap \calB$ is linearly spanned by the unit $1$ and the projections of the form $\chi_C$, being $C$ a clopen subset of $X$ of the form
%\begin{equation}\label{equation-form.of.bs}
%T^{-r}(Z_{-r}) \cap \cdots \cap Z_{-1} \cap \cdots \cap T^{s-1}(Z_{s-1}),
%\end{equation}
%where $Z_{-r},...,Z_{-1},...,Z_{s-1} \in \calP$ and $r,s \geq 0$. Here $\chi_U$ denotes the characteristic function of the clopen subset $U \subseteq X$. We have a decomposition
%$$\calB = \bigoplus_{i \in \Z} \calB_0 (\chi_{X \backslash E}t)^i = \bigoplus_{i \in \Z} \calB_i t^i$$
%with $\calB_i = \chi_{X \backslash(E \cup \cdots \cup T^{i-1}(E))}\calB_0$ and $\calB_{-i} = \chi_{X \backslash(T^{-1}(E) \cup \cdots \cup T^{-i}(E))}\calB_0$ for $i > 0$. By \cite[Lemma 3.8]{AC}, for $i > 0$ the algebra $\calB_i$ is linearly spanned by the terms $\chi_C$, where $C$ is of the form \eqref{equation-form.of.bs} with $s \geq i$, and similarly $\calB_{-i}$ is linearly spanned by the terms $\chi_C$, where $C$ is of the form \eqref{equation-form.of.bs} with $r \geq i$.
There exists a quasi-partition of $X$ (i.e. a countable family of non-empty, pairwise disjoint clopen subsets whose union has full measure) given by the $T$-translates of clopen subsets $W$ of the form
\begin{equation}\label{equation-Wexpression}
W = E \cap T^{-1}(Z_1) \cap \cdots \cap T^{-k+1}(Z_{k-1}) \cap T^{-k}(E)
\end{equation}
for $k \geq 1$ and $Z_i \in \calP$, whenever these are non-empty. In fact, if we write $|W| := k$ (the length of $W$) and $\V := \{W \neq \emptyset \text{ as above}\}$, then for a fixed $W \in \V$ and $0 \leq i < |W|$ the element $\chi_{T^i(W)}$ belongs to $\calB$, and moreover the set of elements
$$e_{ij}(W) = (\chi_{X \backslash E}t)^i \chi_W (t^{-1} \chi_{X \backslash E})^j, \quad 0 \leq i,j < |W|,$$
forms a set of $|W| \times |W|$ matrix units in $\calB$ (that is, they satisfy $e_{ik}(W) e_{lj}(W) = \delta_{k,l} e_{ij}(W)$ for all indices $0 \leq i,j,k,l < |W|$). In addition, by \cite[Proposition 3.11]{AC} the element $h_W := e_{00}(W) + \cdots + e_{|W|-1,|W|-1}(W)$ is central in $\calB$ and we have a $*$-isomorphism
$$h_W \calB \cong M_{|W|}(K).$$
In this way one obtains an injective $*$-representation $\pi : \calB \hookrightarrow \prod_{W \in \V} M_{|W|}(K) =: \gotR_{\calB}$ defined by $\pi(a) = (h_W \cdot a)_W$ \cite[Proposition 3.13]{AC}.

From now on the $*$-algebra $\calB$ corresponding to $(E,\calP)$ as above will be denoted by $\calA(E,\calP)$.

\hspace{0.06cm}

Take now $\{E_n\}_{n \geq 1}$ to be a decreasing sequence of clopen sets of $X$ together with a family $\{\calP_n\}_{n \geq 1}$ consisting of (finite) partitions into clopen sets of the corresponding complements $X \backslash E_n$, satisfying:
\begin{enumerate}[a),leftmargin=1cm]
\item the intersection of all the $E_n$ consists of a single point $y \in X$;
\item $\calP_{n+1} \cup \{E_{n+1}\}$ is a partition of $X$ finer than $\calP_n \cup \{E_n\}$;
\item $\bigcup_{n \geq 1} (\calP_n \cup \{E_n\})$ generates the topology of $X$.
\end{enumerate}
By writing $\V_n$ for the set of all the non-empty subsets $W$ of the form \eqref{equation-Wexpression} corresponding to the pair $(E_n,\calP_n)$, and setting $\calA_n := \calA(E_n,\calP_n)$ and $\gotR_n := \prod_{W \in \V_n} M_{|W|}(K)$, we get injective $*$-representations $\pi_n : \calA_n \hookrightarrow \gotR_n$, in such a way that the diagrams
\begin{equation*}
\vcenter{
	\xymatrix{
	\calA_n \ar@{^{(}->}[r]^{\iota_n} \ar@{^{(}->}[d]^{\pi_n} & \calA_{n+1} \ar@{^{(}->}[d]^{\pi_{n+1}} \ar@{^{(}->}[r]^{\iota_{n+1}} & 	\calA_{n+2} \ar@{^{(}->}[d]^{\pi_{n+2}} \ar@{^{(}->}[r] & \cdots \ar@{^{(}->}[r] & \calA_{\infty} \ar@{^{(}->}[d]^{\pi_{\infty}} \\
	\gotR_n \ar@{^{(}->}[r]^{j_n} & \gotR_{n+1} \ar@{^{(}->}[r]^{j_{n+1}} & \gotR_{n+2} \ar@{^{(}->}[r] & \cdots \ar@{^{(}->}[r] & \gotR_{\infty}
	}
}\label{diagram-An.Rn}
\end{equation*}
commute. Here $\iota_n$ is the natural embedding $\iota_n(\chi_Z t) = \sum_{Z'} \chi_{Z'}t$ where the sum is taken with respect to all the $Z' \in \calP_{n+1}$ satisfying $Z' \subseteq Z$, the maps $j_n : \gotR_n \hookrightarrow \gotR_{n+1}$ are the embeddings given in \cite[Proposition 4.2]{AC}, and $\calA_{\infty}, \gotR_{\infty}$ are the inductive limits of the direct systems $(\calA_n,\iota_n), (\gotR_n,j_n)$, respectively.

\begin{remark}\label{remark-algebra.Ainfty}
The algebra $\calA_{\infty}$ can be explicitly described in terms of the crossed product, as follows. For $U \subseteq X$ an open set, denote by $C_{c,K}(U)$ the ideal of $C_K(X)$ generated by the characteristic functions $\chi_V$, where $V$ ranges over the clopen subsets $V \subseteq X$ contained in $U$. By \cite[Lemma 4.3]{AC}, $\calA_{\infty}$ coincides with the $*$-subalgebra of $\calA = C_K(X) \rtimes_T \Z$ generated by $C_K(X)$ and $C_{c,K}(X \backslash \{y\}) t$. (Recall from a) above that $\{ y \}= \bigcap_{n\ge 1} E_n$.)
\end{remark}

One can define a Sylvester matrix rank function on each $\gotR_n$ by the rule
\begin{equation}\label{equation-rank.over.Rn}
\rk_n(M) = \sum_{W \in \V_n} \mu(W) \Rk(M_W) \quad \text{ for } M = (M_W)_W \in \gotR_n,
\end{equation}
being $\Rk$ the usual rank of matrices. These Sylvester matrix rank functions are compatible with respect to the embeddings $j_n$, so they give rise to a well-defined Sylvester matrix rank function $\rk_{\infty}$ on $\gotR_{\infty}$. 

%\textcolor{red}{P: eliminada frase sobre Sylvester m.r.f sobre $\calA$, doncs no tenim encara tot $\calA$}

%In turn, $\rk_{\infty}$ induces a Sylvester matrix rank function $\rk_{\calA}$ over $\calA$, unique with respect to the property that $\rk_{\calA}(\chi_U) = \mu(U)$ for all clopen subsets $U \subseteq X$.%(a certain compatibility property concerning the measure $\mu$, as the following theorem states. Recall that $\calA = C_K(X) \rtimes_T \Z$, with $T : X \ra X$ a homeomorphism of the totally disconnected, compact metrizable space $X$, and $\mu$ is a $T$-invariant, ergodic and full probability measure on $X$.

\begin{theorem}\cite[Theorem 4.7 and Proposition 4.8]{AC}\label{theorem-rank.function}
Following the above notation, if $\gotR_{\rk} := \ol{\gotR_{\infty}}$ denotes the rank-completion of $\gotR_{\infty}$ with respect to its Sylvester matrix rank function $\rk_{\infty}$ (see Subsection \ref{subsection-regular.rings.*.rank.functions}), then there exists an injective $*$-homomorphism $\pi_{\rk} : \calA \ra \gotR_{\rk}$ making the diagram
\begin{center}
\begin{tikzcd}[remember picture]
    \calA \arrow[r,"\pi_{\rk}"] & \gotR_{\rk} \\
    \calA_{\infty} \arrow[r,"\pi_{\infty}"] & \gotR_{\infty}\\
\end{tikzcd}
\begin{tikzpicture}[overlay,remember picture]
\path (\tikzcdmatrixname-1-1) to node[midway,sloped]{$\supseteq$}
(\tikzcdmatrixname-2-1);
\path (\tikzcdmatrixname-1-2) to node[midway,sloped]{$\supseteq$}
(\tikzcdmatrixname-2-2);
\end{tikzpicture}
\end{center}

\vspace{-0.6cm}

\noindent commutative, and sending the element $t$ to $\pi_{\rk}(t) = \lim_n \pi_n(\chi_{X \backslash E_n}t)$. Moreover, we obtain a Sylvester matrix rank function $\rk_{\calA}$ on $\calA$ by restriction of $\ol{\rk_{\infty}}$ (the extension of $\rk_{\infty}$ to $\gotR_{\rk}$) on $\calA$, which is extremal and unique with respect to the following property:
\begin{equation}\label{equation-unique.property}
\rk_{\calA}(\chi_U) = \mu(U) \quad \text{ for every clopen subset } U \subseteq X.
\end{equation}

Finally, the rank-completion of $\calA$ with respect to $\rk_{\calA}$ gives back $\gotR_{\rk}$, that is $\ol{\calA} = \gotR_{\rk}$.
\end{theorem}

In particular, since $\gotR_{\rk}$ has the structure of a $*$-regular algebra, we can consider the $*$-regular closure of $\calA$ inside $\gotR_{\rk}$, which we denote by $\calR_{\calA} := \calR(\calA,\gotR_{\rk})$, see \cite{AC2} for more details. The $*$-regular closure $\calR_{\calA}$ is extensively studied in \cite{AC2}. For our purposes, it will become important when computing $\calO(\ol{n})$-Betti numbers arising from the generalized odometer algebra $\calO(\ol{n})$ in Section \ref{section-odometer.alg}, and also in Subsection \ref{subsection-rationalseries}.

\subsection{A formula for computing \texorpdfstring{$\calA$}{}-Betti numbers inside \texorpdfstring{$\calB$}{}}\label{subsection-formula.compute.betti.no}

In this subsection we give a formula for computing $\calA$-Betti numbers of elements from matrix algebras over the approximation algebra $\calB = \calA(E,\calP)$. We will use ideas from \cite{Gra16}, although applied to our construction.

By Theorem \ref{theorem-rank.function}, the algebra $\calA = C_K(X) \rtimes_T \Z$ possesses a `canonical' Sylvester matrix rank function $\rk_{\calA}$, unique with respect to Property \eqref{equation-unique.property}. It is then natural to ask which is the set of positive real numbers reached by such a Sylvester matrix rank function.

%As in Definition \ref{definition-l2bettinumber.grouprings}, this always has the structure of a semigroup, inherited from $(\R^+,+)$. The subgroup of $(\R,+)$ generated by $\calC(\calA)$ will be denoted by $\calG(\calA)$. We will see in Section \ref{section-lamplighter.algebra} that, in the special case of the lamplighter group algebra $K[\Gamma]$ with $K \subseteq \C$ any subfield closed under complex conjugation, the group $\calG(\calA)$ coincides with $\calG(\Gamma,K)$.  \textcolor{red}{(de fet per a productes semidirectes de grups mes general tambe coincideix, citar l'altre article!!)}

\begin{definition}\label{definition-vN.dim.cross.prod}
Let $A \in M_k(\calA)$. We define the \textit{$\calA$-Betti number of $A$} by (cf. Definition \ref{definition-l2bettinumber.grouprings}, see also Equation \eqref{equation-vN.dim.rank.UG})
$$b^{\calA}(A) := k - \rk_{\calA}(A).$$
\end{definition}

The set consisting of all $\calA$-Betti numbers of elements $A \in M_k(\calA)$ will be denoted by $\calC(\calA)$. This set has the structure of a semigroup of $(\R^+,+)$ (cf. Definition \ref{definition-l2bettinumber.grouprings}). The subgroup of $(\R,+)$ generated by $\calC(\calA)$ will be denoted by $\calG(\calA)$. It is shown in Section \ref{section-lamplighter.algebra} that the definition of $\calA$-Betti number coincides with Definition \eqref{definition-l2bettinumber.grouprings} in case $\calA$ is the lamplighter group algebra $\Gamma$, and thus $\calC(\calA) = \calC(\Gamma,K)$.\\

We now focus on the approximation algebra $\calB = \calA(E,\calP)$, where recall that $E$ is any non-empty clopen subset of $X$, and $\calP$ a (finite) partition of the complement $X \backslash E$ into clopen subsets. Let $\pi : \calB \ra \gotR_{\calB}$ be the faithful $*$-representation on $\gotR_{\calB} = \prod_{W \in \V} M_{|W|}(K)$ given by $\pi(a) = (h_W \cdot a)_W$. We extend $\pi$ to a faithful $*$-representation, also denoted by $\pi$, over matrix algebras
$$\pi : M_k(\calB) \ra M_k(\gotR_{\calB}) \cong \prod_{W \in \V} M_k(K) \otimes M_{|W|}(K)$$
in a canonical way. Hence $\rk_{\calA}$ can be computed over elements $A \in M_k(\calB)$ by means of the formula (recall \eqref{equation-rank.over.Rn})
$$\rk_{\calA}(A) = \sum_{W \in \V} \mu(W) \Rk(\pi(A)_W),$$
where $\Rk$ is the usual rank of matrices. Note that $\pi(A)_W \in M_k(K) \otimes M_{|W|}(K) = M_{k|W|}(K)$. We thus obtain the following proposition.

\begin{proposition}\label{proposition-vNdim.computation}
%Let $(E,\calP)$ be a pair formed by a non-empty clopen subset $E \subseteq X$ and $\calP$ a (finite) partition of $X \backslash E$ into clopen subsets, and let $\calB = \calA(E,\calP)$ be the corresponding approximation algebra.
With the above notation, for a given element $A \in M_k(\calB)$ we have the formula
$$b^{\calA}(A) = \sum_{W \in \V} \mu(W) \emph{dim}_K(\ker \, \pi(A)_W).$$
Here $\emph{dim}_K(\cdot)$ denotes the usual $K$-dimension of finite-dimensional $K$-vector spaces.
\end{proposition}
\begin{proof}
It is just a matter of computation, using that $\Rk(M) = m - \text{dim}_K(\ker \, M)$ for any matrix $M \in M_m(K)$ and that $\V$ forms a quasi-partition of $X$, so $\sum_{W \in \V} \mu(W) |W| = 1$.%\sum_{W \in \V}\sum_{l=0}^{|W|-1}\mu(T^l(W)) = 1$.%, we have
%\begin{align*}
%\text{dim}_{\calN(G)}(\ker \, A) & = n - \sum_{W \in \V} \Big(n |W| \mu(W) - \text{dim}_K(\ker \, \pi(A)_W) \mu(W) \Big) = \sum_{W \in \V} \text{dim}_K(\ker \, \pi(A)_W) \mu(W).\hfill\qedhere
%\end{align*}
\end{proof}

Fix now $\{e_{ij} \mid 0 \leq i,j < k \}$ a complete system of matrix units for $M_k(K)$. For a fixed $W \in \V$, the family $\{e_{ij} \otimes e_{i'j'}(W) \mid 0 \leq i,j < k;\, 0 \leq i',j' < |W| \}$ is a complete system of matrix units for $M_k(K) \otimes M_{|W|}(K)$, where the $e_{i'j'}(W)$ are the matrix units for $h_W\calB = M_{|W|}(K)$ introduced in Section \ref{section-approx.crossed.product}. In particular $\pi(e_{ij} \otimes 1)_W = e_{ij} \otimes h_W$. Let now $M_k(K) \otimes M_{|W|}(K)$ act on the $K$-vector space $K^k \otimes K^{|W|}$, with $K$-basis $\pazB = \{e_i \otimes e_{i'}(W) \mid 0 \leq i < k;\, 0 \leq i' < |W| \}$, by left matrix multiplication:
$$(e_{ij} \otimes e_{i'j'}(W)) \cdot (e_a \otimes e_{a'}(W)) = \delta_{j,a} \delta_{j',a'} \text{ } e_i \otimes e_{i'}(W).$$
Let $\langle \cdot , \cdot \rangle$ be the scalar product on $K^k \otimes K^{|W|}$ rendering the basis $\pazB$ orthonormal. %, namely $\langle e_i \otimes e_{i'}(W) , e_j \otimes e_{j'}(W) \rangle = \delta_{i,j}\delta_{i',j'}$.
Then for an element $A \in M_k(\calB)$, the entry of the matrix $\pi(A)_W$ corresponding to the $e_{ij} \otimes e_{i'j'}(W)$ component is given by
$$\langle \pi(A)_W \cdot (e_j \otimes e_{j'}(W)), e_i \otimes e_{i'}(W) \rangle.$$
One can think of the matrix $\pi(A)_W$ as the \textit{adjacency-labeled} matrix of an edge-labeled graph $E_A(W)$ defined as follows:
\begin{enumerate}[a),leftmargin=1cm]
\item The vertices are the pairs $(e_i, e_{i'}(W))$, for $0 \leq i < k$, $0 \leq i' < |W|$.
\item We have an arrow from $(e_j, e_{j'}(W))$ to $(e_i, e_{i'}(W))$ if and only if
$$\langle \pi(A)_W \cdot (e_j \otimes e_{j'}(W)), e_i \otimes e_{i'}(W) \rangle \neq 0.$$
In this case, we label the arrow as $d_{(j,j'),(i,i')} := \langle \pi(A)_W \cdot (e_j \otimes e_{j'}(W)), e_i \otimes e_{i'}(W) \rangle$.
\end{enumerate}
Here \textit{adjacency-labeled} matrix means that the coefficient of $\pi(A)_W$ corresponding to the $e_{ij} \otimes e_{i'j'}(W)$ component is exactly $d_{(j,j'),(i,i')}$. There is an example of such a graph in Figure \ref{figure-graph.EA.example}, corresponding to the matrix
$$\pi(A)_W = d_1 \cdot e_{11} \otimes e_{21}(W) + d_2 \cdot e_{k-2,k-2} \otimes e_{21}(W) + d_3 \cdot e_{k-1,k-2} \otimes e_{22}(W) + d_4 \cdot e_{11} \otimes e_{|W|-2,|W|-2}(W).$$

\begin{figure}[ht] \[
\xymatrix@=0.4cm{
		& e_0(W) & e_1(W) & e_2(W) & \cdots & e_{|W|-2}(W) & e_{|W|-1}(W) \\
%e_0 & \bullet & \bullet & \bullet & \cdots & \bullet & \bullet \\
e_0 & \bullet & \bullet & \bullet & \cdots & \bullet & \bullet \\
e_1 & \bullet & \bullet \ar@[red][r]^{d_1} & \bullet & \cdots & \bullet \ar@[red]@(ul,ur)_{d_4} & \bullet \\
 & \vdots & \vdots & \vdots & \ddots & \vdots & \vdots \\
e_{k-2} & \bullet & \bullet \ar@[red][r]^{d_2} & \bullet \ar@[red][d]^{d_3} & \cdots & \bullet & \bullet \\
e_{k-1} & \bullet & \bullet & \bullet & \cdots & \bullet & \bullet \\
} \]
\caption{An example of a graph $E_A(W)$}\label{figure-graph.EA.example}
\end{figure}
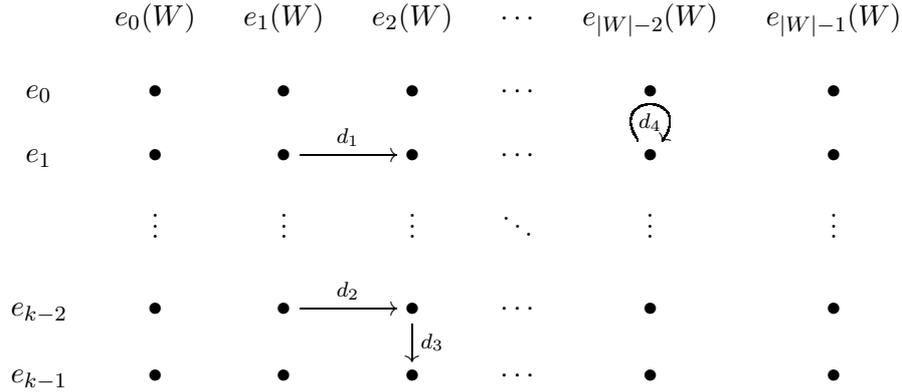

If we denote by $\calG\text{r}_A(W)$ the set consisting of the connected components of the graph $E_A(W)$ and by $A_C$, for $C \in \calG\text{r}_A(W)$, the adjacency-labeled matrix of the corresponding connected component $C$, then the matrix $\pi(A)_W$ is similar to the block-diagonal matrix $\text{diag}(A_{C_1},A_{C_2},...,A_{C_r})$, being $\calG\text{r}_A(W) = \{C_1,...,C_r\}$.
%\begin{equation*}
%\begin{pmatrix}
% \diagentry{A_{C_1}}\\
% &\diagentry{A_{C_2}}\\
% &&\diagentry{\xddots}\\
% &&&\diagentry{A_{C_r}}\\
%\end{pmatrix},\quad \text{ where } \calG_A(W) = \{C_1,...,C_r\}.
%\end{equation*}
As a consequence, and by using Proposition \ref{proposition-vNdim.computation}, we obtain the formula
\begin{equation}\label{equation-Betti.no.final}
\text{dim}_{\calA}(\ker \, A) = \sum_{W \in \V} \sum_{C \in \calG\text{r}_A(W)} \mu(W) \text{dim}_K(\ker \, A_C),
%\text{dim}_K(\ker \, \pi(A)_W) = \sum_{C \in \calG_A(W)} \text{dim}_K(\ker \, A_C),
\end{equation}
which will be used in Section \ref{section-lamplighter.algebra} for explicit computations of $\ell^2$-Betti numbers.

The following lemma, whose proof is trivial, is an adaptation of \cite[Lemma 20]{Gra16} in our notation, and provides a graphical way for computing dimensions of kernels of finite matrices.

\begin{lemma}[Flow Lemma at $(e_i, e_{i'}(W))$]\label{lemma-flow.lemma}
An element $\alpha = \sum_{j,j'} \lambda_{(j,j')} e_j \otimes e_{j'}(W) \in K^k \otimes K^{|W|}$ belongs to the kernel of the matrix $\pi(A)_W$ if and only if, for every vertex $(e_i, e_{i'}(W))$,
$$\sum_{j,j'} \lambda_{(j,j')} d_{(j,j'),(i,i')} = 0.$$
\end{lemma}
%\begin{proof}
%See the proof of \cite[Lemma 20]{Gra16}.
%Write $\pi_E(A)_W = \sum_{i,i',j,j'} d_{(j,j'),(i,i')} e_{ij} \otimes e_{i'j'}(W)$. The element $\alpha = \sum_{j,j'} \lambda_{(j,j')} e_j \otimes e_{j'}(W)$ belongs to $\ker(\pi_E(A)_W)$ if and only if
%$$0 = \pi_E(A)_W (\alpha) = \sum_{i,i'} \Big( \sum_{j,j'} \lambda_{(j,j')} d_{(j,j'),(i,i')} \Big) e_i \otimes e_{i'}(W).$$
%\end{proof}
One can think of this system of equations in a graphical way: if we think of the variable $\lambda_{(i,i')}$ as the label of the vertex $(e_i, e_{i'}(W))$, then the Flow Lemma at this vertex is telling us that
$$\sum_{\substack{\text{all arrows having} \\ \text{range } (i,i')}} (\text{label of the arrow}) \cdot (\text{label of the source of the arrow}) = 0.$$
To see how exactly the Flow Lemma works explicitly, we refer the reader to the appendix given in \cite{Gra16}, where some applications of it in concrete examples are described.

\section{The lamplighter group algebra}\label{section-lamplighter.algebra}

%\textcolor{red}{(nomes posar el cas particular del lamplighter? Llavors cal referencia l'altre article on s'estudien algunes algebres de grup en general!)}

In this section we apply the construction given in Section \ref{section-approx.crossed.product} to the lamplighter group algebra. This algebra is of great relevance because, among other things, it gave the first counterexample to the Strong Atiyah Conjecture (SAC), see for example \cite{GZ}, \cite{DiSc}. We refer the reader to \cite[Section 5]{AC2}, where a further application of the construction from Section \ref{section-approx.crossed.product} to a wider class of group algebras is given.

\begin{definition}
The lamplighter group $\Gamma$ is defined to be the wreath product of the finite group of two elements, $\Z_2$, by $\Z$. In other words,
$$\Gamma = \Z_2 \wr \Z = \Big( \bigoplus_{i \in \Z} \Z_2 \Big) \rtimes_{\sigma} \Z,$$
where the action implementing the semidirect product is the well-known Bernoulli shift $\sigma$ defined by
$$\sigma_n(x)_i = x_{i+n} \quad \text{ for }x = (x_i) \in \bigoplus_{i \in \Z} \Z_2.$$
If we denote by $t$ the generator corresponding to $\Z$ and by $a_i$ the generator corresponding to the $i^{\text{th}}$ copy of $\Z_2$, we have the following presentation for $\Gamma$:
$$\Gamma = \langle t,\{a_i\}_{i \in \Z} \mid a_i^2 , a_ia_ja_ia_j, ta_it^{-1}a_{i-1} \text{ for } i,j \in \Z \rangle.$$
\end{definition}

Let now $K \subseteq \C$ be a subfield of $\C$ closed under complex conjugation, which will be the involution on $K$. %Recall that $\Rk_{M_k(K[\Gamma])}$ denotes the canonical rank function on $M_k(K[\Gamma])$ given by the restriction of the rank function naturally arising from $\calU_k(\Gamma) = M_k(\calU(\Gamma))$ (see Section \ref{subsection-Betti.numbers}).
As in \cite[Section 6]{AC2}, the Fourier transform $\pazF$ gives a $*$-isomorphism $K\Gamma \cong C_K(X) \rtimes_T \Z =: \calA$ through the identifications
$$1 \mapsto \chi_X, \quad e_i := \frac{1 + a_i}{2} \mapsto \chi_{U_i}, \quad t \mapsto t,$$
where $X = \{0,1\}^{\Z}$ is the Cantor set, $T$ the shift map defined by $T(x)_i = x_{i+1}$ for $x = (x_i) \in X$, and $U_i$ is the clopen set consisting of all points $x \in X$ having a $0$ at the $i^{\text{th}}$ component. %For a more detailed exposition and generalizations of this, see \textcolor{red}{referenciar l'altre article si no s'ha fet ja}.

Given $\epsilon_{-k},...,\epsilon_l \in \{0,1\}$, the cylinder set $\{ x=(x_i) \in X \mid x_{-k} = \epsilon_{-k},...,x_l = \epsilon_l \}$ will be throughout denoted by $[\epsilon_{-k} \cdots \underline{\epsilon_0} \cdots \epsilon_l]$. The collection of all the cylinder sets constitutes a basis for the topology of $X$.

We take $\mu$ to be the product measure on $X$, where the $\big( \frac{1}{2}, \frac{1}{2} \big)$-measure on each component $\{0,1\}$ is considered. It is well-known (see e.g. \cite[Example 3.1]{KM}) that $\mu$ is an ergodic, full and shift-invariant probability measure on $X$. Therefore by Theorem \ref{theorem-rank.function} we can construct a Sylvester matrix rank function $\rk_{\calA}$ on $\calA = C_K(X) \rtimes_T \Z$, as exposed in Section \ref{section-approx.crossed.product}.

\begin{proposition}\label{proposition-coincide.lamplighter.group.alg}
As in Subsection \ref{subsection-Betti.numbers}, let $\rk_{K\Gamma}$ be the Sylvester matrix rank function on $K\Gamma$ inherited from $\calU(\Gamma)$. Then $\rk_{K\Gamma}$ and $\rk_{\calA}$ coincide through the Fourier transform $\pazF : K\Gamma \cong \calA$. As a consequence, the $\ell^2$-Betti number of a matrix operator $A \in M_k(K\Gamma)$ equals its $\calA$-Betti number, that is,
\begin{equation}\label{equation-vN.dim.equality}
b^{(2)}(A) = b^{\calA}(\pazF(A)).
\end{equation}
Hence, $\calC(\calA) = \calC(\Gamma,K)$ and $\calG(\calA) = \calG(\Gamma,K)$.
\end{proposition}
\begin{proof}
Both $\rk_{K\Gamma}$ and $\rk_{\calA}$ coincide by Theorem \ref{theorem-rank.function}, as observed in \cite[Subsection 6.1]{AC2}. The result follows from Formula \eqref{equation-vN.dim.rank.UG}.
\end{proof}

\subsection{The approximating algebras \texorpdfstring{$\calA_n$}{} for the lamplighter group algebra}\label{subsection-approximating.algebras.lamplighter}

In this subsection we summarize the contents of \cite[Subsection 6.1]{AC2} in relation to the approximating algebras $\calA_n$ in the particular case of the lamplighter group algebra. We identify $K\Gamma = C_K(X) \rtimes_T \Z$ under the Fourier transform.

We take $E_n = [1 \cdots \underline{1} \cdots 1]$ (with $2n+1$ one's) for the sequence of clopen sets, whose intersection gives the point $y = (...,1,\underline{1},1,...) \in X$. We also take the partitions $\calP_n$ of the complements $X \backslash E_n$ to be
$$\calP_n = \{[0 0 \cdots \underline{0} \cdots 0 0], [0 0 \cdots \underline{0} \cdots 0 1], ..., [0 1 \cdots \underline{1} \cdots 1 1]\}.$$
Let $\calA_n := \calA(E_n, \calP_n)$ be the unital $*$-subalgebra of $\calA = C_K(X) \rtimes_T \Z$ generated by the partial isometries $\{\chi_Z t \mid Z \in \calP_n\}$. %It is easily seen that $\calA_n$ coincides with the unital $*$-subalgebra of $\calA$ generated by the partial isometries $s_i := e_i t$ for $-n \leq i \leq n$, where recall that each $e_i$ is the projection in $K[\Gamma]$ given by $\frac{1 + a_i}{2}$ (equivalently, the characteristic function of the clopen set $[0_i]$), and we put $f_i = 1 - e_i$. %Indeed,
The set $\V_n$ consists of the non-empty $W$ of the following types:
\begin{enumerate}[a),leftmargin=1cm]
\item $W_0 = [1 1 \cdots \underline{1} \cdots 1 1 1]$ of length $1$ (there are $2n+2$ ones);
\item $W_1 = [1 1 \cdots \underline{1} \cdots 1 1 0 1 1 \dots 1 \cdots 1 1]$ of length $2n+2$ (there are $4n+2$ ones, and a zero);
\item $W(\epsilon_1,...,\epsilon_l) = [1 1 \cdots \underline{1} \cdots 1 1 0 \epsilon_1 \cdots \epsilon_l 0 1 1 \cdots 1 \cdots 1 1]$ of length $(2n+3)+l$,
\end{enumerate}
where in the last type $l \geq 0$, and each $\epsilon_i \in \{0,1\}$, but with at most $2n$ consecutive ones in the sequence $(\epsilon_1,...,\epsilon_l)$. We understand $l = 0$ as having no $\epsilon_i$, so there are only $2$ zeros in the middle of the corresponding $W$. %It can be checked by hand that indeed $\ol{\calP}_n$ forms a quasi-partition of $X$, namely that
%$$\sum_{W \in \V_n} |W| \mu(W) = 1.$$

We now recall the definition of the $m$-step Fibonacci sequence, as in \cite[Definition 6.2]{AC2}.
\begin{definition}\label{definition-fibonacci.numbers}
Fix an integer $m \geq 2$. For $k \in \Z^+$, we define the $k^{\text{th}}$ $m$-acci number, denoted by $\text{Fib}_m(k)$, recursively by setting
$$\text{Fib}_m(0) = 0, \quad \text{Fib}_m(1) = 1 \quad \text{ and } \quad \text{Fib}_m(l) = 2^{l-2} \text{ for } 2 \leq l \leq m-1,$$
and for $r \in \Z^+$,
$$\text{Fib}_m(r+m) = \text{Fib}_m(r+m-1) + \cdots + \text{Fib}_m(r).$$
%This sequence is also known in the literature as the $m$-step Fibonacci sequence.
\end{definition}

\begin{lemma}[Lemma 6.3 of \cite{AC2}]\label{lemma-fibonacci.numbers}
For integers $k,m \geq 2$, $\emph{Fib}_m(k)$ is exactly the number of possible sequences $(\epsilon_1,...,\epsilon_l)$ of length $l = k-2$ that one can construct with zeroes and ones, but having at most $m-1$ consecutive one's.
\end{lemma}
%\begin{proof}
%A simple combinatorial argument yields the result.
%\end{proof}

%By using the summation rules
%$$\sum_{k \geq 1} \frac{\text{Fib}_m(k)}{2^k} = 2^{m-1}, \qquad \sum_{k \geq 1} \frac{k \text{Fib}_m(k)}{2^k} = 2^{2m}-(m+1)2^{m-1},$$
%whose proofs are not difficult, we compute
%$$\sum_{W \in \V_n} |W| \mu(W) = \frac{1}{2^{m+1}} + \frac{1}{2^{2m}} \Big( \sum_{k \geq 1} \frac{m \text{Fib}_m(k)}{2^k} + \sum_{k \geq 1} \frac{k\text{Fib}_m(k)}{2^k} \Big) = \frac{m+1}{2^{m+1}} + \Big( 1 - \frac{m+1}{2^{m+1}} \Big) = 1.$$
%\textcolor{red}{(a partir d'aqui crec que no cal! Almenys la part de la clausura *-regular es posara a l'altre pel cas del lamplighter???)}
We thus have injective $*$-representations $\pi_n : \calA_n \hookrightarrow \gotR_n$ defined by the rule $\pi_n(x) = (h_W \cdot x)_W$, where we have the concrete realization $\gotR_n = K \times \prod_{k \geq 1}M_{2n+1+k}(K)^{\text{Fib}_{2n+1}(k)}$. We will use the sequence $\{\calA_n\}_{n \geq 0}$ in the next subsection for computing rational values of $\ell^2$-Betti numbers arising from $\Gamma$.

\subsection{Some computations of \texorpdfstring{$\ell^2$}{}-Betti numbers for the lamplighter group algebra}\label{subsection-computations.Betti.no.lamplighter}

Although it does not appear in our approximating sequence, we now introduce the algebra $\calA_{1/2}$, which gives the first non-trivial approximating algebra exhibiting irrational behavior of $\ell^2$-Betti numbers. Throughout this subsection $K$ will be any subfield of $\C$ closed under complex conjugation, which will be the involution on $K$.

We take as a non-empty clopen subset $E_{1/2} = [1 \underline{1}]$, and the partition $\calP_{1/2} = \{[0\underline{0}],[0\underline{1}],[1\underline{0}]\}$ of the complement $X \backslash E_{1/2}$. The algebra $\calA_{1/2} := \calA(E_{1/2},\calP_{1/2})$ is then the unital $*$-subalgebra of $\calA = C_K(X) \rtimes_T \Z$ generated by the partial isometries $\{\chi_{[0\underline{0}]}t, \chi_{[0\underline{1}]}t, \chi_{[1\underline{0}]}t\}$; equivalently, it is the unital $*$-subalgebra generated by $s_{-1} = e_{-1}t$ and $s_0 = e_0t$. %\textcolor{red}{(comentar que s'estudiara mes endavant amb mes detall?)}

The set $\V_{1/2}$ consists of the non-empty $W$ of the form:

\begin{enumerate}[a),leftmargin=1cm]
\item either $W = [1 \underline{1}1]$ of length $1$, or
\item $W = [1\underline{1}0^{k_1}10^{k_2}1 \cdots 0^{k_r}11]$ of length $k = k_1 + \cdots + k_r + (r+1)$, with $r,k_i \geq 1$,
\end{enumerate}
where $0^{k_i}$ denotes $0 \stackrel{k_i}{\cdots} 0$. It is easily checked that
\begin{align*}
\sum_{W \in \V_{1/2}} |W| \mu(W) = \frac{1}{2^3} + \sum_{r \geq 1} \sum_{k_1,...,k_r \geq 1} \frac{k_1 + \cdots + k_r + (r+1)}{2^{k_1 + \cdots + k_r + (r+3)}} = \frac{1}{2^3} + \sum_{r \geq 1} \frac{3r+1}{2^{r+3}} = 1,
\end{align*}
so indeed the $T$-translates of the $W \in \V_{1/2}$ form a quasi-partition of $X$.

We are now ready to give examples of matrix operators with coefficients in $\Q$, i.e. $A \in M(\Q \Gamma)$, with irrational $\ell^2$-Betti number by using part $(i)$ of Proposition \ref{proposition-coincide.lamplighter.group.alg}, together with the Formula \eqref{equation-Betti.no.final}.
%by taking the approximation algebra $\calA_{1/2} = \calA(E_{1/2},\calP_{1/2})$ constructed from the clopen subset $E_{1/2} = [1\underline{1}]$ and the partition $\calP_{1/2} = \{[0\underline{0}],[0\underline{1}],[1\underline{0}]\}$ of the complement $X \backslash E_{1/2}$. Recall that the quasi-partition $\ol{\calP}_{1/2}$ consists of the translates of the sets $W$ given by
%\begin{enumerate}[a),leftmargin=1cm]
%\item either $W = [1 \underline{1}1]$ of length $1$, or
%\item $W = [1\underline{1}0^{k_1}10^{k_2}1 \cdots 0^{k_r}11]$ of length $k = k_1 + \cdots + k_r + (r+1)$, with $r,k_i \geq 1$.
%\end{enumerate}
%\begin{equation*}
%W = \begin{cases}
%				\text{ either } [1 \underline{1}1] \text{ of length $1$, or }\\
%				[1\underline{1}0^{k_1}10^{k_2}1 \cdots 0^{k_r}11] \text{ of length $k = k_1 + \cdots + k_r + (r+1)$, with } r \geq 1 \text{ and each }k_i \geq 1
%			\end{cases}
%\end{equation*}
Our main result is the following.

\begin{theorem}\label{theorem-irrational.dimensions}
Fix $n$ a non-negative integer. For $0 \leq l \leq n$, take
$$p_l(x) = a_{0,l} + a_{1,l} x + \cdots + a_{m_l,l}x^{m_l}$$
polynomials of degrees $m_l \geq 1$ with non-negative integer coefficients, and we require that $a_{1,0} \geq 1$. Take also $d_1,...,d_n \geq 2$ natural numbers.

There exists then an element $A$ inside some matrix algebra over $\calA_{1/2}$ with rational coefficients such that
\begin{equation}\label{equation-irrational.vNdim}
b^{(2)}(A) = q_0 + q_1 \sum_{k \geq 1} \frac{1}{2^{p_0(k) + p_1(k)d_1^k + \cdots + p_n(k)d_n^k}} \in \calC(\Gamma,\Q),
\end{equation}
where $q_0, q_1$ are non-zero rational numbers which can be explicitly computed. Therefore we get a collection of irrational and even transcendental $\ell^2$-Betti numbers arising from the lamplighter group $\Gamma$. If moreover $a_{0,0} = 0$, then
$$\sum_{k \geq 1} \frac{1}{2^{p_0(k) + p_1(k)d_1^k + \cdots + p_n(k)d_n^k}} \in \calG(\Gamma,\Q).$$
\end{theorem}

\begin{proof}
Since the proof of the theorem is quite technical, we will start by giving some examples of elements whose $\ell^2$-Betti number behave as stated, first by just considering a single polynomial $p(k)$ and then by considering $p(k) d^k$. After that, we will present an explicit element having $\ell^2$-Betti number as in \eqref{equation-irrational.vNdim}.

We start by giving a concrete example, $p(x) = 2 + x + x^2$. Consider the element from $M_{11}(\calA_{1/2})$ given by
\begin{equation*}
\begin{split}
A = & - \chi_{[\underline{0}0]} t^{-1} \cdot e_{1,1} - \chi_{[\underline{0}]} \cdot e_{2,1} - \chi_{[\underline{0}0]} t^{-1} \cdot e_{2,2} \\
& - \chi_{[1\underline{0}]} \cdot e_{3,2} - \chi_{[0 \underline{0}]} t \cdot e_{3,3} - \chi_{[\underline{0}1]} \cdot e_{4,3} \\
& - \chi_{[\underline{0}0]} t^{-1} \cdot e_{4,4} - \chi_{[\underline{0}]} \cdot e_{5,4} - \chi_{[\underline{0}0]} t^{-1} \cdot e_{5,5} \\
& - \chi_{[1\underline{0}]} \cdot e_{6,5} - \chi_{[0\underline{0}]} t \cdot e_{6,6} - \chi_{[\underline{0}1]} \cdot e_{0,6} \\
& - \chi_{[01\underline{0}]} t^2 \cdot e_{8,1} - \chi_{[0\underline{0}]}t \cdot e_{8,8} + \chi_{[\underline{0}]} \cdot e_{9,8} \\
& - \chi_{[0\underline{0}]}t \cdot e_{9,9} - \chi_{[\underline{0}1]} \cdot e_{10,9} \\
& - \chi_{[01\underline{0}]}t \cdot e_{9,7} + \chi_{[\underline{0}10]}t^{-1} \cdot e_{0,7} \\
& + \chi_{[\underline{0}]} \cdot ( \text{Id}_{11} - e_{0,0} - e_{7,7} - e_{10,10} ) - \chi_{[\underline{0}1]} \cdot e_{1,1} \in M_{11}(\Q \Gamma),
\end{split}
\end{equation*}

\noindent being the family $\{e_{i,j} \mid 0 \leq i,j < 11 \}$ a complete system of matrix units for $M_{11}(\Q)$\footnote{It should be mentioned here that this example is similar to the one given in \cite[Section 4]{Gra16}, although some differences are present: apart from the fact that the element is not the same, our construction realizes the `levels' $e_i$ as a result of allowing matrix algebras over $\calA_{1/2}$.}. If we take $W = [1\underline{1}1]$ we observe that $\pi(A)_W$ equals the zero $11 \times 11$ matrix, so its kernel has dimension $11$. Figure \ref{figure-graph.EA.W} gives the prototypical graph $E_A(W)$ that appears in case one takes a $W$ of the second form $[1\underline{1}0^{k_1}10^{k_2}1 \cdots 0^{k_r}11]$ with length $k = k_1 + \cdots + k_r + (r+1)$. By having in mind the Formula \eqref{equation-Betti.no.final}, we study the different types of connected components of the graph $E_A(W)$. From now on, in the diagrams each straight line should be labeled with a $+1$, and each dotted line with a $-1$. Nevertheless, we will explicitly write down some of the labels when necessary in the diagrams, and we will use straight lines in such cases.

\begin{sidewaysfigure}
\centering
\myRule{0.67\textheight}{0.67\textheight}
\[
\xymatrix@R=0.5cm@C=0.3cm{
 & \scriptstyle{e_0(W)} & \scriptstyle{e_1(W)} & \cdots & \scriptstyle{e_{k_1}(W)} & \scriptstyle{e_{k_1+1}(W)} & \scriptstyle{e_{k_1+2}(W)} & \cdots & \scriptstyle{e_{k_1+k_2+1}(W)} & \scriptstyle{e_{k_1 + k_2 + 2}(W)} & \cdots & \cdots & \scriptstyle{e_{k-k_r-1}(W)} & \cdots & \scriptstyle{e_{k-2}(W)} & \scriptstyle{e_{k-1}(W)} \\
e_0 & \bullet & \bullet & \cdots & \bullet & \bullet & \bullet & \cdots & \bullet & \bullet & \cdots & \cdots & \bullet & \cdots & \bullet & \bullet \\
e_1 & \bullet & \bullet \ar@[red]@(ul,ur) \ar@[red]@{.>}[d] & \cdots \ar@[red]@{.>}[l] & \bullet \ar@[red]@{.>}[d] \ar@[red]@{.>}[l] \ar@/^1pc/@[red]@{.>}[dddddddrr] & \bullet & \bullet \ar@[red]@(ul,ur) \ar@[red]@{.>}[d] & \cdots \ar@[red]@{.>}[l] & \bullet \ar@[red]@{.>}[d] \ar@[red]@{.>}[l] \ar@/^1pc/@[red]@{.>}[dddddddrr] & \bullet & \cdots & \cdots & \bullet \ar@[red]@(ul,ur) \ar@[red]@{.>}[d] & \cdots \ar@[red]@{.>}[l] & \bullet \ar@[red]@{.>}[d] \ar@[red]@{.>}[l] & \bullet \\
e_2 & \bullet & \bullet \ar@[red]@(ul,ur) \ar@[red]@{.>}[d] & \cdots \ar@[red]@{.>}[l] & \bullet \ar@[red]@(ul,ur) \ar@[red]@{.>}[l] & \bullet & \bullet \ar@[red]@(ul,ur) \ar@[red]@{.>}[d] & \cdots \ar@[red]@{.>}[l] & \bullet \ar@[red]@(ul,ur) \ar@[red]@{.>}[l] & \bullet & \cdots & \cdots & \bullet \ar@[red]@(ul,ur) \ar@[red]@{.>}[d] & \cdots \ar@[red]@{.>}[l] & \bullet \ar@[red]@(ul,ur) \ar@[red]@{.>}[l] & \bullet \\
e_3 & \bullet & \bullet \ar@[red]@(ul,ur) \ar@[red]@{.>}[r] & \cdots \ar@[red]@{.>}[r] & \bullet \ar@[red]@(ul,ur) \ar@[red]@{.>}[d] & \bullet & \bullet \ar@[red]@(ul,ur) \ar@[red]@{.>}[r] & \cdots \ar@[red]@{.>}[r] & \bullet \ar@[red]@(ul,ur) \ar@[red]@{.>}[d] & \bullet & \cdots & \cdots & \bullet \ar@[red]@(ul,ur) \ar@[red]@{.>}[r] & \cdots \ar@[red]@{.>}[r] & \bullet \ar@[red]@(ul,ur) \ar@[red]@{.>}[d] & \bullet \\
e_4 & \bullet & \bullet \ar@[red]@(ul,ur) \ar@[red]@{.>}[d] & \cdots \ar@[red]@{.>}[l] & \bullet \ar@[red]@(ul,ur) \ar@[red]@{.>}[d] \ar@[red]@{.>}[l] & \bullet & \bullet \ar@[red]@(ul,ur) \ar@[red]@{.>}[d] & \cdots \ar@[red]@{.>}[l] & \bullet \ar@[red]@(ul,ur) \ar@[red]@{.>}[d] \ar@[red]@{.>}[l] & \bullet & \cdots & \cdots & \bullet \ar@[red]@(ul,ur) \ar@[red]@{.>}[d] & \cdots \ar@[red]@{.>}[l] & \bullet \ar@[red]@(ul,ur) \ar@[red]@{.>}[d] \ar@[red]@{.>}[l] & \bullet \\
e_5 & \bullet & \bullet \ar@[red]@(ul,ur) \ar@[red]@{.>}[d] & \cdots \ar@[red]@{.>}[l] & \bullet \ar@[red]@(ul,ur) \ar@[red]@{.>}[l] & \bullet & \bullet \ar@[red]@(ul,ur) \ar@[red]@{.>}[d] & \cdots \ar@[red]@{.>}[l] & \bullet \ar@[red]@(ul,ur) \ar@[red]@{.>}[l] & \bullet & \cdots & \cdots \ar@/^0.5pc/@[red]@{.>}[dddr] & \bullet \ar@[red]@(ul,ur) \ar@[red]@{.>}[d] & \cdots \ar@[red]@{.>}[l] & \bullet \ar@[red]@(ul,ur) \ar@[red]@{.>}[l] & \bullet \\
e_6 & \bullet & \bullet \ar@[red]@(ul,ur) \ar@[red]@{.>}[r] & \cdots \ar@[red]@{.>}[r] & \bullet \ar@[red]@(ul,ur) \ar@/_2pc/@[red]@{.>}[uuuuuu] & \bullet & \bullet \ar@[red]@(ul,ur) \ar@[red]@{.>}[r] & \cdots \ar@[red]@{.>}[r] & \bullet \ar@[red]@(ul,ur) \ar@/_2pc/@[red]@{.>}[uuuuuu] & \bullet & \cdots & \cdots & \bullet \ar@[red]@(ul,ur) \ar@[red]@{.>}[r] & \cdots \ar@[red]@{.>}[r] & \bullet \ar@[red]@(ul,ur) \ar@/_2pc/@[red]@{.>}[uuuuuu] & \bullet \\
e_7 & \bullet & \bullet & \cdots & \bullet & \bullet \ar@[red]@{.>}[ddr] \ar@/_1pc/@[red][uuuuuuul] & \bullet & \cdots & \bullet & \bullet \ar@[red]@{.>}[ddr] \ar@/_1pc/@[red][uuuuuuul] & \cdots & \cdots \ar@[red]@{.>}[ddr] & \bullet & \cdots & \bullet & \bullet \\
e_8 & \bullet & \bullet \ar@[red]@(ul,ur) \ar@[red][d] \ar@[red]@{.>}[r] & \cdots \ar@[red]@{.>}[r] & \bullet \ar@[red]@(ul,ur) \ar@[red][d] & \bullet & \bullet \ar@[red]@(ul,ur) \ar@[red][d] \ar@[red]@{.>}[r] & \cdots \ar@[red]@{.>}[r] & \bullet \ar@[red]@(ul,ur) \ar@[red][d] & \bullet & \cdots & \cdots & \bullet \ar@[red]@(ul,ur) \ar@[red][d] \ar@[red]@{.>}[r] & \cdots \ar@[red]@{.>}[r] & \bullet \ar@[red]@(ul,ur) \ar@[red][d] & \bullet \\
e_9 & \bullet & \bullet \ar@[red]@(ul,ur) \ar@[red]@{.>}[r] & \cdots \ar@[red]@{.>}[r] & \bullet \ar@[red]@(ul,ur) \ar@[red]@{.>}[d] & \bullet & \bullet \ar@[red]@(ul,ur) \ar@[red]@{.>}[r] & \cdots \ar@[red]@{.>}[r] & \bullet \ar@[red]@(ul,ur) \ar@[red]@{.>}[d] & \bullet & \cdots & \cdots & \bullet \ar@[red]@(ul,ur) \ar@[red]@{.>}[r] & \cdots \ar@[red]@{.>}[r] & \bullet \ar@[red]@(ul,ur) \ar@[red]@{.>}[d] & \bullet \\
e_{10} & \bullet & \bullet & \cdots & \bullet & \bullet & \bullet & \cdots & \bullet & \bullet & \cdots & \cdots & \bullet & \cdots & \bullet & \bullet
}
\]
\caption{The graph $E_A(W)$ for a $W = [1\underline{1}0^{k_1}10^{k_2}1 \cdots 0^{k_r}11]$ of length $k$. Each straight line should be labeled with a $+1$, and each dotted line with a $-1$.}
\label{figure-graph.EA.W}
\end{sidewaysfigure}
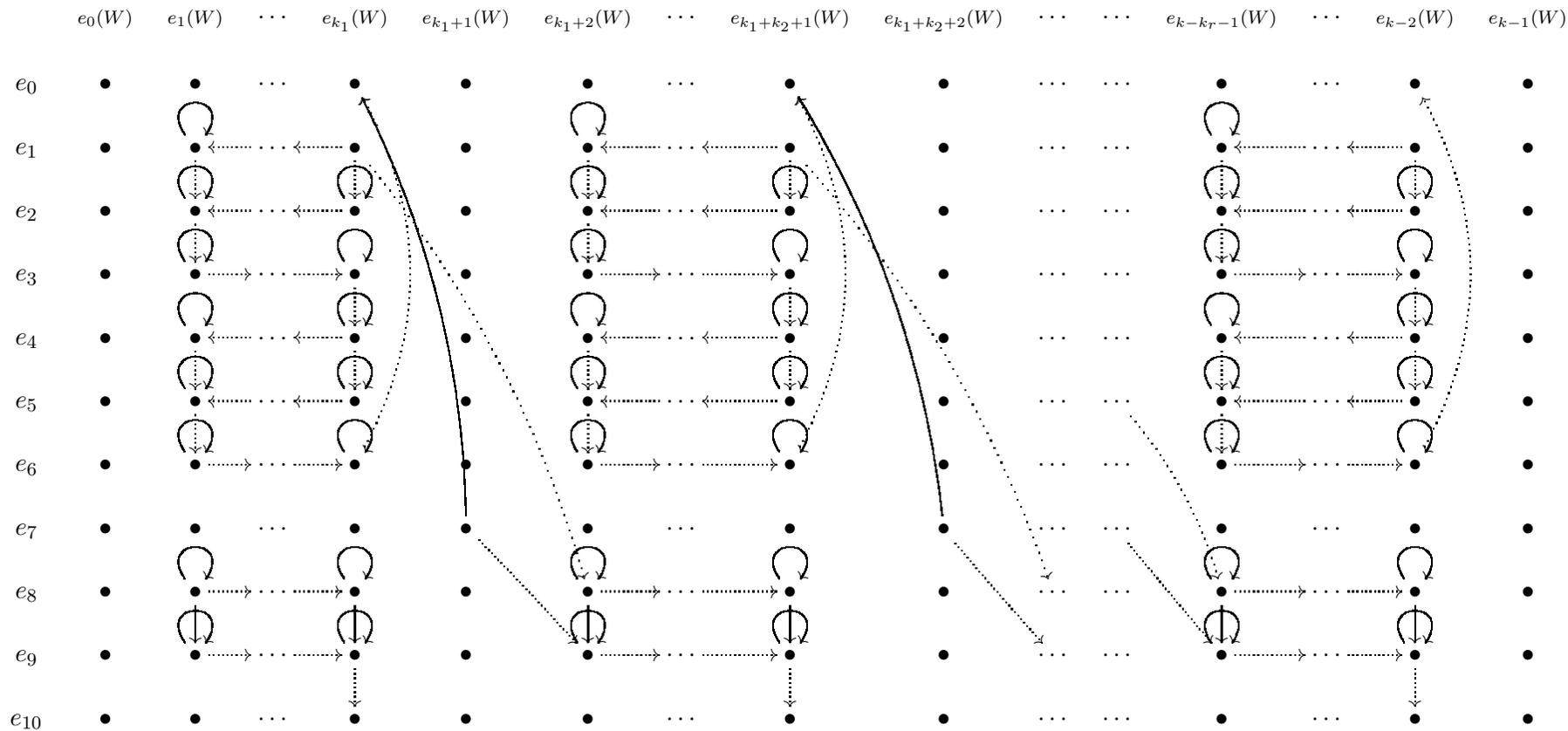

\noindent In this concrete example we have four different types of connected components $C$ for the graph, namely:
\begin{enumerate}[a),leftmargin=0.8cm]
\item $C_1$, given by the graphs with only one vertex $\bullet$. Note that in this case $\text{dim}_{\Q}(\ker \, A_{C_1}) = 1$, and we have $12 + 8r + 3(k_1+\cdots + k_r)$ connected components of this kind.
\item $C_2$, given by the graph
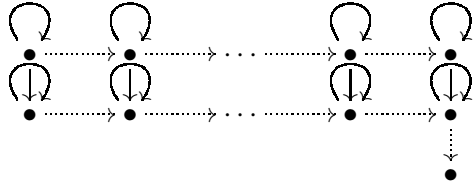
\vspace{0.2cm}
\begin{equation*}
\xymatrix @R=0.4cm @C=0.9cm {
%	T(W) & T^2(W) & \cdots & T^{k_1-1}(W) & T^{k_1}(W) \\
% \bullet & \bullet & \cdots & \bullet & \bullet \\
% \bullet & \bullet & \cdots & \bullet & \bullet \\
% \bullet & \bullet & \cdots & \bullet & \bullet \\
% \bullet & \bullet & \cdots & \bullet & \bullet \\
% \bullet & \bullet & \cdots & \bullet & \bullet \\
% \bullet & \bullet & \cdots & \bullet & \bullet \\
% \bullet & \bullet & \cdots & \bullet & \bullet \\
 \bullet \ar@[red]@(ul,ur) \ar@[red][d] \ar@[red]@{.>}[r] & \bullet \ar@[red]@(ul,ur) \ar@[red][d] \ar@[red]@{.>}[r] & \cdots \ar@[red]@{.>}[r] & \bullet \ar@[red]@(ul,ur) \ar@[red][d] \ar@[red]@{.>}[r] & \bullet \ar@[red]@(ul,ur) \ar@[red][d] \\
 \bullet \ar@[red]@(ul,ur) \ar@[red]@{.>}[r] & \bullet \ar@[red]@(ul,ur) \ar@[red]@{.>}[r] & \cdots \ar@[red]@{.>}[r] & \bullet \ar@[red]@(ul,ur) \ar@[red]@{.>}[r] & \bullet \ar@[red]@(ul,ur) \ar@[red]@{.>}[d] \\
 &  &  &  & \bullet \\
}
\end{equation*}
Let us apply the Flow Lemma to compute $\text{dim}_{\Q}(\ker \, A_{C_2})$. We write $\lambda_{(i,j)}$ to denote the labels of the vertices of the $i^{\text{th}}$ row, for $1 \leq i \leq 3$ and $1 \leq j \leq k_1$. Here $k_1$ corresponds to the number of columns of the graph. We analyze the case where $k_1 \geq 2$, being the case $k_1=1$ completely analogous, and giving the same final result. Note that for $i = 3$ we only have the vertex corresponding to $j=k_1$. The Flow Lemma gives the following set of equations for the vertices:
\begin{equation*}
\begin{cases}
\lambda_{(1,1)} = 0, \\
\lambda_{(2,1)} + \lambda_{(1,1)} = 0, \\
\text{for } 1 \leq j \leq k_1-1, \\
\begin{cases}
\lambda_{(1,j+1)} - \lambda_{(1,j)} = 0, \\         
\lambda_{(2,j+1)} + \lambda_{(1,j+1)} - \lambda_{(2,j)} = 0,
\end{cases} \\
- \lambda_{(2,k_1)} = 0.
\end{cases}
\end{equation*}
The solution of this system of equations is $\lambda_{(1,j)} = \lambda_{(2,j)} = 0$ for all $1 \leq j \leq k_1$, giving one degree of freedom, namely $\lambda_{(3,k_1)}$. Therefore $\text{dim}_{\Q}(\ker \, A_{C_2}) = 1$. We only have one connected component of this kind.
\item $C_3$, given by the graph
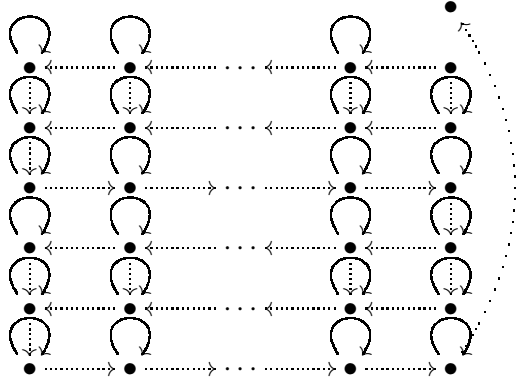
\vspace{-0.1cm}
\begin{equation*}
\xymatrix @R=0.4cm @C=0.9cm {	
	% T^{l+1}(W) & T^{l+2} & \cdots & T^{l + k_i}(W) & T^{l+k_i + 1}(W) & T^{l+k_i+2}(W) & T^{l+k_i+3}(W) & \cdots & T^{l+k_i+k_{i+1}+1}(W) \\
  & & & & \bullet \\
 \bullet \ar@[red]@(ul,ur) \ar@[red]@{.>}[d] & \bullet \ar@[red]@(ul,ur) \ar@[red]@{.>}[d] \ar@[red]@{.>}[l] & \cdots \ar@[red]@{.>}[l] & \bullet \ar@[red]@(ul,ur) \ar@[red]@{.>}[d] \ar@[red]@{.>}[l] & \bullet \ar@[red]@{.>}[d] \ar@[red]@{.>}[l] \\
 \bullet \ar@[red]@(ul,ur) \ar@[red]@{.>}[d] & \bullet \ar@[red]@(ul,ur) \ar@[red]@{.>}[l] & \cdots \ar@[red]@{.>}[l] & \bullet \ar@[red]@(ul,ur) \ar@[red]@{.>}[l] & \bullet \ar@[red]@(ul,ur) \ar@[red]@{.>}[l] \\
 \bullet \ar@[red]@(ul,ur) \ar@[red]@{.>}[r] & \bullet \ar@[red]@(ul,ur) \ar@[red]@{.>}[r] & \cdots \ar@[red]@{.>}[r] & \bullet \ar@[red]@(ul,ur) \ar@[red]@{.>}[r] & \bullet \ar@[red]@(ul,ur) \ar@[red]@{.>}[d] \\
 \bullet \ar@[red]@(ul,ur) \ar@[red]@{.>}[d] & \bullet \ar@[red]@(ul,ur) \ar@[red]@{.>}[d] \ar@[red]@{.>}[l] & \cdots \ar@[red]@{.>}[l] & \bullet \ar@[red]@(ul,ur) \ar@[red]@{.>}[d] \ar@[red]@{.>}[l] & \bullet \ar@[red]@(ul,ur) \ar@[red]@{.>}[d] \ar@[red]@{.>}[l] \\
 \bullet \ar@[red]@(ul,ur) \ar@[red]@{.>}[d] & \bullet \ar@[red]@(ul,ur) \ar@[red]@{.>}[l] & \cdots \ar@[red]@{.>}[l] & \bullet \ar@[red]@(ul,ur) \ar@[red]@{.>}[l] & \bullet \ar@[red]@(ul,ur) \ar@[red]@{.>}[l] \\
 \bullet \ar@[red]@(ul,ur) \ar@[red]@{.>}[r] & \bullet \ar@[red]@(ul,ur) \ar@[red]@{.>}[r] & \cdots \ar@[red]@{.>}[r] & \bullet \ar@[red]@(ul,ur) \ar@[red]@{.>}[r] & \bullet \ar@[red]@(ul,ur) \ar@/_2pc/@[red]@{.>}[uuuuuu] \\
% & & \text{for $k_r \geq 2$} & & \\
} %\qquad \qquad \xymatrix@C=1cm@R=1.5em{
% \bullet \\
% \bullet \ar@[red][d]^{-1} \\
% \bullet \ar@[red]@(ul,ur) \ar@[red][d]^{-1} \\
% \bullet \ar@[red]@(ul,ur) \ar@[red][d]^{-1} \\
% \bullet \ar@[red]@(ul,ur) \ar@[red][d]^{-1} \\
% \bullet \ar@[red]@(ul,ur) \ar@[red][d]^{-1} \\
% \bullet \ar@[red]@(ul,ur) \ar@/_2pc/@[red][uuuuuu]_{-1} \\
% \text{for $k_r = 1$} \\
%}
\end{equation*}
Let us apply again the Flow Lemma to compute $\text{dim}_{\Q}(\ker \, A_{C_3})$. The labels of the vertices of the $i^{\text{th}}$ row will again be denoted by $\lambda_{(i,j)}$, for $1 \leq i \leq 7$ and $1 \leq j \leq k_r$. We restrict ourselves to the case where $k_r \geq 2$ (the case $k_r=1$ gives the same result for $\text{dim}_{\Q}(\ker \, A_{C_3})$). Note that for $i = 1$ we only have the vertex corresponding to $j=k_r$. The Flow Lemma gives the following set of equations for the vertices:
\begin{equation*}
\begin{cases}
\text{for } 1 \leq j \leq k_r-1, \\
\begin{cases}
\lambda_{(2,j)} - \lambda_{(2,j+1)} = 0, \\
\lambda_{(4,j+1)} - \lambda_{(4,j)} = 0, \\
\lambda_{(5,j)} - \lambda_{(5,j+1)} = 0, \\
\lambda_{(7,j+1)} - \lambda_{(7,j)} = 0,
\end{cases} 
\end{cases}
\begin{cases}
\lambda_{(3,k_r)} - \lambda_{(2,k_r)} = 0, \\
\lambda_{(4,1)} - \lambda_{(3,1)} = 0, \\
\lambda_{(5,k_r)} - \lambda_{(4,k_r)} = 0, \\
\lambda_{(6,k_r)} - \lambda_{(5,k_r)} = 0, \\
\lambda_{(7,1)} - \lambda_{(6,1)} = 0, \\
- \lambda_{(7,k_r)} = 0,
\end{cases}
\begin{cases}
\text{for } 1 \leq j \leq k_r-1, \\
\begin{cases}
\lambda_{(3,j)} - \lambda_{(2,j)} - \lambda_{(3,j+1)} = 0, \\
\lambda_{(6,j)} - \lambda_{(5,j)} - \lambda_{(6,j+1)} = 0,
\end{cases} 
\end{cases}
\end{equation*}
The first set of equations gives $\lambda_{(2,j)} = \lambda_{(2,1)}, \lambda_{(4,j)} = \lambda_{(4,1)}, \lambda_{(5,j)} = \lambda_{(5,1)}$, and $\lambda_{(7,j)} = \lambda_{(7,1)}$ for all $1 \leq j \leq k_r$. Analogously, the third set of equations gives $\lambda_{(3,j)} = \lambda_{(3,1)} - (j-1)\lambda_{(2,1)}$, and $\lambda_{(6,j)} = \lambda_{(6,1)} - (j-1)\lambda_{(5,1)}$ for all $2 \leq j \leq k_r$. For now, we are left with $\lambda_{(1,k_r)}$, and $\lambda_{(i,1)}$ for all $2 \leq i \leq 7$ as possible degrees of freedom. But the second set of equations then implies $\lambda_{(2,1)} = \lambda_{(3,1)} - (k_r-1)\lambda_{(2,1)}, \lambda_{(3,1)} = \lambda_{(4,1)} = \lambda_{(5,1)}, \lambda_{(5,1)} = -(k_r-1)\lambda_{(5,1)}$ and $\lambda_{(6,1)} = \lambda_{(7,1)} = 0$. Solving this gives $\lambda_{(i,1)} = 0$ for all $2 \leq i \leq 7$.

The conclusion is that we have only one degree of freedom, namely $\lambda_{(1,k_r)}$, and so $\text{dim}_{\Q}(\ker \, A_{C_3}) = 1$. We only have one connected component of this kind.
\item Finally $C_4$, given by the graphs
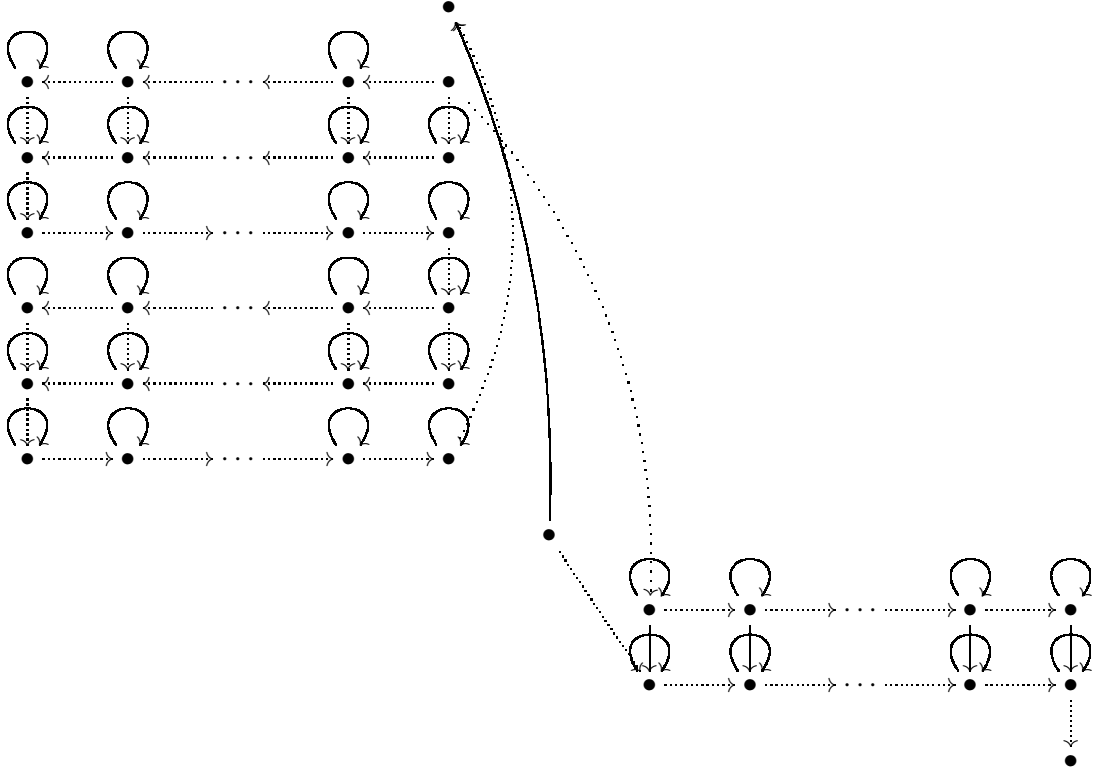
\vspace{-0.3cm}
\begin{equation*}
\xymatrix @R=0.6cm @C=0.9cm {
	% T^{l+1}(W) & T^{l+2} & \cdots & T^{l + k_i}(W) & T^{l+k_i + 1}(W) & T^{l+k_i+2}(W) & T^{l+k_i+3}(W) & \cdots & T^{l+k_i+k_{i+1}+1}(W) \\
  & & & & \bullet & & & & & & \\
 \bullet \ar@[red]@(ul,ur) \ar@[red]@{.>}[d] & \bullet \ar@[red]@(ul,ur) \ar@[red]@{.>}[d] \ar@[red]@{.>}[l] & \cdots \ar@[red]@{.>}[l] & \bullet \ar@[red]@(ul,ur) \ar@[red]@{.>}[d] \ar@[red]@{.>}[l] & \bullet \ar@[red]@{.>}[d] \ar@[red]@{.>}[l] \ar@/^2pc/@[red]@{.>}[dddddddrr] & & & & & & \\
 \bullet \ar@[red]@(ul,ur) \ar@[red]@{.>}[d] & \bullet \ar@[red]@(ul,ur) \ar@[red]@{.>}[l] & \cdots \ar@[red]@{.>}[l] & \bullet \ar@[red]@(ul,ur) \ar@[red]@{.>}[l] & \bullet \ar@[red]@(ul,ur) \ar@[red]@{.>}[l] & & & & & & \\
 \bullet \ar@[red]@(ul,ur) \ar@[red]@{.>}[r] & \bullet \ar@[red]@(ul,ur) \ar@[red]@{.>}[r] & \cdots \ar@[red]@{.>}[r] & \bullet \ar@[red]@(ul,ur) \ar@[red]@{.>}[r] & \bullet \ar@[red]@(ul,ur) \ar@[red]@{.>}[d] & & & & & & \\
 \bullet \ar@[red]@(ul,ur) \ar@[red]@{.>}[d] & \bullet \ar@[red]@(ul,ur) \ar@[red]@{.>}[d] \ar@[red]@{.>}[l] & \cdots \ar@[red]@{.>}[l] & \bullet \ar@[red]@(ul,ur) \ar@[red]@{.>}[d] \ar@[red]@{.>}[l] & \bullet \ar@[red]@(ul,ur) \ar@[red]@{.>}[d] \ar@[red]@{.>}[l] & & & & & & \\
 \bullet \ar@[red]@(ul,ur) \ar@[red]@{.>}[d] & \bullet \ar@[red]@(ul,ur) \ar@[red]@{.>}[l] & \cdots \ar@[red]@{.>}[l] & \bullet \ar@[red]@(ul,ur) \ar@[red]@{.>}[l] & \bullet \ar@[red]@(ul,ur) \ar@[red]@{.>}[l] & & & & & & \\
 \bullet \ar@[red]@(ul,ur) \ar@[red]@{.>}[r] & \bullet \ar@[red]@(ul,ur) \ar@[red]@{.>}[r] & \cdots \ar@[red]@{.>}[r] & \bullet \ar@[red]@(ul,ur) \ar@[red]@{.>}[r] & \bullet \ar@[red]@(ul,ur) \ar@/_2pc/@[red]@{.>}[uuuuuu] & & & & & & \\
 & & & & & \bullet \ar@[red]@{.>}[ddr] \ar@/_1pc/@[red][uuuuuuul] & & & & & \\
 & & & & & & \bullet \ar@[red]@(ul,ur) \ar@[red][d] \ar@[red]@{.>}[r] & \bullet \ar@[red]@(ul,ur) \ar@[red][d] \ar@[red]@{.>}[r] & \cdots \ar@[red]@{.>}[r] & \bullet \ar@[red]@(ul,ur) \ar@[red][d] \ar@[red]@{.>}[r] & \bullet \ar@[red]@(ul,ur) \ar@[red][d] \\
 & & & & & & \bullet \ar@[red]@(ul,ur) \ar@[red]@{.>}[r] & \bullet \ar@[red]@(ul,ur) \ar@[red]@{.>}[r] & \cdots \ar@[red]@{.>}[r] & \bullet \ar@[red]@(ul,ur) \ar@[red]@{.>}[r] & \bullet \ar@[red]@(ul,ur) \ar@[red]@{.>}[d] \\
 & & & & & & & & & & \bullet \\
}
\end{equation*}
These connected components are the essential part of the whole graph since they give rise to the irrationality of $b^{(2)}(A)$. Note that we have $r-1$ connected components of this kind, and we fix one of them, corresponding to the lengths $k_l$ and $k_{l+1}$ respectively, for $1 \leq l \leq r-1$. Note that $k_l$ coincides with the number of columns of the top-left graph, and $k_{l+1}$ coincides with the number of columns of the bottom-right graph.

Once more, we apply the Flow Lemma to compute $\text{dim}_{\Q}(\ker \, A_{C_4})$. We write $\lambda_{(i,j)}$ to denote the labels of the vertices of the top-left graph, where $1 \leq i \leq 7$ are the rows, and $1 \leq j \leq k_l$ are the columns. Note that for $i=1$ we only have the vertex corresponding to $j=k_l$. Similarly, $\mu_{(i,j)}$ will denote the labels of the vertices of the bottom-right graph, where now $1 \leq i \leq 3$ and $1 \leq j \leq k_{l+1}$. As before, note that for $i=3$ we only have the vertex corresponding to $j=k_{l+1}$. We use $\rho$ to denote the label of the isolated vertex being in the middle of the graph. We have the following set of equations:
\begin{equation*}
\begin{cases}
\text{for } 1 \leq j \leq k_l-1, \\
\begin{cases}
\lambda_{(2,j)} - \lambda_{(2,j+1)} = 0, \\
\lambda_{(4,j+1)} - \lambda_{(4,j)} = 0, \\
\lambda_{(5,j)} - \lambda_{(5,j+1)} = 0, \\
\lambda_{(7,j+1)} - \lambda_{(7,j)} = 0,
\end{cases} 
\end{cases}
\begin{cases}
\lambda_{(3,k_l)} - \lambda_{(2,k_l)} = 0, \\
\lambda_{(4,1)} - \lambda_{(3,1)} = 0, \\
\lambda_{(5,k_l)} - \lambda_{(4,k_l)} = 0, \\
\lambda_{(6,k_l)} - \lambda_{(5,k_l)} = 0, \\
\lambda_{(7,1)} - \lambda_{(6,1)} = 0, \\
\rho - \lambda_{(7,k_l)} = 0,
\end{cases}
\begin{cases}
\text{for } 1 \leq j \leq k_l-1, \\
\begin{cases}
\lambda_{(3,j)} - \lambda_{(2,j)} - \lambda_{(3,j+1)} = 0, \\
\lambda_{(6,j)} - \lambda_{(5,j)} - \lambda_{(6,j+1)} = 0,
\end{cases} 
\end{cases}
\end{equation*}
\begin{equation*}
\begin{cases}
\mu_{(1,1)} - \lambda_{(2,k_l)} = 0, \\
\mu_{(2,1)} + \mu_{(1,1)} - \rho = 0, \\
\text{for } 1 \leq j \leq k_{l+1}-1, \\
\begin{cases}
\mu_{(1,j+1)} - \mu_{(1,j)} = 0, \\         
\mu_{(2,j+1)} + \mu_{(1,j+1)} - \mu_{(2,j)} = 0,
\end{cases} \\
- \mu_{(2,k_{l+1})} = 0.
\end{cases}
\end{equation*}
The first set of equations gives $\lambda_{(2,j)} = \lambda_{(2,1)}, \lambda_{(4,j)} = \lambda_{(4,1)}, \lambda_{(5,j)} = \lambda_{(5,1)}$, and $\lambda_{(7,j)} = \lambda_{(7,1)}$ for all $1 \leq j \leq k_l$. Analogously, the third set of equations gives $\lambda_{(3,j)} = \lambda_{(3,1)} - (j-1)\lambda_{(2,1)}$, and $\lambda_{(6,j)} = \lambda_{(6,1)} - (j-1)\lambda_{(5,1)}$ for all $2 \leq j \leq k_l$. Now the fourth set of equations gives $\mu_{(1,j)} = \mu_{(1,1)}$, and $\mu_{(2,j)} = \mu_{(2,1)} - (j-1) \mu_{(1,1)}$ for all $2 \leq j \leq k_{l+1}$, together with $\mu_{(1,1)} = \lambda_{(2,k_l)} = \lambda_{(2,1)}$, $\mu_{(2,1)} + \mu_{(1,1)} = \rho$ and $\mu_{(2,k_{l+1})} = 0$.

For now, the possible degrees of freedom reduces to $\lambda_{(1,k_l)}$, $\lambda_{(i,1)}$ for all $2 \leq i \leq 7$, $\rho$, and $\mu_{(3,k_{l+1})}$. But imposing the restrictions coming from the second set of equations gives $\rho = \lambda_{(7,1)} = \lambda_{(6,1)} = k_l \lambda_{(5,1)} = k_l \lambda_{(4,1)} = k_l \lambda_{(3,1)} = k_l^2 \lambda_{(2,1)}$, with the condition $(k_l^2 - k_{l+1})\lambda_{(2,1)} = 0$. This implies that the possible degrees of freedom are now $\lambda_{(1,k_l)}$, $\lambda_{(2,1)}$, and $\mu_{(3,k_{l+1})}$, together with the restriction $(k_l^2-k_{l+1}) \lambda_{(2,1)} = 0$.

The conclusion is that, depending on whether $k_l^2 = k_{l+1}$ or not, we have either three degrees of freedom or two, respectively. Thus we get
\[ \text{dim}_{\Q}(\ker \, A_{C_4}) = \left\{ \begin{array}{lr}
																				3 & \text{ if $k_{l+1} = k_l^2$} \\
																				2 & \text{ otherwise}
																			 \end{array}\right\} = 2 + \delta_{k_{l+1},k_l^2}. \]
As commented, we have $r-1$ connected components of this kind.
\end{enumerate}
%\begin{proposition}
%We have the following dimensions for the kernels:
%\begin{enumerate}[1)]
%\item $\text{dim}(\ker \, A_{C_1}) = 1$.
%\item $\text{dim}(\text{Ker}(A_{C_2})) = 1$.
%\item If we denote by $k_i$ and $k_{i+1}$ the respective lengths of the graph given by $A_{C_3}$, then
%\begin{equation*}
%\text{dim}(\text{Ker}(A_{C_3})) = \begin{cases}
%																		2 & \text{ if } k_{i+1} = k_i^2 - 1 \\
%																		1 & \text{ otherwise}
%																	\end{cases} = 1 + \delta_{k_{i+1},k_i^2-1}
%\end{equation*}
%\item If we denote by $k$ the length of the graph given by $A_{C_4}$, then
%\begin{equation*}
%\text{dim}(\text{Ker}(A_{C_4})) = \begin{cases}
%																		1 & \text{ if } k = 1 \\
%																		0 & \text{ otherwise }
%																	\end{cases} = \delta_{k,1}
%\end{equation*}
%\end{enumerate}
%\end{proposition}
%\begin{proof}
%Just apply the flow Lemma \ref{lemma_flow.lemma}. 
%\end{proof}
We can now compute the $l^2$-Betti number of $A$ using Proposition \ref{proposition-vNdim.computation} and Formula \eqref{equation-Betti.no.final}:
\vspace{-0.4cm}
\begin{align*}
b^{(2)}(A) & = \sum_{W \in \V_{1/2}} \mu(W) \text{dim}_{\Q}(\ker \, \pi(A)_W) \\
& = \mu([1\ul{1}1]) \text{dim}_{\Q}(\ker \, \pi(A)_{[1\ul{1}1]}) + \sum_{r \geq 1} \sum_{k_1,...,k_r \geq 1} \mu([1\ul{1}0^{k_1}10^{k_2} \cdots 10^{k_r}11]) \Big( \sum_{i=1}^4 \dim_{\Q}(\ker \, A_{C_i}) \Big) \\
& = \frac{11}{2^3} + \sum_{r \geq 1} \sum_{k_1,...,k_r \geq 1} \frac{1}{2^{k_1+\cdots+k_r+(r+3)}} \Big[(12+8r+3(k_1+\cdots+k_r))+1+1+\sum_{l=1}^{r-1}(2 + \delta_{k_{l+1},k_l^2}) \Big] \\
& = \frac{11}{8} + \frac{1}{8} \sum_{r \geq 1} \frac{1}{2^r} \Big[ \sum_{k_1,...,k_r \geq 1} \frac{12+10r}{2^{k_1} \cdots 2^{k_r}} + 3 \sum_{k_1,...,k_r \geq 1} \frac{k_1 + \cdots + k_r}{2^{k_1} \cdots 2^{k_r}} + \sum_{k_1,...,k_r \geq 1} \sum_{l=1}^{r-1} \frac{\delta_{k_{l+1},k_l^2}}{2^{k_1} \cdots 2^{k_r}} \Big] \\
& = \frac{11}{8} + \frac{1}{8} \sum_{r \geq 1} \frac{1}{2^r} \Big[ (12+10r) \Big(\sum_{k \geq 1} \frac{1}{2^k} \Big)^r + 3r \Big( \sum_{k \geq 1} \frac{k}{2^k} \Big)\Big(\sum_{k \geq 1} \frac{1}{2^k}\Big)^{r-1} \\
& \hspace{6.52cm} + (r-1)\Big( \sum_{k \geq 1} \frac{1}{2^{k^2+k}} \Big) \Big( \sum_{k \geq 1} \frac{1}{2^k} \Big)^{r-2} \Big] \\
& = \frac{11}{8} + \frac{1}{8} \sum_{r \geq 1}\frac{12+16r}{2^r} + \frac{1}{8}\Big( \sum_{r \geq 1} \frac{r-1}{2^r} \Big) \Big( \sum_{k \geq 1}\frac{1}{2^{k^2+k}}\Big) \\
& = \frac{55}{8} + \frac{1}{2} \sum_{k \geq 1} \frac{1}{2^{k^2+k+2}} \approx 6.9082337619...
\vspace{-0.4cm}
\end{align*}
which is an irrational number (its binary expansion is clearly non-periodic).

With this example in mind, we now construct an element whose $\ell^2$-Betti number equals
$$\vspace{-0.1cm}\frac{12n+31}{8} + \frac{2^{a_0}}{8} \sum_{k \geq 1}\frac{1}{2^{p(k)}},\vspace{-0.1cm}$$
being $p(x) = a_0 + a_1x + \cdots + a_nx^n$ a polynomial of degree $n \geq 2$ with non-negative integer coefficients, whose linear coefficient satisfies $a_1 \geq 1$. 
%\begin{enumerate}[$\cdot$]
%\item $a_i \in \mathbb{Z}$ for every $i$.
%\item $a_1k + \cdots + a_n k^n \neq 1$ for any $k \geq 1$.
%\item $(a_1 - 1)k + \cdots + a_nk^n \geq 1$ for any $k \geq 1$.
%\end{enumerate}
We consider the element from $M_{3n+5}(\calA_{1/2})$ given by

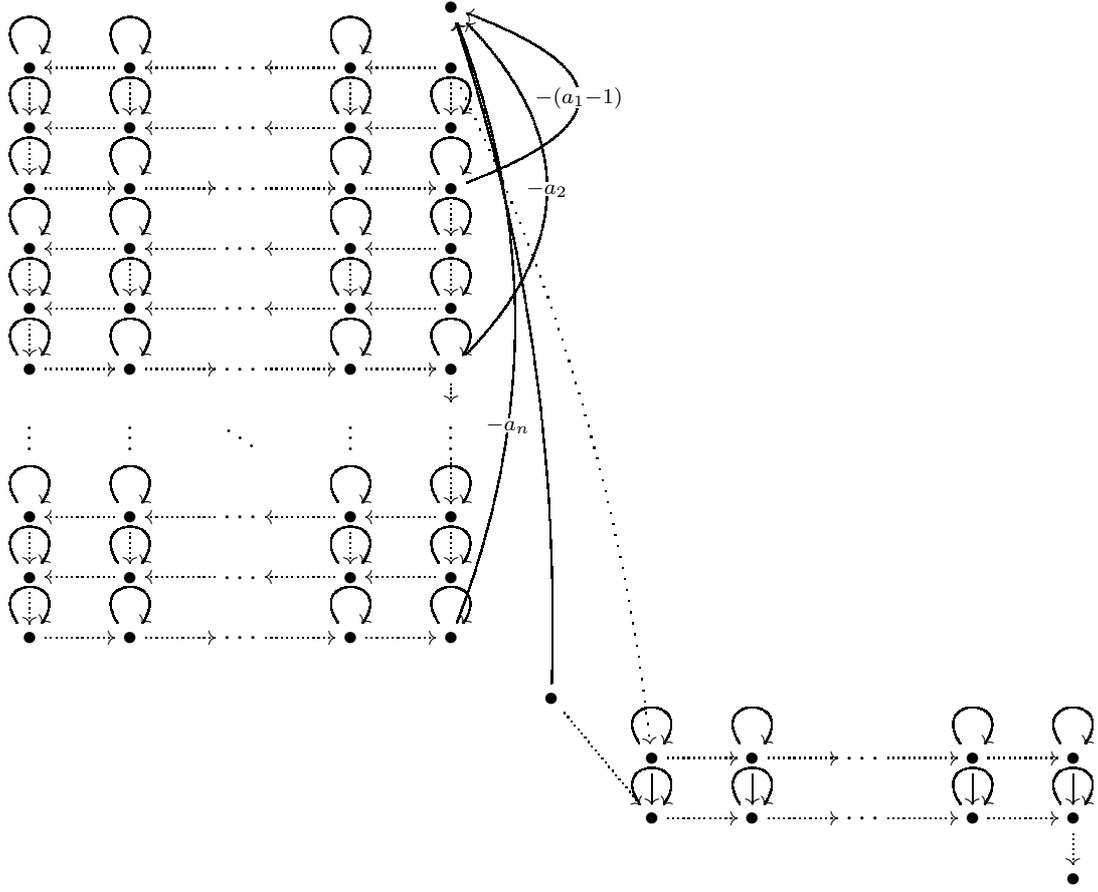
\begin{equation*}
\begin{split}
A = & - \chi_{[\underline{0}0]} t^{-1} \cdot e_{1,1} - \chi_{[\underline{0}]} \cdot e_{2,1} - \chi_{[\underline{0}0]} t^{-1} \cdot e_{2,2} \\
& - \chi_{[1\underline{0}]} \cdot e_{3,2} - \chi_{[0 \underline{0}]} t \cdot e_{3,3} - \chi_{[\underline{0}1]} \cdot e_{4,3} - (a_1 - 1) \chi_{[\underline{0}1]} \cdot e_{0,3} \\
& + \sum_{i=1}^{n-2} \Big( - \chi_{[\underline{0}0]} t^{-1} \cdot e_{3i+1,3i+1} - \chi_{[\underline{0}]} \cdot e_{3i+2,3i+1} - \chi_{[\underline{0} 0]} t^{-1} \cdot e_{3i+2,3i+2} \\
& \quad - \chi_{[1\underline{0}]} \cdot e_{3i+3,3i+2} - \chi_{[0 \underline{0}]} t \cdot e_{3i+3,3i+3} - \chi_{[\underline{0}1]} \cdot e_{3i+4,3i+3} - a_{i+1} \chi_{[\underline{0}1]} \cdot e_{0,3i+3} \Big) \\
& - \chi_{[\underline{0} 0]} t^{-1} \cdot e_{3n-2,3n-2} - \chi_{[\underline{0}]} \cdot e_{3n-1,3n-2} - \chi_{[\underline{0} 0]} t^{-1} \cdot e_{3n-1,3n-1} \\
& - \chi_{[1\underline{0}]} \cdot e_{3n,3n-1} - \chi_{[0 \underline{0}]} t \cdot e_{3n,3n} - a_n \chi_{[\underline{0}1]} \cdot e_{0,3n} \\
& - \chi_{[01\underline{0}]} t^2 \cdot e_{3n+2,1} - \chi_{[0 \underline{0}]}t \cdot e_{3n+2,3n+2} + \chi_{[\underline{0}]} \cdot e_{3n+3,3n+2} \\
& - \chi_{[0 \underline{0}]}t \cdot e_{3n+3,3n+3} - \chi_{[\underline{0}1]} \cdot e_{3n+4,3n+3} - \chi_{[01\underline{0}]}t \cdot e_{3n+3,3n+1} + \chi_{[\underline{0}10]}t^{-1} \cdot e_{0,3n+1} \\
& + \chi_{[\underline{0}]} \cdot \left( \text{Id}_{3n+5} - e_{0,0} - e_{3n+1,3n+1} - e_{3n+4,3n+4} \right) - \chi_{[\underline{0}1]} \cdot e_{1,1} \in M_{3n+5}(\Q \Gamma).
\end{split}
\end{equation*}
\vspace{0.2cm}

Again, $\pi(A)_W$ gives the zero $(3n+5) \times (3n+5)$ matrix for $W = [1 \underline{1}1]$, so its kernel has dimension $3n+5$. For a $W = [1\underline{1}0^{k_1}10^{k_2}1 \cdots 0^{k_r}11]$ of length $k = k_1 + \cdots + k_r + (r+1)$, its graph $E_A(W)$ has four different types of connected components $C$, namely:
\begin{enumerate}[a),leftmargin=0.8cm]
\item $C_1$, given by the graphs with only one vertex $\bullet$. Here $\text{dim}_{\Q}(\ker \, A_{C_1}) = 1$, and we have $(3n+6)+(3n+2)r + 3(k_1+\cdots + k_r)$ connected components of this kind.
\item $C_2$, given by the graph
\vspace{0.2cm}
\begin{equation*}
\xymatrix @R=0.4cm @C=0.9cm {
 \bullet \ar@[red]@(ul,ur) \ar@[red][d] \ar@[red]@{.>}[r] & \bullet \ar@[red]@(ul,ur) \ar@[red][d] \ar@[red]@{.>}[r] & \cdots \ar@[red]@{.>}[r] & \bullet \ar@[red]@(ul,ur) \ar@[red][d] \ar@[red]@{.>}[r] & \bullet \ar@[red]@(ul,ur) \ar@[red][d] \\
 \bullet \ar@[red]@(ul,ur) \ar@[red]@{.>}[r] & \bullet \ar@[red]@(ul,ur) \ar@[red]@{.>}[r] & \cdots \ar@[red]@{.>}[r] & \bullet \ar@[red]@(ul,ur) \ar@[red]@{.>}[r] & \bullet \ar@[red]@(ul,ur) \ar@[red]@{.>}[d] \\
 &  &  &  & \bullet \\
% & & \text{for $k_1 \geq 2$} & & \\
} %\qquad \qquad \xymatrix@=1em{
% \bullet \ar@[red]@(ul,ur) \ar@[red][d]^{1} \\
% \bullet \ar@[red]@(ul,ur) \ar@[red][d]^{-1} \\
% \bullet \\
% \text{for $k_1 = 1$} \\
%}
\end{equation*}
The Flow Lemma gives $\text{dim}_{\Q}(\ker \, A_{C_2}) = 1$. We only have one connected component of this kind.
\item $C_3$, given by the graph
\vspace{0.0cm}
\begin{equation*}
\xymatrix @R=0.45cm @C=0.9cm {
	% T^{l+1}(W) & T^{l+2} & \cdots & T^{l + k_i}(W) & T^{l+k_i + 1}(W) & T^{l+k_i+2}(W) & T^{l+k_i+3}(W) & \cdots & T^{l+k_i+k_{i+1}+1}(W) \\
  & & & & \bullet \\
 \bullet \ar@[red]@(ul,ur) \ar@[red]@{.>}[d] & \bullet \ar@[red]@(ul,ur) \ar@[red]@{.>}[d] \ar@[red]@{.>}[l] & \cdots \ar@[red]@{.>}[l] & \bullet \ar@[red]@(ul,ur) \ar@[red]@{.>}[d] \ar@[red]@{.>}[l] & \bullet \ar@[red]@{.>}[d] \ar@[red]@{.>}[l] \\
 \bullet \ar@[red]@(ul,ur) \ar@[red]@{.>}[d] & \bullet \ar@[red]@(ul,ur) \ar@[red]@{.>}[l] & \cdots \ar@[red]@{.>}[l] & \bullet \ar@[red]@(ul,ur) \ar@[red]@{.>}[l] & \bullet \ar@[red]@(ul,ur) \ar@[red]@{.>}[l] \\
 \bullet \ar@[red]@(ul,ur) \ar@[red]@{.>}[r] & \bullet \ar@[red]@(ul,ur) \ar@[red]@{.>}[r] & \cdots \ar@[red]@{.>}[r] & \bullet \ar@[red]@(ul,ur) \ar@[red]@{.>}[r] & \bullet \ar@[red]@(ul,ur) \ar@/_4pc/@[red][uuu]|-{-(a_1-1)} \ar@[red]@{.>}[d] \\
 \bullet \ar@[red]@(ul,ur) \ar@[red]@{.>}[d] & \bullet \ar@[red]@(ul,ur) \ar@[red]@{.>}[d] \ar@[red]@{.>}[l] & \cdots \ar@[red]@{.>}[l] & \bullet \ar@[red]@(ul,ur) \ar@[red]@{.>}[d] \ar@[red]@{.>}[l] & \bullet \ar@[red]@(ul,ur) \ar@[red]@{.>}[d] \ar@[red]@{.>}[l] \\
 \bullet \ar@[red]@(ul,ur) \ar@[red]@{.>}[d] & \bullet \ar@[red]@(ul,ur) \ar@[red]@{.>}[l] & \cdots \ar@[red]@{.>}[l] & \bullet \ar@[red]@(ul,ur) \ar@[red]@{.>}[l] & \bullet \ar@[red]@(ul,ur) \ar@[red]@{.>}[l] \\
 \bullet \ar@[red]@(ul,ur) \ar@[red]@{.>}[r] & \bullet \ar@[red]@(ul,ur) \ar@[red]@{.>}[r] & \cdots \ar@[red]@{.>}[r] & \bullet \ar@[red]@(ul,ur) \ar@[red]@{.>}[r] & \bullet \ar@[red]@(ul,ur) \ar@/_3pc/@[red][uuuuuu]|-(0.45){-a_2} \ar@[red]@{.>}[d] \\
\vdots & \vdots & \ddots & \vdots & \vdots \ar@[red]@{.>}[d] \\
 \bullet \ar@[red]@(ul,ur) \ar@[red]@{.>}[d] & \bullet \ar@[red]@(ul,ur) \ar@[red]@{.>}[d] \ar@[red]@{.>}[l] & \cdots \ar@[red]@{.>}[l] & \bullet \ar@[red]@(ul,ur) \ar@[red]@{.>}[d] \ar@[red]@{.>}[l] & \bullet \ar@[red]@(ul,ur) \ar@[red]@{.>}[d] \ar@[red]@{.>}[l] \\
 \bullet \ar@[red]@(ul,ur) \ar@[red]@{.>}[d] & \bullet \ar@[red]@(ul,ur) \ar@[red]@{.>}[l] & \cdots \ar@[red]@{.>}[l] & \bullet \ar@[red]@(ul,ur) \ar@[red]@{.>}[l] & \bullet \ar@[red]@(ul,ur) \ar@[red]@{.>}[l] \\
 \bullet \ar@[red]@(ul,ur) \ar@[red]@{.>}[r] & \bullet \ar@[red]@(ul,ur) \ar@[red]@{.>}[r] & \cdots \ar@[red]@{.>}[r] & \bullet \ar@[red]@(ul,ur) \ar@[red]@{.>}[r] & \bullet \ar@[red]@(ul,ur) \ar@/_2pc/@[red][uuuuuuuuuu]|-(0.4){-a_n} \\
% & & \text{for $k_r \geq 2$} & & \\
} %\qquad \qquad \xymatrix@=0.7cm{
% \bullet \\
% \bullet \ar@[red][d]^{-1} \\
% \bullet \ar@[red]@(ul,ur) \ar@[red][d]^{-1} \\
% \bullet \ar@[red]@(ul,ur) \ar@/_2pc/@[red][uuu]_{-(a_1-1)} \ar@[red][d]^{-1} \\
% \bullet \ar@[red]@(ul,ur) \ar@[red][d]^{-1} \\
% \bullet \ar@[red]@(ul,ur) \ar@[red][d]^{-1} \\
% \bullet \ar@[red]@(ul,ur) \ar@/_2pc/@[red][uuuuuu]_{-a_2} \ar@[red][d]^{-1} \\
% \vdots \ar@[red][d]^{-1} \\
% \bullet \ar@[red]@(ul,ur) \ar@[red][d]^{-1} \\
% \bullet \ar@[red]@(ul,ur) \ar@[red][d]^{-1} \\
% \bullet \ar@[red]@(ul,ur) \ar@/_2pc/@[red][uuuuuuuuuu]_{-a_n} \\
% \text{for $k_r = 1$} \\
%}
\end{equation*}
We get $\text{dim}_{\Q}(\ker \, A_{C_3}) = 1$. We only have one connected component of this type.

\item Finally $C_4$, given by the graphs
\vspace{-0.4cm}
\begin{equation*}
\xymatrix @R=0.4cm @C=0.9cm {
	% T^{l+1}(W) & T^{l+2} & \cdots & T^{l + k_i}(W) & T^{l+k_i + 1}(W) & T^{l+k_i+2}(W) & T^{l+k_i+3}(W) & \cdots & T^{l+k_i+k_{i+1}+1}(W) \\
  & & & & \bullet & & & & & & \\
 \bullet \ar@[red]@(ul,ur) \ar@[red]@{.>}[d] & \bullet \ar@[red]@(ul,ur) \ar@[red]@{.>}[d] \ar@[red]@{.>}[l] & \cdots \ar@[red]@{.>}[l] & \bullet \ar@[red]@(ul,ur) \ar@[red]@{.>}[d] \ar@[red]@{.>}[l] & \bullet \ar@[red]@{.>}[d] \ar@[red]@{.>}[l] \ar@/^1pc/@[red]@{.>}[dddddddddddrr] & & & & & & \\
 \bullet \ar@[red]@(ul,ur) \ar@[red]@{.>}[d] & \bullet \ar@[red]@(ul,ur) \ar@[red]@{.>}[l] & \cdots \ar@[red]@{.>}[l] & \bullet \ar@[red]@(ul,ur) \ar@[red]@{.>}[l] & \bullet \ar@[red]@(ul,ur) \ar@[red]@{.>}[l] & & & & & & \\
 \bullet \ar@[red]@(ul,ur) \ar@[red]@{.>}[r] & \bullet \ar@[red]@(ul,ur) \ar@[red]@{.>}[r] & \cdots \ar@[red]@{.>}[r] & \bullet \ar@[red]@(ul,ur) \ar@[red]@{.>}[r] & \bullet \ar@[red]@(ul,ur) \ar@/_4pc/@[red][uuu]|-{-(a_1-1)} \ar@[red]@{.>}[d] & & & & & & \\
 \bullet \ar@[red]@(ul,ur) \ar@[red]@{.>}[d] & \bullet \ar@[red]@(ul,ur) \ar@[red]@{.>}[d] \ar@[red]@{.>}[l] & \cdots \ar@[red]@{.>}[l] & \bullet \ar@[red]@(ul,ur) \ar@[red]@{.>}[d] \ar@[red]@{.>}[l] & \bullet \ar@[red]@(ul,ur) \ar@[red]@{.>}[d] \ar@[red]@{.>}[l] & & & & & & \\
 \bullet \ar@[red]@(ul,ur) \ar@[red]@{.>}[d] & \bullet \ar@[red]@(ul,ur) \ar@[red]@{.>}[l] & \cdots \ar@[red]@{.>}[l] & \bullet \ar@[red]@(ul,ur) \ar@[red]@{.>}[l] & \bullet \ar@[red]@(ul,ur) \ar@[red]@{.>}[l] & & & & & & \\
 \bullet \ar@[red]@(ul,ur) \ar@[red]@{.>}[r] & \bullet \ar@[red]@(ul,ur) \ar@[red]@{.>}[r] & \cdots \ar@[red]@{.>}[r] & \bullet \ar@[red]@(ul,ur) \ar@[red]@{.>}[r] & \bullet \ar@[red]@(ul,ur) \ar@/_3pc/@[red][uuuuuu]|-{-a_2} \ar@[red]@{.>}[d] & & & & & & \\
\vdots & \vdots & \ddots & \vdots & \vdots \ar@[red]@{.>}[d] & & & & & & \\
 \bullet \ar@[red]@(ul,ur) \ar@[red]@{.>}[d] & \bullet \ar@[red]@(ul,ur) \ar@[red]@{.>}[d] \ar@[red]@{.>}[l] & \cdots \ar@[red]@{.>}[l] & \bullet \ar@[red]@(ul,ur) \ar@[red]@{.>}[d] \ar@[red]@{.>}[l] & \bullet \ar@[red]@(ul,ur) \ar@[red]@{.>}[d] \ar@[red]@{.>}[l] & & & & & & \\
 \bullet \ar@[red]@(ul,ur) \ar@[red]@{.>}[d] & \bullet \ar@[red]@(ul,ur) \ar@[red]@{.>}[l] & \cdots \ar@[red]@{.>}[l] & \bullet \ar@[red]@(ul,ur) \ar@[red]@{.>}[l] & \bullet \ar@[red]@(ul,ur) \ar@[red]@{.>}[l] & & & & & & \\
 \bullet \ar@[red]@(ul,ur) \ar@[red]@{.>}[r] & \bullet \ar@[red]@(ul,ur) \ar@[red]@{.>}[r] & \cdots \ar@[red]@{.>}[r] & \bullet \ar@[red]@(ul,ur) \ar@[red]@{.>}[r] & \bullet \ar@[red]@(ul,ur) \ar@/_2pc/@[red][uuuuuuuuuu]|-(0.33){-a_n} & & & & & & \\
 & & & & & \bullet \ar@[red]@{.>}[ddr] \ar@/_1pc/@[red][uuuuuuuuuuul] & & & & & \\
 & & & & & & \bullet \ar@[red]@(ul,ur) \ar@[red][d] \ar@[red]@{.>}[r] & \bullet \ar@[red]@(ul,ur) \ar@[red][d] \ar@[red]@{.>}[r] & \cdots \ar@[red]@{.>}[r] & \bullet \ar@[red]@(ul,ur) \ar@[red][d] \ar@[red]@{.>}[r] & \bullet \ar@[red]@(ul,ur) \ar@[red][d] \\
 & & & & & & \bullet \ar@[red]@(ul,ur) \ar@[red]@{.>}[r] & \bullet \ar@[red]@(ul,ur) \ar@[red]@{.>}[r] & \cdots \ar@[red]@{.>}[r] & \bullet \ar@[red]@(ul,ur) \ar@[red]@{.>}[r] & \bullet \ar@[red]@(ul,ur) \ar@[red]@{.>}[d] \\
 & & & & & & & & & & \bullet \\
}
\end{equation*}
Here
\[ \text{dim}_{\Q}(\ker \, A_{C_4}) = \left\{ \begin{array}{lr}
																				3 & \text{ if $k_{i+1} = p(k_i) - k_i - a_0$} \\
																				2 & \text{ otherwise}
																			 \end{array}\right\} = 2 + \delta_{k_{i+1},p(k_i) - k_i - a_0}. \]
We have $r-1$ connected components of this kind.
\end{enumerate}
Putting everything together, we compute
\begin{align*}
b^{(2)}(A) & = \mu([1\ul{1}1]) \text{dim}_{\Q}(\ker \, \pi(A)_{[1\ul{1}1]}) + \sum_{r \geq 1} \sum_{k_1,...,k_r \geq 1} \mu([1\ul{1}0^{k_1}10^{k_2} \cdots 10^{k_r}11]) \Big( \sum_{i=1}^4 \dim_{\Q}(\ker \, A_{C_i}) \Big) \\
& = \frac{3n+5}{2^3} + \sum_{r \geq 1} \sum_{k_1,...,k_r \geq 1} \frac{1}{2^{k_1+ \cdots +k_r + (r+3)}} \Big[ (3n+6) + (3n+4)r + 3(k_1 + \cdots + k_r) \Big] \\
& \hspace{1.68cm} + \sum_{r \geq 1} \sum_{k_1,...,k_r \geq 1} \frac{1}{2^{k_1+ \cdots +k_r + (r+3)}} \Big[ \sum_{i=1}^{r-1} \delta_{k_{i+1},p(k_i) - k_i - a_0} \Big] \\
& = \frac{3n+5}{8} + \frac{1}{8} \sum_{r \geq 1} \frac{(3n+6)+(3n+10)r}{2^r} + \frac{2^{a_0}}{8} \Big(\sum_{r \geq 1} \frac{r-1}{2^r}\Big) \Big( \sum_{k \geq 1} \frac{1}{2^{p(k)}} \Big) \\
& = \frac{12n+31}{8} + \frac{2^{a_0}}{8} \sum_{k \geq 1} \frac{1}{2^{p(k)}},
\end{align*}
which is an irrational number since the degree of the polynomial is at least $2$. Note that, in the special case $a_0 = 0$, we obtain a number of the form $q_0 + \frac{1}{2^m} \alpha$, being $\alpha = \sum_{k \geq 1} \frac{1}{2^{p(k)}}$.

Let us do a step further: we now construct an element with $\ell^2$-Betti number exactly
$$\frac{12n+35}{8} + \frac{1}{8} \sum_{k \geq 1} \frac{1}{2^{k+p(k)d^k}},$$
with $p(x)$ a polynomial of degree at least $1$ with non-negative integer coefficients, and $d \geq 2$ a natural number. We consider the element from $M_{3n+6}(\calA_{1/2})$ given by

\begin{equation*}
\begin{split}
A = & - d \chi_{[1\underline{0}]} \cdot e_{1,2} - \chi_{[0 \underline{0}]} t \cdot e_{1,1} - a_0 \chi_{[\underline{0}1]} \cdot e_{0,1} \\
& - d \chi_{[\underline{0} 0]} t^{-1} \cdot e_{2,2} - \chi_{[\underline{0}]} \cdot e_{3,2} - d \chi_{[\underline{0} 0]} t^{-1} \cdot e_{3,3} \\
& - d \chi_{[1\underline{0}]} \cdot e_{4,3} - \chi_{[0 \underline{0}]} t \cdot e_{4,4} - \chi_{[\underline{0}1]} \cdot e_{5,4} - a_1 \chi_{[\underline{0}1]} \cdot e_{0,4} \\
& + \sum_{i=1}^{n-2} \Big( - \chi_{[\underline{0} 0]} t^{-1} \cdot e_{3i+2,3i+2} - \chi_{[\underline{0}]} \cdot e_{3i+3,3i+2} - \chi_{[\underline{0} 0]} t^{-1} \cdot e_{3i+3,3i+3} \\
& \quad - \chi_{[1\underline{0}]} \cdot e_{3i+4,3i+3} - \chi_{[0 \underline{0}]} t \cdot e_{3i+4,3i+4} - \chi_{[\underline{0}1]} \cdot e_{3i+5,3i+4} - a_{i+1} \chi_{[\underline{0}1]} \cdot e_{0,3i+4} \Big) \\
& - \chi_{[\underline{0} 0]} t^{-1} \cdot e_{3n-1,3n-1} - \chi_{[\underline{0}]} \cdot e_{3n,3n-1} - \chi_{[\underline{0} 0]} t^{-1} \cdot e_{3n,3n} \\
& - \chi_{[1\underline{0}]} \cdot e_{3n+1,3n} - \chi_{[0 \underline{0}]} t \cdot e_{3n+1,3n+1} - a_n \chi_{[\underline{0}1]} \cdot e_{0,3n+1} \\
& - \chi_{[01\underline{0}]} t^2 \cdot e_{3n+3,2} - \chi_{[0 \underline{0}]}t \cdot e_{3n+3,3n+3} + \chi_{[\underline{0}]} \cdot e_{3n+4,3n+3} \\
& - \chi_{[0 \underline{0}]} t \cdot e_{3n+4,3n+4} - \chi_{[\underline{0}1]} \cdot e_{3n+5,3n+4} - \chi_{[01\underline{0}]} t \cdot e_{3n+4,3n+2} + \chi_{[\underline{0}10]} t^{-1} \cdot e_{0,3n+2} \\
& + \chi_{[\underline{0}]} \cdot \left( \text{Id}_{3n+6} - e_{0,0} - e_{3n+2,3n+2} - e_{3n+5,3n+5} \right) - \chi_{[\underline{0}1]} \cdot e_{2,2} \in M_{3n+6}(\Q \Gamma).
%\sum_{i=0}^{n-1} \left( E_{3i+1,3i+1} + E_{3i+2,3i+2} + E_{3i+3,3i+3} \right) + E_{3n+2,3n+2} + E_{3n+3,3n+3} \right)
\end{split}
\end{equation*}
\text{ }

\noindent For $W = [1 \underline{1}1]$, $\pi(A)_W$ gives the zero $(3n+6) \times (3n+6)$ matrix, so its kernel has dimension $3n+6$. For a $W = [1\underline{1}0^{k_1}10^{k_2}1 \cdots 0^{k_r}11]$ of length $k = k_1 + \cdots + k_r + (r+1)$, its graph $E_A(W)$ has four different types of connected components $C$, namely:
\begin{enumerate}[a),leftmargin=0.8cm]
\item $C_1$, given by the graphs with only one vertex $\bullet$. Here $\text{dim}_{\Q}(\ker \, A_{C_1}) = 1$, and we have $(3n+7) + (3n+3)r + 3(k_1+\cdots + k_r)$ connected components of this type.
\item $C_2$, given by the graph
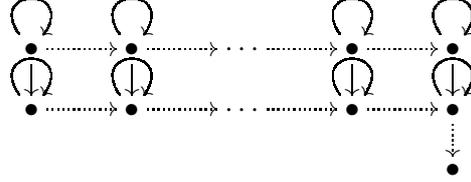
\vspace{-0.1cm}
\begin{equation*}
\xymatrix @R=0.4cm @C=0.9cm {
 \bullet \ar@[red]@(ul,ur) \ar@[red][d] \ar@[red]@{.>}[r] & \bullet \ar@[red]@(ul,ur) \ar@[red][d] \ar@[red]@{.>}[r] & \cdots \ar@[red]@{.>}[r] & \bullet \ar@[red]@(ul,ur) \ar@[red][d] \ar@[red]@{.>}[r] & \bullet \ar@[red]@(ul,ur) \ar@[red][d] \\
 \bullet \ar@[red]@(ul,ur) \ar@[red]@{.>}[r] & \bullet \ar@[red]@(ul,ur) \ar@[red]@{.>}[r] & \cdots \ar@[red]@{.>}[r] & \bullet \ar@[red]@(ul,ur) \ar@[red]@{.>}[r] & \bullet \ar@[red]@(ul,ur) \ar@[red]@{.>}[d] \\
 &  &  &  & \bullet \\
}
\end{equation*}
We compute $\text{dim}_{\Q}(\ker \, A_{C_2}) = 1$. We only have one connected component of this type.
\item $C_3$, given by the graph
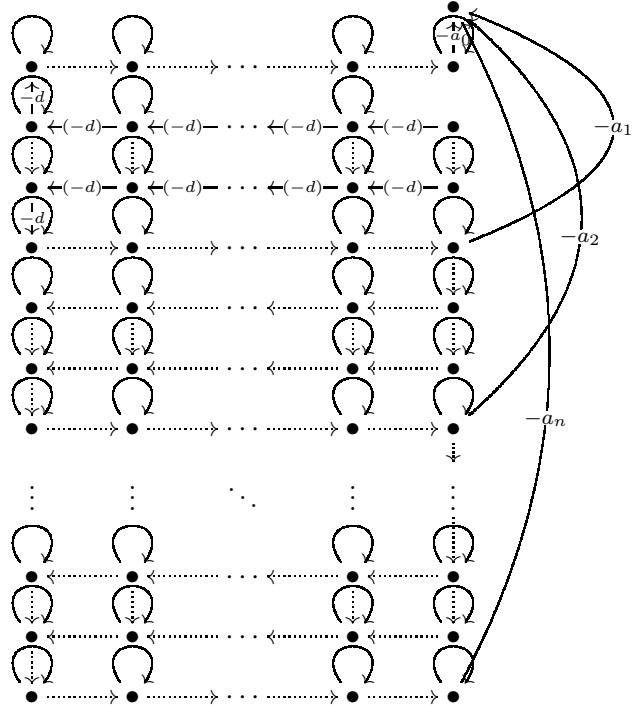
\vspace{-0.3cm}
\begin{equation*}
\xymatrix @R=0.4cm @C=0.9cm {
%\xymatrix@=0.7cm{
  & & & & \bullet \\
 \bullet \ar@[red]@(ul,ur) \ar@[red]@{.>}[r] & \bullet \ar@[red]@(ul,ur) \ar@[red]@{.>}[r] & \cdots \ar@[red]@{.>}[r] & \bullet \ar@[red]@(ul,ur) \ar@[red]@{.>}[r] & \bullet \ar@[red]@(ul,ur) \ar@[red][u]|-{\scriptscriptstyle -a_0} \\
 \bullet \ar@[red]@(ul,ur) \ar@[red][u]|-{\scriptscriptstyle -d} \ar@[red]@{.>}[d] & \bullet \ar@[red]@(ul,ur) \ar@[red]@{.>}[d] \ar@[red][l]|-{\scriptscriptstyle (-d)} & \cdots \ar@[red][l]|-{\scriptscriptstyle (-d)} & \bullet \ar@[red]@(ul,ur) \ar@[red]@{.>}[d] \ar@[red][l]|-{\scriptscriptstyle (-d)} & \bullet \ar@[red]@{.>}[d] \ar@[red][l]|-{\scriptscriptstyle (-d)} \\
 \bullet \ar@[red]@(ul,ur) \ar@[red][d]|-{\scriptscriptstyle -d} & \bullet \ar@[red]@(ul,ur) \ar@[red][l]|-{\scriptscriptstyle (-d)} & \cdots \ar@[red][l]|-{\scriptscriptstyle (-d)} & \bullet \ar@[red]@(ul,ur) \ar@[red][l]|-{\scriptscriptstyle (-d)} & \bullet \ar@[red]@(ul,ur) \ar@[red][l]|-{\scriptscriptstyle (-d)} \\
 \bullet \ar@[red]@(ul,ur) \ar@[red]@{.>}[r] & \bullet \ar@[red]@(ul,ur) \ar@[red]@{.>}[r] & \cdots \ar@[red]@{.>}[r] & \bullet \ar@[red]@(ul,ur) \ar@[red]@{.>}[r] & \bullet \ar@[red]@(ul,ur) \ar@/_5pc/@[red][uuuu]|-{-a_1} \ar@[red]@{.>}[d] \\
 \bullet \ar@[red]@(ul,ur) \ar@[red]@{.>}[d] & \bullet \ar@[red]@(ul,ur) \ar@[red]@{.>}[d] \ar@[red]@{.>}[l] & \cdots \ar@[red]@{.>}[l] & \bullet \ar@[red]@(ul,ur) \ar@[red]@{.>}[d] \ar@[red]@{.>}[l] & \bullet \ar@[red]@(ul,ur) \ar@[red]@{.>}[d] \ar@[red]@{.>}[l] \\
 \bullet \ar@[red]@(ul,ur) \ar@[red]@{.>}[d] & \bullet \ar@[red]@(ul,ur) \ar@[red]@{.>}[l] & \cdots \ar@[red]@{.>}[l] & \bullet \ar@[red]@(ul,ur) \ar@[red]@{.>}[l] & \bullet \ar@[red]@(ul,ur) \ar@[red]@{.>}[l] \\
 \bullet \ar@[red]@(ul,ur) \ar@[red]@{.>}[r] & \bullet \ar@[red]@(ul,ur) \ar@[red]@{.>}[r] & \cdots \ar@[red]@{.>}[r] & \bullet \ar@[red]@(ul,ur) \ar@[red]@{.>}[r] & \bullet \ar@[red]@(ul,ur) \ar@/_4pc/@[red][uuuuuuu]|-(0.45){-a_2} \ar@[red]@{.>}[d] \\
\vdots & \vdots & \ddots & \vdots & \vdots \ar@[red]@{.>}[d] \\
 \bullet \ar@[red]@(ul,ur) \ar@[red]@{.>}[d] & \bullet \ar@[red]@(ul,ur) \ar@[red]@{.>}[d] \ar@[red]@{.>}[l] & \cdots \ar@[red]@{.>}[l] & \bullet \ar@[red]@(ul,ur) \ar@[red]@{.>}[d] \ar@[red]@{.>}[l] & \bullet \ar@[red]@(ul,ur) \ar@[red]@{.>}[d] \ar@[red]@{.>}[l] \\
 \bullet \ar@[red]@(ul,ur) \ar@[red]@{.>}[d] & \bullet \ar@[red]@(ul,ur) \ar@[red]@{.>}[l] & \cdots \ar@[red]@{.>}[l] & \bullet \ar@[red]@(ul,ur) \ar@[red]@{.>}[l] & \bullet \ar@[red]@(ul,ur) \ar@[red]@{.>}[l] \\
 \bullet \ar@[red]@(ul,ur) \ar@[red]@{.>}[r] & \bullet \ar@[red]@(ul,ur) \ar@[red]@{.>}[r] & \cdots \ar@[red]@{.>}[r] & \bullet \ar@[red]@(ul,ur) \ar@[red]@{.>}[r] & \bullet \ar@[red]@(ul,ur) \ar@/_3pc/@[red][uuuuuuuuuuu]|-(0.4){-a_n} \\
% & & \text{for $k_r \geq 2$} & & \\
}% \qquad \qquad \xymatrix@C=0.8cm@R=0.6cm{
% \bullet \ar@[red]@(ul,ur) \ar@[red][d]^{-a_0} \\
% \bullet \ar@[red]@(ul,ur) \ar@[red][d]^{-1} \ar@/^1.5pc/@[red][u]^{-d} \\
% \bullet \ar@[red]@(ul,ur) \ar@[red][d]^{-d} \\
% \bullet \ar@[red]@(ul,ur) \ar@/_2pc/@[red][uu]_{-a_1} \ar@[red][d]^{-1} \\
% \bullet \ar@[red]@(ul,ur) \ar@[red][d]^{-1} \\
% \bullet \ar@[red]@(ul,ur) \ar@[red][d]^{-1} \\
% \bullet \ar@[red]@(ul,ur) \ar@/_2pc/@[red][uuuuu]_{-a_2} \ar@[red][d]^{-1} \\
%\vdots \ar@[red][d]^{-1} \\
% \bullet \ar@[red]@(ul,ur) \ar@[red][d]^{-1} \\
% \bullet \ar@[red]@(ul,ur) \ar@[red][d]^{-1} \\
% \bullet \ar@[red]@(ul,ur) \ar@/_2pc/@[red][uuuuuuuuu]_{-a_n} \\
%\text{for $k_r = 1$} \\
%}
\end{equation*}
We get $\text{dim}_{\Q}(\ker \, A_{C_3}) = 1$. We only have one connected component of this type.
\item Finally $C_4$, given by the graphs
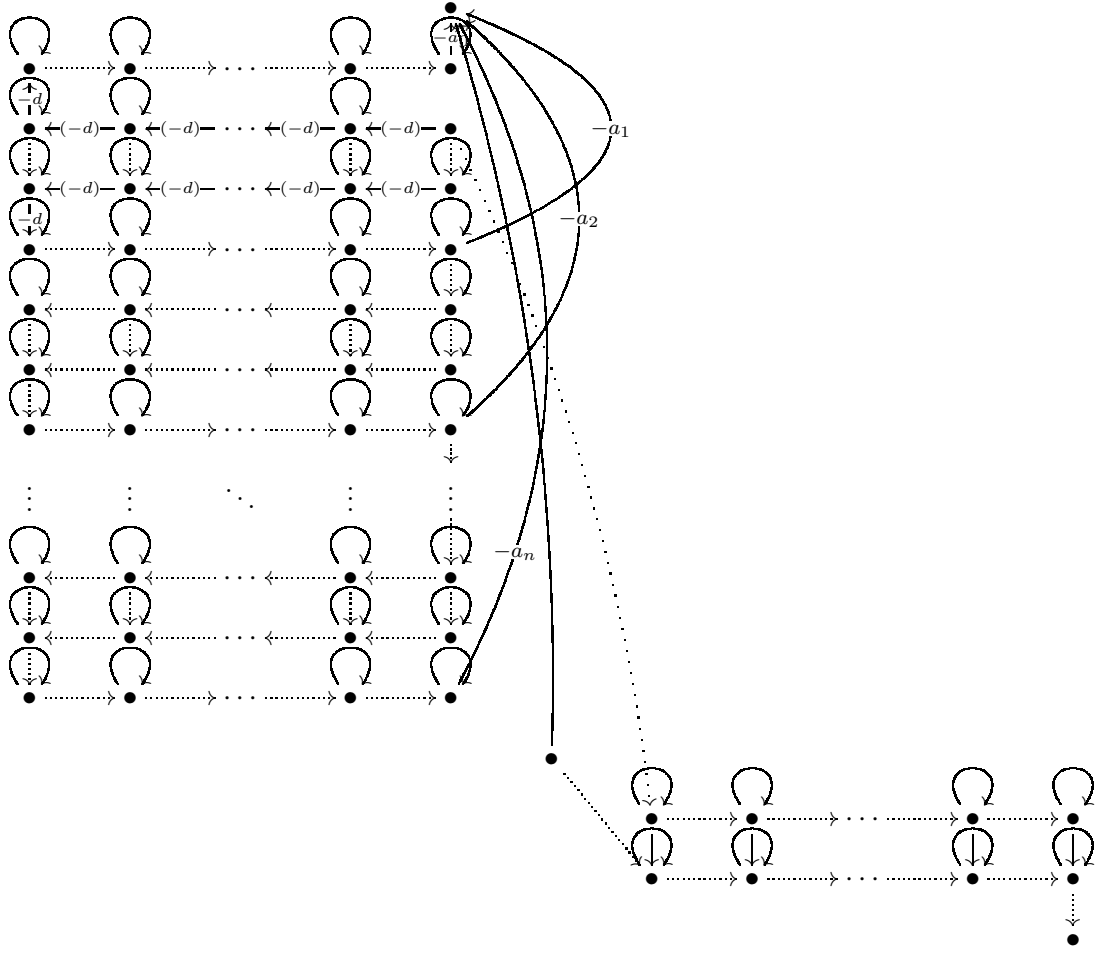
\begin{equation*}
\xymatrix @R=0.4cm @C=0.9cm {
	% T^{l+1}(W) & T^{l+2} & \cdots & T^{l + k_i}(W) & T^{l+k_i + 1}(W) & T^{l+k_i+2}(W) & T^{l+k_i+3}(W) & \cdots & T^{l+k_i+k_{i+1}+1}(W) \\
	& & & & \bullet & & & & & & \\
	\bullet \ar@[red]@(ul,ur) \ar@[red]@{.>}[r] & \bullet \ar@[red]@(ul,ur) \ar@[red]@{.>}[r] & \cdots \ar@[red]@{.>}[r] & \bullet \ar@[red]@(ul,ur) \ar@[red]@{.>}[r] & \bullet \ar@[red]@(ul,ur) \ar@[red][u]|-{\scriptscriptstyle -a_0} & & & & & & \\
	\bullet \ar@[red]@(ul,ur) \ar@[red]@{.>}[d] \ar@[red][u]|-{\scriptscriptstyle -d} & \bullet \ar@[red]@(ul,ur) \ar@[red]@{.>}[d] \ar@[red][l]|-{\scriptscriptstyle (-d)} & \cdots \ar@[red][l]|-{\scriptscriptstyle (-d)} & \bullet \ar@[red]@(ul,ur) \ar@[red]@{.>}[d] \ar@[red][l]|-{\scriptscriptstyle (-d)} & \bullet \ar@[red]@{.>}[d] \ar@[red][l]|-{\scriptscriptstyle (-d)} \ar@/^1pc/@[red]@{.>}[dddddddddddrr] & & & & & & \\
 \bullet \ar@[red]@(ul,ur) \ar@[red][d]|-{\scriptscriptstyle -d} & \bullet \ar@[red]@(ul,ur) \ar@[red][l]|-{\scriptscriptstyle (-d)} & \cdots \ar@[red][l]|-{\scriptscriptstyle (-d)} & \bullet \ar@[red]@(ul,ur) \ar@[red][l]|-{\scriptscriptstyle (-d)} & \bullet \ar@[red]@(ul,ur) \ar@[red][l]|-{\scriptscriptstyle (-d)}  & & & & & & \\
 \bullet \ar@[red]@(ul,ur) \ar@[red]@{.>}[r] & \bullet \ar@[red]@(ul,ur) \ar@[red]@{.>}[r] & \cdots \ar@[red]@{.>}[r] & \bullet \ar@[red]@(ul,ur) \ar@[red]@{.>}[r] & \bullet \ar@[red]@(ul,ur) \ar@/_5pc/@[red][uuuu]|-{-a_1} \ar@[red]@{.>}[d] & & & & & & \\
 \bullet \ar@[red]@(ul,ur) \ar@[red]@{.>}[d] & \bullet \ar@[red]@(ul,ur) \ar@[red]@{.>}[d] \ar@[red]@{.>}[l] & \cdots \ar@[red]@{.>}[l] & \bullet \ar@[red]@(ul,ur) \ar@[red]@{.>}[d] \ar@[red]@{.>}[l] & \bullet \ar@[red]@(ul,ur) \ar@[red]@{.>}[d] \ar@[red]@{.>}[l] & & & & & & \\
 \bullet \ar@[red]@(ul,ur) \ar@[red]@{.>}[d] & \bullet \ar@[red]@(ul,ur) \ar@[red]@{.>}[l] & \cdots \ar@[red]@{.>}[l] & \bullet \ar@[red]@(ul,ur) \ar@[red]@{.>}[l] & \bullet \ar@[red]@(ul,ur) \ar@[red]@{.>}[l] & & & & & & \\
 \bullet \ar@[red]@(ul,ur) \ar@[red]@{.>}[r] & \bullet \ar@[red]@(ul,ur) \ar@[red]@{.>}[r] & \cdots \ar@[red]@{.>}[r] & \bullet \ar@[red]@(ul,ur) \ar@[red]@{.>}[r] & \bullet \ar@[red]@(ul,ur) \ar@/_4pc/@[red][uuuuuuu]|-{-a_2} \ar@[red]@{.>}[d] & & & & & & \\
\vdots & \vdots & \ddots & \vdots & \vdots \ar@[red]@{.>}[d] & & & & & & \\
 \bullet \ar@[red]@(ul,ur) \ar@[red]@{.>}[d] & \bullet \ar@[red]@(ul,ur) \ar@[red]@{.>}[d] \ar@[red]@{.>}[l] & \cdots \ar@[red]@{.>}[l] & \bullet \ar@[red]@(ul,ur) \ar@[red]@{.>}[d] \ar@[red]@{.>}[l] & \bullet \ar@[red]@(ul,ur) \ar@[red]@{.>}[d] \ar@[red]@{.>}[l] & & & & & & \\
 \bullet \ar@[red]@(ul,ur) \ar@[red]@{.>}[d] & \bullet \ar@[red]@(ul,ur) \ar@[red]@{.>}[l] & \cdots \ar@[red]@{.>}[l] & \bullet \ar@[red]@(ul,ur) \ar@[red]@{.>}[l] & \bullet \ar@[red]@(ul,ur) \ar@[red]@{.>}[l] & & & & & & \\
 \bullet \ar@[red]@(ul,ur) \ar@[red]@{.>}[r] & \bullet \ar@[red]@(ul,ur) \ar@[red]@{.>}[r] & \cdots \ar@[red]@{.>}[r] & \bullet \ar@[red]@(ul,ur) \ar@[red]@{.>}[r] & \bullet \ar@[red]@(ul,ur) \ar@/_3pc/@[red][uuuuuuuuuuu]|-(0.2){-a_n} & & & & & & \\
 & & & & & \bullet \ar@[red]@{.>}[ddr] \ar@/_1pc/@[red][uuuuuuuuuuuul] & & & & & \\
 & & & & & & \bullet \ar@[red]@(ul,ur) \ar@[red][d] \ar@[red]@{.>}[r] & \bullet \ar@[red]@(ul,ur) \ar@[red][d] \ar@[red]@{.>}[r] & \cdots \ar@[red]@{.>}[r] & \bullet \ar@[red]@(ul,ur) \ar@[red][d] \ar@[red]@{.>}[r] & \bullet \ar@[red]@(ul,ur) \ar@[red][d] \\
 & & & & & & \bullet \ar@[red]@(ul,ur) \ar@[red]@{.>}[r] & \bullet \ar@[red]@(ul,ur) \ar@[red]@{.>}[r] & \cdots \ar@[red]@{.>}[r] & \bullet \ar@[red]@(ul,ur) \ar@[red]@{.>}[r] & \bullet \ar@[red]@(ul,ur) \ar@[red]@{.>}[d] \\
 & & & & & & & & & & \bullet \\
}
\end{equation*}
Here
\[ \text{dim}_{\Q}(\ker \, A_{C_4}) = \left\{ \begin{array}{lr}
																				3 & \text{ if $k_{i+1} = p(k_i)d^{k_i}$} \\
																				2 & \text{ otherwise}
																			 \end{array}\right\} = 2 + \delta_{k_{i+1},p(k_i)d^{k_i}}. \]
We have $r-1$ connected components of this kind.
\end{enumerate}
In this case, we compute
\begin{align*}
b^{(2)}(A) & = \mu([1\ul{1}1]) \text{dim}_{\Q}(\ker \, \pi(A)_{[1\ul{1}1]}) + \sum_{r \geq 1} \sum_{k_1,...,k_r \geq 1} \mu([1\ul{1}0^{k_1}10^{k_2} \cdots 10^{k_r}11]) \Big( \sum_{i=1}^4 \dim_{\Q}(\ker \, A_{C_i}) \Big) \\
& = \frac{3n+6}{2^3} + \sum_{r \geq 1} \sum_{k_1,...,k_r \geq 1} \frac{1}{2^{k_1+ \cdots +k_r + (r+3)}} \Big[ (3n+7) + (3n+5)r + 3(k_1 + \cdots + k_r) \Big] \\
& \qquad \qquad \quad + \sum_{r \geq 1} \sum_{k_1,...,k_r \geq 1} \frac{1}{2^{k_1+ \cdots +k_r + (r+3)}} \Big[ \sum_{i=1}^{r-1} \delta_{k_{i+1},p(k_i)d^{k_i}} \Big] \\
& = \frac{3n+6}{8} + \frac{1}{8} \sum_{r \geq 1} \frac{(3n+7)+(3n+11)r}{2^r} + \frac{1}{8} \Big( \sum_{r \geq 1} \frac{r-1}{2^r} \Big) \Big( \sum_{k \geq 1} \frac{1}{2^{k + p(k)d^k}} \Big) \\
& = \frac{12n+35}{8} + \frac{1}{8} \sum_{k \geq 1} \frac{1}{2^{k+p(k)d^k}}
\end{align*}
which is again an irrational number, and even transcendental (see e.g. \cite{Tan}). Just as before, the rational number accompanying $\alpha = \sum_{k \geq 1} \frac{1}{2^{k+p(k)d^k}}$ is of the form $\frac{1}{2^m}$.

After these examples one can derive the pattern in order to obtain an exponent of the form $p_0(k) + p_1(k)d_1^k + \cdots + p_n(k) d_n^k$ by simply adding more levels, i.e. by considering matrices of higher dimension, and gluing the corresponding graphs in an appropriate way. We write down the corresponding element that gives rise to such a pattern. If we let $N = m_0 + \cdots + m_n$ to be the sum of the degrees of the polynomials $p_0,...,p_n$ respectively, with $p_i(x)=\sum_{j=0}^{m_i} a_{j,i}x^j$ for $0 \leq i \leq n$, then the element realizing the preceding pattern belongs to $M_{3N+n+5}(\calA_{1/2})$, and is given explicitly by

\hfill\tikzmark{right}
\begin{equation*}
\begin{split}
A = & \tikzmark{1} - \chi_{[\underline{0} 0]}t^{-1} \cdot e_{1,1} - \chi_{[\underline{0}]} \cdot e_{2,1} - \chi_{[\underline{0} 0]}t^{-1} \cdot e_{2,2} \\  
& - \chi_{[1\underline{0}]} \cdot e_{3,2} - \chi_{[0 \underline{0}]} t \cdot e_{3,3} - \chi_{[\underline{0}1]} \cdot e_{4,3} - (a_{1,0} - 1) \chi_{[\underline{0}1]} \cdot e_{0,3} \\
& + \sum_{j=1}^{m_0-2} \Big( - \chi_{[\underline{0} 0]} t^{-1} \cdot e_{3j+1,3j+1} - \chi_{[\underline{0}]} \cdot e_{3j+2,3j+1} - \chi_{[\underline{0} 0]} t^{-1} \cdot e_{3j+2,3j+2} \\
& \qquad \qquad - \chi_{[1\underline{0}]} \cdot e_{3j+3,3j+2} - \chi_{[0 \underline{0}]} t \cdot e_{3j+3,3j+3} \\
& \qquad \qquad - \chi_{[\underline{0}1]} \cdot e_{3j+4,3j+3} - a_{j+1,0} \chi_{[\underline{0}1]} \cdot e_{0,3j+3} \Big) \\
& - \chi_{[\underline{0} 0]} t^{-1} \cdot e_{3m_0-2,3m_0-2} - \chi_{[\underline{0}]} \cdot e_{3m_0-1,3m_0-2} - \chi_{[\underline{0} 0]} t^{-1} \cdot e_{3m_0-1,3m_0-1} \\
& \tikzmark{2} - \chi_{[1\underline{0}]} \cdot e_{3m_0,3m_0-1} - \chi_{[0 \underline{0}]} t \cdot e_{3m_0,3m_0} - a_{m_0,0} \chi_{[\underline{0}1]} \cdot e_{0,3m_0} \\
& \quad \\
& \qquad \qquad \qquad \qquad \qquad - \chi_{[\underline{0}1]} \cdot e_{3m_0+2,1} \\
& \quad \\
& \tikzmark{3} - d_1 \chi_{[1\underline{0}]} \cdot e_{3m_0+1,3m_0+2} - \chi_{[0 \underline{0}]} t \cdot e_{3m_0+1,3m_0+1} - a_{0,1} \chi_{[\underline{0}1]} \cdot e_{0,3m_0+1} \\
& - d_1 \chi_{[\underline{0} 0]} t^{-1} \cdot e_{3m_0+2,3m_0+2} - \chi_{[\underline{0}]} \cdot e_{3m_0+3,3m_0+2} \\
& - d_1 \chi_{[\underline{0} 0]} t^{-1} \cdot e_{3m_0+3,3m_0+3} - d_1 \chi_{[1\underline{0}]} \cdot e_{3m_0+4,3m_0+3} \\
& - \chi_{[0 \underline{0}]} t \cdot e_{3m_0+4,3m_0+4} - \chi_{[\underline{0}1]} \cdot e_{3m_0+5,3m_0+4} - a_{1,1} \chi_{[\underline{0}1]} \cdot e_{0,3m_0+4} \\
& + \sum_{j=1}^{m_1-2} \Big( - \chi_{[\underline{0} 0]} t^{-1} \cdot e_{3m_0+3j+2,3m_0+3j+2} - \chi_{[\underline{0}]} \cdot e_{3m_0+3j+3,3m_0+3j+2} \\
& \qquad \qquad - \chi_{[\underline{0} 0]} t^{-1} \cdot e_{3m_0+3j+3,3m_0+3j+3} - \chi_{[1\underline{0}]} \cdot e_{3m_0+3j+4,3m_0+3j+3} \\
& \qquad \qquad - \chi_{[0 \underline{0}]} t \cdot e_{3m_0+3j+4,3m_0+3j+4} - \chi_{[\underline{0}1]} \cdot e_{3m_0+3j+5,3m_0+3j+4} \\
& \qquad \qquad - a_{j+1,1} \chi_{[\underline{0}1]} \cdot e_{0,3m_0+3j+4} \Big) \\
& - \chi_{[\underline{0} 0]} t^{-1} \cdot e_{3(m_0+m_1)-1,3(m_0+m_1)-1} - \chi_{[\underline{0}]} \cdot e_{3(m_0+m_1),3(m_0+m_1)-1} \\
& - \chi_{[\underline{0} 0]} t^{-1} \cdot e_{3(m_0+m_1),3(m_0+m_1)} - \chi_{[1\underline{0}]} \cdot e_{3(m_0+m_1)+1,3(m_0+m_1)} \\
& \tikzmark{4} - \chi_{[0 \underline{0}]} t \cdot e_{3(m_0+m_1)+1,3(m_0+m_1)+1} - a_{m_1,1} \chi_{[\underline{0}1]} \cdot e_{0,3(m_0+m_1)+1} \\
& \quad \\
& \qquad \qquad \qquad \qquad \qquad - \chi_{[\underline{0}1]} \cdot e_{3(m_0+m_1)+3,1} \\
& \qquad \qquad \qquad \qquad \qquad \qquad \qquad \qquad \vdots \\
& \qquad \qquad \qquad \qquad \qquad - \chi_{[\underline{0}1]} \cdot e_{3(N-m_n)+n+1,1} \\
& \quad \\
& \tikzmark{5} - d_n \chi_{[1\underline{0}]} \cdot e_{3(N-m_n)+n,3(N-m_n)+n+1} - \chi_{[0 \underline{0}]} t \cdot e_{3(N-m_n)+n,3(N-m_n)+n} \\
& - a_{0,n} \chi_{[\underline{0}1]} \cdot e_{0,3(N-m_n)+n} - d_n \chi_{[\underline{0} 0]} t^{-1} \cdot e_{3(N-m_n)+n+1,3(N-m_n)+n+1} \\
& - \chi_{[\underline{0}]} \cdot e_{3(N-m_n)+n+2,3(N-m_n)+n+1} - d_n \chi_{[\underline{0} 0]} t^{-1} \cdot e_{3(N-m_n)+n+2,3(N-m_n)+n+2} \\
& \tikzmark{6} - d_n \chi_{[1\underline{0}]} \cdot e_{3(N-m_n)+n+3,3(N-m_n)+n+2} - \chi_{[0 \underline{0}]} t \cdot e_{3(N-m_n)+n+3,3(N-m_n)+n+3} \\
\end{split}
\end{equation*}

\begin{tikzpicture}[overlay, remember picture]
\node[anchor=base] (a) at (pic cs:1) {\vphantom{h}};
\node[anchor=base] (b) at (pic cs:2) {\vphantom{g}};
\draw [decoration={brace,amplitude=0.5em},decorate,ultra thick,black]
 (a.north -| {pic cs:right}) -- (b.south -| {pic cs:right}) node[midway,below,rotate=90] {$1^{\text{st}}$ polynomial};
\end{tikzpicture}
\begin{tikzpicture}[overlay, remember picture]
\node[anchor=base] (a) at (pic cs:3) {\vphantom{h}};
\node[anchor=base] (b) at (pic cs:4) {\vphantom{g}};
\draw [decoration={brace,amplitude=0.5em},decorate,ultra thick,black]
 (a.north -| {pic cs:right}) -- (b.south -| {pic cs:right}) node[midway,below,rotate=90] {$2^{\text{nd}}$ polynomial};
\end{tikzpicture}
\begin{tikzpicture}[overlay, remember picture]
\node[anchor=base] (a) at (pic cs:5) {\vphantom{h}};
\node[anchor=base] (b) at (pic cs:6) {\vphantom{g}};
\draw [decoration={brace,amplitude=0.5em},decorate,ultra thick,black]
 (a.north -| {pic cs:right}) -- (b.south -| {pic cs:right}) node[midway,below,rotate=90] {$n^{\text{th}}$ polynomial};
\end{tikzpicture}

\begin{equation*}
\begin{split}
\text{ } & \tikzmark{7} - \chi_{[\underline{0}1]} \cdot e_{3(N-m_n)+n+4,3(N-m_n)+n+3} - a_{1,n} \chi_{[\underline{0}1]} \cdot e_{0,3(N-m_n)+n+3} \\
& + \sum_{j=1}^{m_n-2} \Big( - \chi_{[\underline{0} 0]} t^{-1} \cdot e_{3(N-m_n)+3j+n+1,3(N-m_n)+3j+n+1} \\
& \qquad \qquad - \chi_{[\underline{0}]} \cdot e_{3(N-m_n)+3j+n+2,3(N-m_n)+3j+n+1} \\
& \qquad \qquad - \chi_{[\underline{0} 0]} t^{-1} \cdot e_{3(N-m_n)+3j+n+2,3(N-m_n)+3j+n+2} \\
& \qquad \qquad - \chi_{[1\underline{0}]} \cdot e_{3(N-m_n)+3j+n+3,3(N-m_n)+3j+n+2} \\
& \qquad \qquad - \chi_{[0 \underline{0}]} t \cdot e_{3(N-m_n)+3j+n+3,3(N-m_n)+3j+n+3} \\
& \qquad \qquad - \chi_{[\underline{0}1]} \cdot e_{3(N-m_n)+3j+n+4,3(N-m_n)+3j+n+3} \\
& \qquad \qquad - a_{j+1,n} \chi_{[\underline{0}1]} \cdot e_{0,3(N-m_n)+3j+n+3} \Big) \\
& - \chi_{[\underline{0} 0]} t^{-1} \cdot e_{3N+n-2,3N+n-2} - \chi_{[\underline{0}]} \cdot e_{3N+n-1,3N+n-2} \\
& - \chi_{[\underline{0} 0]} t^{-1} \cdot e_{3N+n-1,3N+n-1} - \chi_{[1\underline{0}]} \cdot e_{3N+n,3N+n-1} \\
& \tikzmark{8} - \chi_{[0 \underline{0}]} t \cdot e_{3N+n,3N+n} - a_{m_n,n} \chi_{[\underline{0}1]} \cdot e_{0,3N+n} \\
& \qquad \qquad \quad - \chi_{[01\underline{0}]} t^2 \cdot e_{3N+n+2,1} \\
& \qquad \qquad \quad + \chi_{[\underline{0}10]} t^{-1} \cdot e_{0,3N+n+1} \\
& \qquad \qquad \quad - \chi_{[01\underline{0}]} t \cdot e_{3N+n+3,3N+n+1} \\
& \tikzmark{9} - \chi_{[0 \underline{0}]} t \cdot e_{3N+n+2,3N+n+2} + \chi_{[\underline{0}]} \cdot e_{3N+n+3,3N+n+2} \\
& \tikzmark{10} - \chi_{[0 \underline{0}]} t \cdot e_{3N+n+3,3N+n+3} - \chi_{[\underline{0}1]} \cdot e_{3N+n+4,3N+n+3} \\
& + \chi_{[\underline{0}]} \cdot \left( \text{Id}_{3N+n+5} - e_{0,0} - e_{3N+n+1,3N+n+1} - e_{3N+n+4,3N+n+4} \right) - \chi_{[\underline{0}1]} \cdot e_{1,1} \in M_{3N+n+5}(\Q \Gamma).
\end{split}
\end{equation*}

\begin{tikzpicture}[overlay, remember picture]
\node[anchor=base] (a) at (pic cs:7) {\vphantom{h}};
\node[anchor=base] (b) at (pic cs:8) {\vphantom{g}};
\draw [decoration={brace,amplitude=0.5em},decorate,ultra thick,black]
 (a.north -| {pic cs:right}) -- (b.south -| {pic cs:right}) node[midway,below,rotate=90] {$n^{\text{th}}$ polynomial};
\end{tikzpicture}
\begin{tikzpicture}[overlay, remember picture]
\node[anchor=base] (a) at (pic cs:9) {\vphantom{h}};
\node[anchor=base] (b) at (pic cs:10) {\vphantom{g}};
\draw [decoration={brace,amplitude=0.5em},decorate,ultra thick,black]
 (a.north -| {pic cs:right}) -- (b.south -| {pic cs:right}) node[midway,below,rotate=90] {last graph};
\end{tikzpicture}

\noindent The elements in between connect the different polynomials $p_i$, and the contributions (monomials) of the polynomials (that is, the sum $p_0(k) + p_1(k)d_1^k + \cdots + p_n(k)d_n^k$) are accumulated in the $\chi_{[\underline{0}1]} \cdot e_{0,0}$ component. A simplified schematic of a prototypical graph appearing here is as follows.

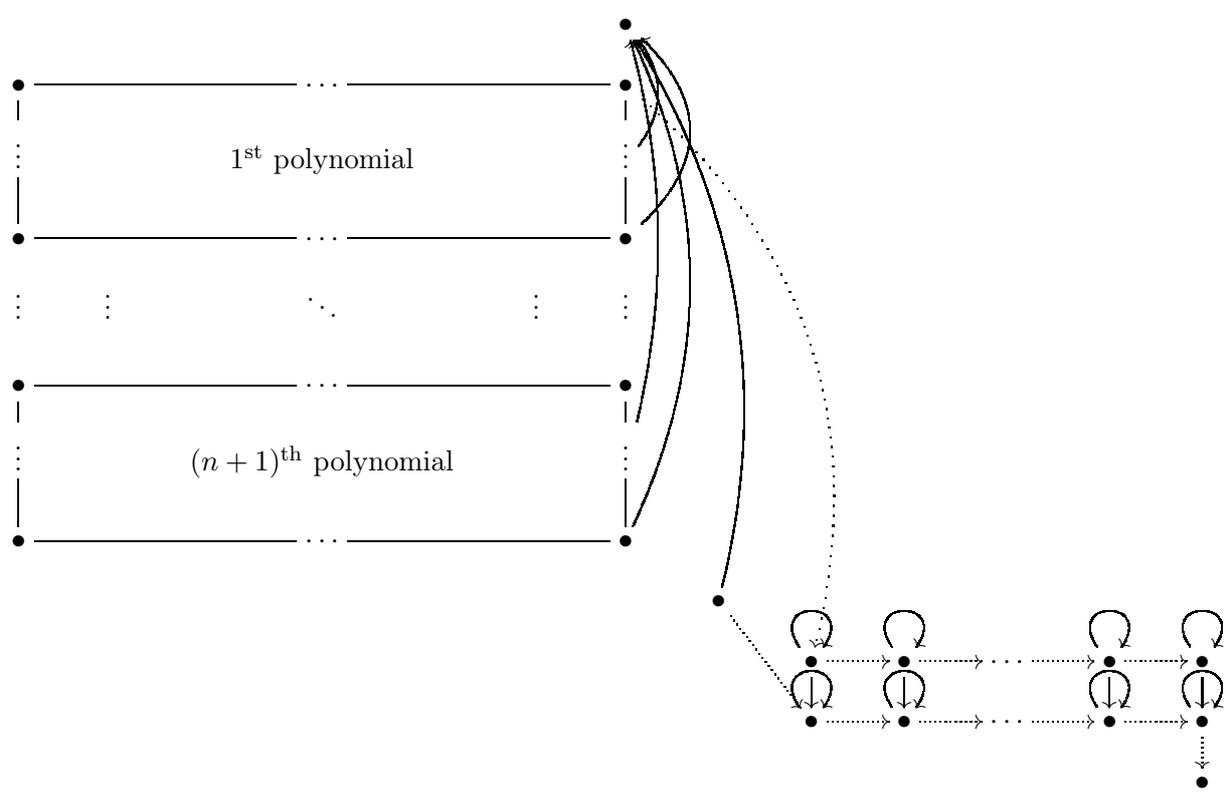
\begin{equation*}
\xymatrix@C=0.8cm@R=0.4cm{
 & & & & \bullet & & & & & & \\
 \bullet \ar@{-}@[red][rr] & & \cdots \ar@{-}@[red][rr] & & \bullet \ar@{-}@[red][d] \ar@/^3pc/@[red]@{.>}[ddddddddrr] & & & & & & \\
 \vdots \ar@{-}@[red][u] \ar@{-}@[red][d] & & \text{$1^{\text{st}}$ polynomial} & & \vdots \ar@{-}@[red][d] \ar@/_1pc/@[red][uu] & & & & & & \\
 \bullet \ar@{-}@[red][rr] & & \cdots \ar@{-}@[red][rr] & & \bullet \ar@/_2pc/@[red][uuu] & & & & & & \\
% \bullet \ar@[red][d]^{-1} & \bullet \ar@[red][d]^{-1} \ar@[red][l]^{-1} & \cdots \ar@[red][l]^{-1} & \bullet \ar@[red][d]^{-1} \ar@[red][l]^{-1} & \bullet \ar@[red][d]^{-1} \ar@[red][l]^{-1} & & & & \cdots & & \\
% \bullet \ar@[red][d]^{-1} & \bullet \ar@[red][l]^{-1} & \cdots \ar@[red][l]^{-1} & \bullet \ar@[red][l]^{-1} & \bullet \ar@[red][l]^{-1} & & & & \cdots & & \\
% \bullet \ar@[red][r]^{-1} & \bullet \ar@[red][r]^{-1} & \cdots \ar@[red][r]^{-1} & \bullet \ar@[red][r]^{-1} & \bullet \ar@/_2pc/@[red][uuuuu]_{-a_2} \ar@[red][d]^{-1} & & & & \cdots & & \\
\vdots & \vdots & \ddots & \vdots & \vdots & & & & & & \\
 \bullet \ar@{-}@[red][d] \ar@{-}@[red][rr] & & \cdots \ar@{-}@[red][rr] & & \bullet \ar@{-}@[red][d] & & & & & & \\
 \vdots \ar@{-}@[red][d] & & \text{$(n+1)^{\text{th}}$ polynomial} & & \vdots \ar@{-}@[red][d] \ar@/_1pc/@[red][uuuuuu] & & & & & & \\
 \bullet \ar@{-}@[red][rr] & & \cdots \ar@{-}@[red][rr] & & \bullet \ar@/_2pc/@[red][uuuuuuu] & & & & & & \\
 & & & & & \bullet \ar@[red]@{.>}[ddr] \ar@/_2pc/@[red][uuuuuuuul] & & & & & \\
 & & & & & & \bullet \ar@[red]@(ul,ur) \ar@[red][d] \ar@[red]@{.>}[r] & \bullet \ar@[red]@(ul,ur) \ar@[red][d] \ar@[red]@{.>}[r] & \cdots \ar@[red]@{.>}[r] & \bullet \ar@[red]@(ul,ur) \ar@[red][d] \ar@[red]@{.>}[r] & \bullet \ar@[red]@(ul,ur) \ar@[red][d] \\
 & & & & & & \bullet \ar@[red]@(ul,ur) \ar@[red]@{.>}[r] & \bullet \ar@[red]@(ul,ur) \ar@[red]@{.>}[r] & \cdots \ar@[red]@{.>}[r] & \bullet \ar@[red]@(ul,ur) \ar@[red]@{.>}[r] & \bullet \ar@[red]@(ul,ur) \ar@[red]@{.>}[d] \\
 & & & & & & & & & & \bullet \\
}
\end{equation*}
Using the same procedure as before, a straightforward (but quite tedious) computation allows us to conclude the proof. We only write down the contribution corresponding to the components $C_4$ of the graph $E_A(W)$ where $W$ is of the second type, which are the ones that contribute to the irrationality of the $\ell^2$-Betti number. We get
\begin{align*}
\text{dim}_{\Q}(\ker \, A_{C_4}) %& = \left\{ \begin{array}{lr}
														%						3 & \text{ if $k_{i+1} = p_0(k_i) - k_i - a_{0,0} + p_1(k_i)d_1^{k_i} + \cdots + p_n(k_i)d_n^{k_i}$} \\
														%						2 & \text{ otherwise}
														%					 \end{array}\right\} \\
																			 = 2 + \delta_{k_{i+1},p_0(k_i) - k_i - a_{0,0} + p_1(k_i)d_1^{k_i} + \cdots + p_n(k_i)d_n^{k_i}}.
																			\end{align*}
We have $r-1$ connected components of this kind. Thus the contribution to $b^{(2)}(A)$ coming from these graphs is

\begin{align*}
& \sum_{r \geq 1} \sum_{k_1,...,k_r \geq 1} \frac{1}{2^{k_1 + \cdots + k_r + (r+3)}}\Big[ 2(r-1) + \sum_{i=1}^{r-1} \delta_{k_{i+1},p_0(k_i) - k_i - a_{0,0} + p_1(k_i)d_1^{k_i} + \cdots + p_n(k_i)d_n^{k_i}} \Big] \\
& \qquad \quad = \frac{2}{8} \sum_{r \geq 1}\frac{r-1}{2^r} + \frac{2^{a_{0,0}}}{8}\Big( \sum_{r \geq 1} \frac{r-1}{2^r} \Big) \Big( \sum_{k \geq 1} \frac{1}{2^{p_0(k) + p_1(k)d_1^{k} + \cdots + p_n(k)d_n^{k}}} \Big) \\
& \qquad \quad = \frac{1}{4} + \frac{2^{a_{0,0}}}{8} \sum_{k \geq 1} \frac{1}{2^{p_0(k) + p_1(k)d_1^{k} + \cdots + p_n(k)d_n^{k}}}.
\end{align*}
We leave the rest of the details to the reader. Note, again, that in the special case $a_{0,0} = 0$ the irrational number $\alpha = \sum_{k \geq 1} \frac{1}{2^{p_0(k) + p_1(k)d_1^{k} + \cdots + p_n(k)d_n^{k}}}$ is accompanied by a rational number of the form $\frac{1}{2^m}$.\\

For the second part of the theorem, note that under the assumption $a_{0,0} = 0$ we have always obtained irrational numbers of the form $q_0 + q_1 \alpha$, being $\alpha$ the irrational number
$$\alpha = \sum_{k \geq 1} \frac{1}{2^{p_0(k) + p_1(k)d_1^{k} + \cdots + p_n(k)d_n^{k}}},$$
and $q_0,q_1$ non-zero rational numbers with $q_1$ of the form $\frac{1}{2^m}$ for some $m \geq 1$. In particular, $q_0 + q_1 \alpha$ belongs to $\calG(\Gamma,\Q)$. By a result of Ara and Goodearl \cite[Corollary 6.14]{AG}, the group $\calG(\Gamma,\Q)$ contains the rational numbers, and hence the element $q_1 \alpha$ belongs to $\calG(\Gamma,\Q)$ too. Since $q_1 = \frac{1}{2^m}$, we get
$$\alpha = q_1 \alpha + \stackrel{2^m}{\cdots} + q_1 \alpha \in \calG(\Gamma,\Q).$$
This completes the proof of the theorem.
\end{proof}

\begin{remark}\label{remark-exotic.patterns}
By making use of the ideas and techniques developed in Theorem \ref{theorem-irrational.dimensions}, one can construct elements whose associated $\ell^2$-Betti number has a binary expansion following different kinds of exotic patterns. For example, by gluing in an appropriate way two graphs corresponding to different polynomials $p_1(x)$ and $p_2(x)$, that is, by constructing graphs of the form $C_3$ but substituting the bottom right graph (the one contributing to $k_{i+1}$) by another graph corresponding to the polynomial $p_2(x)$, we can obtain terms in the computation of the associated $\ell^2$-Betti number of the form
$$\sum_{k_1,k_2 \geq 1} \frac{\delta_{p_1(k_1),p_2(k_2)}}{2^{k_1+k_2}},$$
and more generally of the form
$$\sum_{k_1,k_2 \geq 1} \frac{\delta_{p_0(k_1) + p_1(k_1)d_1^{k_1} + \cdots + p_n(k_1) d_n^{k_1},q_0(k_2) + q_1(k_2)l_1^{k_2} + \cdots + q_m(k_2) l_m^{k_2}}}{2^{k_1+k_2}},$$
being $p_i(x),q_j(x)$ polynomials satisfying the hypotheses of the theorem, and $d_i,l_j \geq 2$ integers.
\end{remark}

To conclude this section, it may be instructive to compute some rational values of $\ell^2$-Betti numbers. In \cite{GZ} (cf. \cite{DiSc,AG,Gra14}) the authors compute the $\ell^2$-Betti number of the element $a_0 = e_0t + t^{-1} e_0 = \chi_{X \backslash E_0} t + t^{-1} \chi_{X \backslash E_0}$, which belongs to the first of our approximating algebras, $\calA_0$. We will compute, in general, the $\ell^2$-Betti number of the element $a_n = \chi_{X \backslash E_n} t + t^{-1} \chi_{X \backslash E_n} \in \Q \Gamma$ belonging to the $*$-subalgebra $\calA_n$. Under $\pi_n$ --and recalling that $\gotR_n = \Q \times \prod_{k \geq 1} M_{m+k}(\Q)^{\text{Fib}_m(k)}$ with $m = 2n+1$-- it gives
$$\pi_n(a_n) = ( 0, ( t_{m+k} + t_{m+k}^* , \stackrel{\text{Fib}_m(k)}{...}, t_{m+k} + t_{m+k}^*)_{k \geq 1}),$$
where $t_r$ is the $r \times r$ lower-triangular matrix given by
%$$t_r = e_{10} + e_{21} + \cdots + e_{r-1,r-2},$$
%being the family $\{e_{ij}\}_{0 \leq i,j \leq r-1}$ a complete system of matrix units for $M_r(K)$. 
\[ t_r = \left( \begin{array}{@{}c@{}}
		\begin{matrix}
			0 & & & & \bigzero \\
			1 & 0 & & & \\
			& \ddots & \ddots & & \\
			& & \ddots & 0 &  \\
			\bigzero & & & 1 & 0
		\end{matrix}											
	 \end{array} \right). \]
It is then straightforward to show that
$$\text{dim}_{\Q}(\ker \, (t_{m+k} + t_{m+k}^*)) = \begin{cases}
																						1 & \text{ if $k$ is even} \\
																						0 & \text{ otherwise}
																					\end{cases}$$
so
\begin{align*}
b^{(2)}(a_n) = \frac{1}{2^{m+1}} + \sum_{k \geq 1} \frac{\text{Fib}_m(2k)}{2^{2m+2k}}.
\end{align*}
This sum can be computed by using the summation rule
$$\sum_{k \geq 1} \frac{\text{Fib}_m(2k)}{2^{2k}} = 2^{m-1} \frac{2^{m+1}-1}{2^{m+2}+1},$$%, \qquad \sum_{k \geq 1} \frac{k \text{Fib}_m(2k+1)}{2^{2k+1}} = 2^{m-1} \frac{2+2^{m+2}-2^{m+1}}{1+2^{m+2}},$$
whose proof can be found in \cite[Lemma 3.2.11]{Claramunt}. We then have
$$b^{(2)}(a_n) = \frac{3}{1+2^{2n+3}}.$$
Note that $b^{(2)}(a_0) = \frac{1}{3}$, and we recover the result from \cite{GZ,DiSc}. Also, as $n \ra \infty$, this value tends to zero, as expected since $a_n \ra t + t^{-1}$ in rank, which is invertible inside $\gotR_{\rk}$.

\subsection{Rational series and \texorpdfstring{$\ell^2$}{}-Betti numbers}\label{subsection-rationalseries}

In this section we let again $K \subseteq \C$ be any field closed under complex conjugation. A more algebraic perspective to attack the problem of computing values from $\calG(\calA_{1/2}) \subseteq \calG(\Gamma,K)$ is through the determination of the $*$-regular closure $\calR_{1/2}$ of the algebra $\calA_{1/2}$ seen inside $\gotR_{1/2} = K \times \prod_{k \geq 1} M_{k+2}(K)^{\text{Fib}_2(k)}$ through the embedding $\pi_{1/2} : \calA_{1/2} \hookrightarrow \gotR_{1/2}$, and taking advantage of \cite[Proposition 4.1]{AC2} which states that the values of ranks of elements from $\calR_{1/2}$ are all included in $\calG(\Gamma,K)$.

In order to clarify the exact relationship of this approach with Atiyah's problem, it is convenient to introduce the following definitions for a general unital ring.
\begin{definition}\label{def:semigroupsCandCprime}
Let $R$ be a unital ring and let $\rk$ be a Sylvester matrix rank function on $R$. We denote by $\calC_{\rk} (R)$ the set of real numbers of the form $k- \rk(A)$ for $A\in M_k(R)$, $k\ge 1$. We denote by $\calC'_{\rk}(R)$ the set of real numbers of the form $\rk(A)$ for $A\in M_k(R)$, $k\ge 1$. Both $\calC_{\rk}(R)$ and $\calC_{\rk}'(R)$ are subsemigroups of $(\R^+,+)$. We denote by $\calG _{\rk}(R)$ the subgroup of $\R$ generated by $\calC_{\rk}(R)$. Clearly
$$\calG_{\rk}(R)= \calC _{\rk}(R) -  \calC _{\rk}(R) =  \calC _{\rk}'(R) - \calC_{\rk}'(R).$$
\end{definition}
\begin{remark}\label{rem:semigroupsCandCprime}
With this notation, \cite[Corollary 6.2]{Jaik} says that, if $\calU $ is a $*$-regular ring and $R$ is a $*$-subring of $\calU$ such that $\calU$ is the $*$-regular closure of $R$ in $\calU$, then we have
$$\calG_{\rk} (R) = \calG_{\rk}(\calU)$$
for every Sylvester matrix rank function $\rk$ on $\calU$. Thus the main object of study within this approach is the group $\calG_{\rk}(R)$. Of course $\calC (G,K) = \calC_{\rk}(KG)$ for any discrete group $G$, where $\rk$ is the canonical rank. It would be interesting to know conditions under which $\calC_{\rk}(R)= \calC_{\rk}'(R)$ and under which $\calC_{\rk}(R)= \calG_{\rk}(R)\cap \R^+$, or $\calC_{\rk}'(R)= \calG_{\rk}(R)\cap \R^+$. Note that
$$\forall g\in \calG_{\rk}(R)\, \exists t\in \Z^+ : g+t\in \calC_{\rk}(R) \iff 
\forall g\in \calG_{\rk}(R)\, \exists t\in \Z^+ : g+t\in \calC_{\rk}'(R).$$
In case $R$ is von Neumann regular, we have that $\calC_{\rk}(R) = \calC_{\rk}'(R)$ and the above equivalent conditions hold automatically, for every Sylvester matrix rank function $\rk$.
\end{remark}

We briefly remind some concepts from \cite{AC2}. Write $\calA_{1/2} = \bigoplus_{i\in \Z} \calA_{1/2,i}t^i$, where $\calA_{1/2,0}= \calA_{1/2}\cap C_K(X)$ and
$$\calA_{1/2,i} = \chi_{X \setminus (E_{1/2} \cup T(E_{1/2}) \cup \cdots \cup T^{i-1}(E_{1/2}))} \calA_{1/2,0} \quad \text{and} \quad \calA_{1/2,-i} = \chi_{X \setminus (T^{-1}(E_{1/2}) \cup \cdots \cup T^{-i}(E_{1/2}))} \calA_{1/2,0}$$
for $i > 0$. Following \cite[Subsection 4.2]{AC2}, we consider a skew partial power series ring $\calA_{1/2,0}[[t;T]]$ by taking infinite formal sums
$$\sum_{i \geq 0} b_i (\chi_{X \backslash E_{1/2}} t)^i = \sum_{i \geq 0} b_i t^i , \text{ where } b_i \in \calA_{1/2,i} \text{ for all } i \geq 0.$$
Specializing \cite[Definition 4.17]{AC2} to our situation, we see that a {\it special term} of positive degree $i$ is a monomial $b_it^i$ in $\calA_{1/2,i}$, where $b_i = \chi_S$ with
$$S = [10^{i_1}10^{i_2}1\cdots 10^{i_r}\ul{1}],\quad \text{ where } r\ge 1, i_1,\dots ,i_r\ge 1  \text{ and } i= i_1+\cdots +i_r +r.$$  
Moreover, since $E_{1/2}\cap T^{-1}(E_{1/2})\ne \emptyset$ and $S_0= T^{-1}(S_1)= [\ul{1}]$, we get from \cite[Definition 4.7(c)]{AC2} that the unique special term of degree $0$ is $\chi_{[\ul{1}]}$. 

Next, consider the subset $\calS_{1/2}[[t;T]]$ of $\calA_{1/2,0}[[t;T]]$ consisting of those elements
$$\sum_{i \geq 0} b_i (\chi_{X \backslash E_{1/2}}t)^i = \sum_{i \geq 0} b_i t^i$$
such that each $b_it^i \in \calA_{1/2,i}t^i$ belongs to the $K$-linear span of the special terms of degree $i$. We denote by $\calS_{1/2} [t;T]$ the subspace of $\calS_{1/2} [[t;T]]$ consisting of elements with finite support.

Let $\pazX=\{x_1,x_2,\dots  \}$ be an infinite countable set, and consider the algebras $K\langle \pazX \rangle $  and $K\langle \langle \pazX \rangle \rangle $ of non-commutative polynomials and non-commutative formal power series, respectively, with coefficients in $K$. We consider the degree in $K\langle \pazX \rangle $ given by $d(x_i) = i+1$ for all $i\ge 1$ and $d(1)=0$. With this degree function, $K\langle \pazX \rangle $ is a graded $K$-algebra, where $K$ has the trivial degree.

The following result is shown in \cite[Lemma 6.5 and Proposition 6.6]{AC2}. Note that \cite[Propositions 6.6 and 6.7]{AC2} are stated only for $n\ge 1$, but they hold, with the same proof, for $n=1/2$.
\begin{proposition}\label{prop:isospecial}
The linear subspaces $\calS_{1/2}[t;T]$ and $\calS_{1/2}[[t;T]]$ are indeed subalgebras of  $\calA_{1/2}[[t;T]]$. There is an isomorphism of unital graded algebras $K\langle \pazX \rangle \cong \calS_{1/2} [t; T]$ sending the element $x_{i_r}\cdots x_{i_1}$ to $\chi_{[10^{i_1}10^{i_2}1\cdots 10^{i_r}\ul{1}]}t^{i}$, where $i=i_1+\cdots + i_r+r$, and $1$ to $\chi_{[\ul{1}]}$. This isomorphism extends naturally to an isomorphism between $K\langle \langle \pazX \rangle \rangle $ and $\calS_{1/2}[[t;T]]$.
\end{proposition}
The special terms of the form $\chi_{[10^i\ul{1}]}t^{i+1}$, for $i\ge 1$, which are the algebra generators of $\calS_{1/2}[t;T]$, are called {\it pure special terms}, see the proof of \cite[Proposition 6.6]{AC2}.

Let $\pazX^*$ denote the free monoid on $\pazX$. We now recall the concept of the {\it Hadamard product} $\odot$ on $K\langle \langle \pazX \rangle \rangle $, see \cite{BR} for more details. 
If $a= \sum_{w\in \pazX^*}a_ww$ and $b= \sum_{w\in \pazX^*} b_ww$ are elements in $K\langle \langle \pazX \rangle \rangle $, set
$$a\odot b =  \Big(\sum_{w\in \pazX^*}a_ww\Big)\odot \Big(\sum_{w\in \pazX^*} b_ww\Big) = \sum_{w\in \pazX^*} (a_wb_w)w \in K\langle \langle \pazX \rangle \rangle .$$
We will denote by $K\langle \langle \pazX \rangle \rangle ^{\circ}$ the algebra of non-commutative formal power series endowed with the Hadamard product.
Note that $K\langle \langle \pazX \rangle \rangle ^{\circ}$ is a $*$-regular ring with respect to the involution given by $\ol{\sum_{w\in \pazX^*}a_ww}= \sum _{w\in \pazX^*} \ol{a_w}w$. 
The algebra of {\it non-commutative rational series}, denoted by $K_{\text{rat}}\langle \pazX \rangle $ is defined as the division closure of $K\langle \pazX \rangle $ in $K\langle \langle \pazX \rangle \rangle $, that is, $K_{\text{rat}}\langle \pazX \rangle $ is the smallest subalgebra of $K\langle \langle \pazX \rangle \rangle$ containing $K\langle \pazX \rangle $ and closed under inversion, see \cite{BR} for details. It is well-known that $K_{\text{rat}}\langle \pazX \rangle $ is closed under the Hadamard product \cite[Theorems 1.5.5 and 1.7.1]{BR}.

Let $p_{1/2} :=p_{E_{1/2}}= \pi_{1/2}(\chi_{E_{1/2}})$ be the projection in $\gotR_{1/2}= \prod_{W\in \mathbb V_{1/2}} M_{|W|}(K) = K\times \prod_{k \geq 1} M_{k+2}(K)^{\text{Fib}_2(k)}$ which has a $1$ in the left upper corner of each matrix factor $M_{|W|}(K)$, and $0's$ elsewhere, that is,
$(p_{1/2})_W= e_{00}(W)$ for each $W\in \mathbb V_{1/2}$. Recall from the beginning of Subsection \ref{subsection-computations.Betti.no.lamplighter} that the elements of $\mathbb V_{1/2}$ are all the clopen sets of the form $[1\ul{1}0^{i_1}10^{i_2}1\cdots 10^{i_r}11]$, where $r\ge 0$ and $i_1,\dots i_r\ge 1$ (we interpret this term to be $[1\ul{1}1]$ if $r=0$).

By \cite[Proposition 4.20]{AC2}, there is a $*$-isomorphism from $(\calS_{1/2}[[t;T]],\odot,-)$ onto $p_{1/2}\gotR_{1/2} p_{1/2}$. Composing the isomorphism $K\langle \langle \pazX \rangle \rangle \cong \calS_{1/2} [[t;T]]$ from Proposition \ref{prop:isospecial} (which respects the respective Hadamard products and involutions) with this isomorphism we obtain a $*$-isomorphism
$$\Delta \colon K\langle \langle \pazX \rangle \rangle^{\circ}  \overset{\cong}{\longrightarrow} p_{1/2}\gotR_{1/2} p_{1/2}$$
such that
$$\Delta (f)_{[1\ul{1}0^{i_1}1\cdots 10^{i_r}11]} = (f, x_{i_r}\cdots x_{i_1})e_{00}([1\ul{1}0^{i_1}1\cdots 10^{i_r}11]) \qquad (f\in K\langle \langle \pazX \rangle \rangle )$$
for each $i_1,\dots ,i_r\ge 1$. Here, $(f, x_{i_r}\cdots x_{i_1})$ is the coefficient of $f$ corresponding to the monomial $x_{i_r}\cdots x_{i_1}$.

For each $n\ge 1$, set $\pazX_n=\{ x_1,\dots , x_n\}$. Let $L$ be a {\it language} over $\pazX_n$, i.e., a subset of the free monoid $\pazX_n^*$ on $\pazX_n$. We now give a formula for the rank of the element $\Delta(\ul{L})$, where
$$\ul{L} := \sum_{w\in L} w \in K\langle \langle \pazX_n \rangle \rangle $$
is the {\it characteristic series} of $L$ (\cite[p.8]{BR}). Note that, for $W=[1\ul{1}0^{i_1}10^{i_2}1\cdots 10^{i_r}11]\in \mathbb V_{1/2}$, we have
$$\Delta (\ul{L})_W= \begin{cases} e_{00}(W) & \text{ if } x_{i_r}\cdots x_{i_2}x_{i_1}\in L \\
0  & \text{ if } x_{i_r}\cdots x_{i_2}x_{i_1}\notin L\end{cases}.$$
Moreover, we have $\mu (W)= 2^{-(i_1+\cdots +i_r+r+3)}= 2^{-3}2^{-d(w)}$, where $w=x_{i_r}\cdots x_{i_2}x_{i_1}$. Let $s(L)= \sum_{j\ge 0} \alpha_j x^j$ be the series in $\Z^+[[x]]$ such that $\alpha_j$ is the number of words $w\in L$ such that $d(w)= j$. Then we have
\begin{equation}\label{equation-formula.for.rk.rat.series}
\rk(\Delta(\ul{L})) = \sum_{W \in \V_{1/2}} \mu(W) \Rk(\Delta(\ul{L})_W) =   \frac{1}{2^3}\sum_{w\in L} \frac{1}{2^{d(w)}} = \frac{1}{8}s(L)\Big(\frac{1}{2}\Big).
\end{equation}
Here we make use of \eqref{equation-rank.over.Rn} for the first equality, of the above observations for the second, and of the definition of $s(L)$ for the third. We say that $s(L) \in \Z^+[[x]]$ is the {\it generating function} of $L$ with respect to $d$, and we denote by $\alpha_L$ the real number $\rk (\Delta (\ul{L}))$ described in \eqref{equation-formula.for.rk.rat.series}.

We now recall the notion of a rational language (see \cite[Chapter 3]{BR}).

\begin{definition}\label{definition-rational.language}
The set of \textit{rational languages} over a finite set $F$ is the smallest set of subsets of the free monoid $F^*$ containing all the finite subsets of $F^*$ and closed under the following operations:
\begin{enumerate}[a)]
	\item union $L_1\cup L_2$;
	\item product $L_1L_2 := \{w_1 w_2 \in F^* \mid w_1 \in L_1, w_2 \in L_2\}$;
	\item $*$-product $L^* := \bigcup_{k \geq 0} L^k$.
\end{enumerate}
\end{definition}

A characterization of the rational languages is the following (\cite[Lemma 3.1.4]{BR}): a language over a finite alphabet $F$ is rational if and only if it is the support of some rational series $z \in \Z^+ \langle \langle F \rangle \rangle$. In fact, if $L$ is a rational language over $F$, its characteristic series $\ul{L} = \sum_{w \in L}w$ belongs to $R_{\mathrm{rat}}\langle F \rangle^{\circ}$ for any semiring $R$  (\cite[Proposition 3.2.1]{BR}).

It turns out that the set of rational languages over $F$, which we will denote by $\gotL(F)$, forms a Boolean subalgebra of $P (F^*)$, the power set of $F^*$ (\cite[Corollary 3.1.5]{BR}). However, the converse may not be true. Indeed, for a subfield $K$ of $\C$ closed under complex conjugation, given a rational series $z = \sum_{w \in F^*} \lambda_w w \in K_{\mathrm{rat}}\langle F \rangle$ it may be the case that its support $L = \text{supp}(z)$ is not a rational language. We put
$$\mathfrak K(F) := \{ \text{supp}(z) \mid z \in K_{\mathrm{rat}}\langle F \rangle^{\circ}\} \subseteq P(F^*),$$
and we let $\mathbf{B}_{\text{rat}}^K(F)$ be the Boolean algebra of subsets of $F^*$ generated by $\mathfrak K (F)$. Then by \cite[Proposition 6.10]{AC2} the $*$-regular closure of $K_{\text{rat}}\langle F\rangle ^{\circ}$ in $K\langle \langle F\rangle \rangle^{\circ}$ is contained in the set of formal power series whose support belongs to $\mathbf{B}_{\mathrm{rat}}^K(F)$, and for each set $L\in \mathbf{B}_{\mathrm{rat}}^K(F)$, the characteristic series $\ul{L}$ belongs to the $*$-regular closure.

We gather some of these facts in the following.  

\begin{proposition}\label{proposition-numbers.from.languages}
Let $n\ge 1$ and let $\pazX_n=\{x_1,\dots , x_n\}$, endowed with the degree function $d(x_i)= i+1$. With the above notation, we have $\gotL(\pazX_n) \subseteq \mathfrak K (\pazX_n) \subseteq \textbf{\emph{B}}_{\mathrm{rat}}^K(\pazX_n)$. The sets $\gotL(\pazX_n)$ and $\textbf{\emph{B}}_{\mathrm{rat}}^K(\pazX_n)$ are Boolean subalgebras of $P(\pazX_n^*)$, but this is not the case in general for $\mathfrak K (\pazX_n)$. Moreover, if $L\in \mathbf{B}_{\mathrm{rat}}^K(\pazX_n)$, then $\alpha_L\in \calG (\Gamma, K)$.
\end{proposition}
\begin{proof}
See \cite[Chapter 3]{BR}. The  fact that $\alpha_L\in \calG (\Gamma, K)$ for  $L\in \mathbf{B}_{\mathrm{rat}}^K(\pazX_n)$ follows from \cite[Propositions 4.1, 4.26, 6.7 and 6.10]{AC2}.
\end{proof}

We now show that rational languages give rise to rational $\ell^2$-Betti numbers.

\begin{lemma}\label{lemma-generating.function.rational.language}
If a series $\sum_{j \geq 0} \alpha_j x^j \in \Z^+[[x]]$ is the generating function of a rational language $L \in \gotL(\pazX_n)$ with respect to $d$, then it is a rational series with constant term either $0$ or $1$. Consequently, $\alpha_L \in \Q$.
\end{lemma}
\begin{proof}
Suppose $\alpha_j = \#(L \cap d^{-1}(\{j\}))$, being $L \in \gotL(\pazX_n)$ a rational language. In particular the series $\ul{L}= \sum_{w \in L}w$ is rational. Define a map $\rho : \pazX_n \ra \Z^+[[x]]$ by setting $\rho(b) = x^{d(b)}$. By \cite[Proposition 1.4.2]{BR}, the map $\rho$ extends uniquely to a morphism $\rho : \Z^+\langle \langle \pazX_n \rangle \rangle \ra \Z^+[[x]]$ which induces the identity on $\Z^+$ and, moreover, preserves rationality. Then
$$\rho\Big( \sum_{w \in L}w \Big) = \sum_{w \in L}x^{d(w)} = \sum_{j \geq 0} \alpha_j x^j$$
is rational in $\Z^+[[x]]$, and clearly its constant term is either $0$ or $1$.

By \cite[Proposition 6.1.1]{BR}, there exist polynomials $p(x),q(x) \in \Z[x]$ such that the series $\sum_{j \geq 0}\alpha_j x^j$ is the power series expansion of the rational function $\frac{p(x)}{1-xq(x)}$, that is,
$$s(L) = \sum_{j \geq 0}\alpha_j x^j = \frac{p(x)}{1-xq(x)} = p(x) \Big[1 + \sum_{j \geq 0}(xq(x))^j\Big].$$
Consequently, if $L \in \gotL(\pazX_n)$,  then $\alpha_L = \frac{1}{8}s(L)\big(\frac{1}{2}\big) \in \Q$.
\end{proof}

We obtain now a concrete irrational algebraic number of the form $\alpha_L$ for a suitable $L\in \mathbf{B}_{\text{rat}}^{\Q}(\pazX_2)$.

\begin{theorem}\label{thm:irrat-algl2betti}
There is $L\in \mathbf{B}_{{\rm rat}}^{\Q}(\pazX_2)$ such that $\alpha_L=\frac{1}{4}\sqrt{\frac{2}{7}}$. In particular, $\frac{1}{4}\sqrt{\frac{2}{7}}\in \calG(\Gamma, \mathbb Q)$.
\end{theorem}

\begin{proof}
Let $L$ be the language on $\pazX_2=\{x_1,x_2\}$ given by
$$L= \{ w\in \pazX_2^* \mid |w|_{x_1} =|w|_{x_2} \}.$$
Here $|w|_{x_i}$ is the number of appearances of $x_i$ in $w$, for $i=1,2$. 
Then by \cite[Example 3.4.1]{BR}, $L\in \mathbf{B}_{\text{rat}}^{\Q}(\pazX_2)$ (although $L$ is {\it not} a rational language). If $w\in L$ then, since $d(x_1)= 2$ and $d(x_2)=3$, we get that $d(w)=5l$, where $l=|w|_{x_1}=|w|_{x_2}$. Therefore we get from Proposition \ref{proposition-numbers.from.languages} that
$$\alpha_L = \frac{1}{8} \sum_{j \geq 0} \frac{\alpha_j}{2^j} = \frac{1}{8} \sum_{l\ge 0} \binom{2l}{l} \frac{1}{2^{5l}} \in \calG (\Gamma, \Q).$$
Since
$$\sum_{l\ge 0} \binom{2l}{l} x^l = \frac{1}{\sqrt{1-4x}} \quad \text{ for } |x| < \frac{1}{4},$$
we obtain $\alpha _L = \frac{1}{4}\sqrt{\frac{2}{7}} \in \calG(\Gamma, \Q)$.	
\end{proof}

Taking different choices of variables $x_r,x_s$, with $r\ne s$, we indeed obtain that $\calG (\Gamma, \Q)$ contains all the algebraic numbers
$$\frac{1}{8}\sqrt{\frac{2^t}{2^t-1}}, \quad \text{ for } t \geq 3.$$

Similar formulas can be obtained to compute the ranks of elements coming from the copy of the algebra $\Q_{\text{rat}}\langle \pazX\rangle$ in the $*$-regular closure $\calR_n$ of $\calA_n$ in $\mathfrak R_n$, where $\calA_n$ are the approximations of $\Q \Gamma$ described in Subsection \ref{subsection-approximating.algebras.lamplighter}. The only differences are  that we have different degree functions on $\pazX$, depending on $n$, and that the factor $1/8$ must be substituted by the factor $2^{-2n-2}$.

\begin{remark}
By Theorem \ref{thm:irrat-algl2betti}, $\frac{1}{4}\sqrt{\frac{2}{7}}$ is a ``virtual'' $\ell^2$-Betti number arising from the lamplighter group $\Gamma $. This means that there are two $\ell^2$-Betti numbers arising from $\Gamma $, $\alpha_1$ and $\alpha_2$, such that $\frac{1}{4}\sqrt{\frac{2}{7}} = \alpha_1 - \alpha_2$. It would be interesting to know whether $\frac{1}{4}\sqrt{\frac{2}{7}}$ is itself an $\ell^2$-Betti number, and in that case, to determine a concrete matrix over $\mathbb Q \Gamma$ witnessing this fact. In particular, this would solve a question posed by Grabowski in \cite[p. 32]{Gra14}, see also \cite[Question 3]{Gra16}. We do not know any example of a group $G$ for which $\calC (G,\mathbb Q) \ne \calG (G,\mathbb Q)\cap \mathbb R^+$. 
\end{remark}

We close this section with the following problem:
\begin{question}\label{quest:algebraics}
With the notation established before, is $\alpha_L$ always an algebraic number when $L \in \mathbf B_{{\rm rat}}^{\Q}(\pazX_n)$?
\end{question}

\section{The odometer algebra}\label{section-odometer.alg}

In this last section we focus on computing the whole set of values that the Sylvester matrix rank function constructed from Theorem \ref{theorem-rank.function} can achieve in the case of the generalized odometer algebra. We first recall its definition.

Fix a sequence of natural numbers $\ol{n} = (n_i)_{i \in \N}$ with $n_i \geq 2$ for all $i \in \N$, and consider $X_i$ to be the finite space $\{0,1,...,n_i-1\}$ endowed with the discrete topology. From these we form the topological space $X = \prod_{i \in \N}X_i$ endowed with the product topology, which is in fact a Cantor space. Let $T$ be the homeomorphism on $X$ given by the odometer, namely for $x = (x_i) \in X$, $T$ is given by
$$T(x) = \begin{cases}
					(x_1+1,x_2,x_3,...) & \text{ if } x_1 \neq n_1-1, \\
					(0,...,0,x_{m+1} + 1,x_{m+2},...) & \text{ if } x_j = n_j-1 \text{ for $1 \leq j \leq m$ and } x_{m+1} \neq n_m-1, \\
					(0,0,...) & \text{ if } x_i = n_i-1 \text{ for all }i \in \N.
				 \end{cases}$$
Note that the odometer action is just addition of $(1,0,...)$ by carry-over.

Let $(K,-)$ be any field with a positive definite involution $-$. The \textit{generalized odometer algebra} is defined as the crossed product $*$-algebra $\calO(\ol{n}) := C_K(X) \rtimes_T \Z$. We obtain a measure $\mu$ on $X$ by taking the usual product measure, where we consider the measure on each component $X_i$ which assigns mass $\frac{1}{n_i}$ on each point in $X_i$. It is well-known (e.g. \cite[Section VIII.4]{Dav}) that $\mu$ is an ergodic, full and $T$-invariant probability measure on $X$, which in turn coincides with the Haar measure $\wh{\mu}$ on $X$ if one considers $X$ as an abelian group with addition by carry-over. We denote by $\rk_{\calO(\ol{n})}$ the Sylvester matrix rank function on $\calO(\ol{n})$ given by Theorem \ref{theorem-rank.function}.

\subsection{Characterizing the \texorpdfstring{$*$}{}-regular closure \texorpdfstring{$\calR_{\calO(\ol{n})}$}{}}\label{subsection-odometer.reg.closure}

We explicitly compute the $*$-regular closure $\calR_{\calO(\ol{n})}$ of the odometer algebra. Recall that it is defined as the $*$-regular closure of $\calO(\ol{n})$ inside $\gotR_{\rk}$ through the embedding $\calO(\ol{n}) \hookrightarrow \gotR_{\rk}$ given by Theorem \ref{theorem-rank.function}. We will use the notation $[a_1a_2 \dots a_r]$ for the cylinder sets $\{ (x_j)_{j\in \N}\mid x_i=a_i, 1 \leq i \leq r\}$, where $0\le a_i < n_i$, $1 \leq i \leq r$.

Define new integers $p_m := n_1 \cdots n_m$ for $m \in \N$. At each level $m \geq 1$, we take $E_m = [0 0 \cdots 0]$ (with $m$ zeroes) for the sequence of clopen sets, whose intersection gives the point $y = (0,0,...) \in X$. We take the partition $\calP_m$ of the complement $X \backslash E_m$ to be the obvious one, namely
$$\calP_m = \{[{1} 0 \cdots 0], ..., [{(n_1-1)} 0 \cdots 0],..., [{(n_1-1)} (n_2-1) \cdots (n_m-1)]\}.$$
We denote by $\calO(\ol{n})_m$ the unital $*$-subalgebra of $\calO(\ol{n})$ generated by the partial isometries $\{\chi_Z t \mid Z \in \calP_m\}$. The quasi-partition $\ol{\calP}_m$ is really simple in this case: write $Z_{m,l} =  T^l(E_m)$ for $1 \le l < p_m$. Note that these clopen sets form exactly the partition $\calP_m$, and that $T(Z_{m,p_m-1}) = E_m$. Therefore there is only one possible $W \in \V_m$, which has length $p_m$ and is given by
$$W = E_m \cap T^{-1}(Z_{m,1}) \cap T^{-2}(Z_{m,2}) \cap \cdots \cap T^{-p_m+1}(Z_{m,p_m-1}) \cap T^{-p_m}(E_m) = E_m.$$
%Clearly
%$$\sum_{W \in \V_m} |W| \mu(W) = p_m \mu(E_m) = 1,$$
%as we already know.

The representations $\pi_m : \calO(\ol{n})_m \to \gotR_m, x \mapsto (h_W \cdot x)_W$ become $*$-isomorphisms, where $\gotR_m = M_{p_m}(K)$. %Indeed, we have $e_{00}(E_m) = \chi_{E_m}$ and $e_{ii}(E_m) = \chi_{T^i(E_m)} = \chi_{Z_{m,i-1}}$ for $1 \leq i \leq p_m-1$, and so $h_{E_m} = \sum_{i=0}^{p_m-1} e_{ii}(E_m) = \chi_X = 1$.
Under this identification, the embeddings $\iota_m : \calO(\ol{n})_m \hookrightarrow \calO(\ol{n})_{m+1}$ become the block-diagonal embeddings
$$M_{p_m}(K) \hookrightarrow M_{p_{m+1}}(K), \quad x \mapsto \text{diag}(x,\stackrel{n_{m+1}}{\dots},x),$$
%$$M_{p_m}(K) \ra M_{p_{m+1}}(K), \quad x \mapsto \begin{pmatrix}
%																									x & & \bigzero \\
%																									& \ddots \stackinset{c}{-.1in}{c}{.15in}{\footnotesize$n_{m+1}$}{} & \\
%																									\bigzero & & x
%																								 \end{pmatrix}$$
so that $\calO(\ol{n})_{\infty} := \varinjlim_m \calO(\ol{n})_m \cong \varinjlim_m M_{p_m}(K)$. Note that $\calO(\ol{n})_{\infty}$ is already $*$-regular, but it does not contain $\calO(\ol{n})$: in fact, it is contained in $\calO(\ol{n})$. Intuitively, in order to get containment of the whole algebra $\calO(\ol{n})$ we need to adjoin to $\calO(\ol{n})_{\infty}$ the element $t$.

\begin{definition}\label{definition-An.t}
For every $m \geq  1$, we denote by $\calO(\ol{n})^t_m$ the unital $*$-subalgebra of $\calO(\ol{n})$ generated by $\calO(\ol{n})_m$ and $t$.
\end{definition}
We can completely characterize these $*$-subalgebras.
\begin{lemma}\label{lemma-charac.An.t}
There exists a $*$-isomorphism $\calO(\ol{n})^t_m \cong M_{p_m}(K[t^{p_m}, t^{-p_m}])$.
\end{lemma}
\begin{proof}
For $0 \leq i,j < p_m$, the elements $e_{ij}^{(m)} := e_{ij}(E_m)$ form a complete system of matrix units inside $\calO(\ol{n})^t_m$, so there is an isomorphism $\calO(\ol{n})^t_m \cong M_{p_m}(T)$ with $T$ being the centralizer of the family $\{e_{ij}^{(m)} \mid 0 \leq i,j < p_m\}$ in $\calO(\ol{n})^t_m$. The isomorphism is given explicitly by
$$a \mapsto \sum_{i,j=0}^{p_m-1} a_{ij} e_{ij}^{(m)}, \quad \text{ with } a_{ij} = \sum_{k=0}^{p_m-1} e_{ki}^{(m)} \cdot a \cdot e_{jk}^{(m)} \in T,$$
which is also a $*$-isomorphism. Since $t^{p_m} e_{ij}^{(m)} t^{-p_m} = e_{ij}^{(m)}$ and %$t = \sum_{i=0}^{p_m-2} e_{i+1,i}^{(m)} + t^{p_m} e_{0,p_m-1}^{(m)} \in M_{p_m}(K[t^{p_m},t^{-p_m}])$ 
%$$t^{p_m} e_{ij}^{(m)} t^{-p_m} = t^{p_m} t^i \chi_{E_m} t^{-j} t^{-p_m} = t^i \chi_{T^{p_m}(E_m)} t^{-j} = t^i \chi_{E_m} t^{-j} = e_{ij}^{(m)}$$
%and
$$t = \sum_{i=0}^{p_m-1} t e_{ii}^{(m)} = \sum_{i=0}^{p_m-2} e_{i+1,i}^{(m)} + t^{p_m} e_{0,p_m-1}^{(m)} \in M_{p_m}(K[t^{p_m},t^{-p_m}]),$$
we deduce that $T = K[t^{p_m},t^{-p_m}]$, as desired.
\end{proof}
The inclusion $\calO(\ol{n})^t_m \subseteq \calO(\ol{n})^t_{m+1}$ translates to an embedding from $M_{p_m}(K[t^{p_m},t^{-p_m}])$ to \linebreak
 $M_{p_{m+1}}(K[t^{p_{m+1}},t^{-p_{m+1}}])$ that extends the previous one $M_{p_m}(K) \hookrightarrow M_{p_{m+1}}(K)$, and sends the element $t^{p_m} \cdot \text{Id}_{p^m}$ to the element
%e_{ij} \mapsto \sum_{k=0}^{n_{m+1}-1} e_{i+kp_m, j+kp_m}, \quad t^{p_m} \text{Id}_{p_m} \mapsto 
$$\left( \begin{array}{@{}c@{}}
					\begin{matrix}
						\bigzero_{p_m} & & & \bigzero_{p_m} & t^{p_{m+1}} \cdot \text{Id}_{p_m} \\
						\text{Id}_{p_m}	& \bigzero_{p_m} & & & \bigzero_{p_m} \\
						 & \ddots & \ddots \stackinset{c}{-.1in}{c}{.15in}{\footnotesize$n_{m+1}$}{} & & \\
						 & & \ddots & \bigzero_{p_m} & \\
						\hspace*{0.2cm} \bigzero_{p_m} & & & \text{Id}_{p_m} & \bigzero_{p_m}
					\end{matrix}											
				 \end{array}\right).$$% \in M_{p_{m+1}}(K[t^{p_{m+1}},t^{-p_{m+1}}]).$$

\begin{lemma}\label{lemma-charac.An.t.2}
$\calO(\ol{n})$ is $*$-isomorphic to the direct limit $\varinjlim_m M_{p_m}(K[t^{p_m},t^{-p_m}])$ with respect to the previous embeddings.
\end{lemma}
\begin{proof}
By Remark \ref{remark-algebra.Ainfty} the algebra $\calO(\ol{n})_{\infty}$ contains $C_K(X)$, hence the $*$-subalgebra of $\calO(\ol{n})$ generated by $t$ and $\calO(\ol{n})_{\infty}$ is $\calO(\ol{n})$ itself. Now each $\calO(\ol{n})_m$ sits inside $\calO(\ol{n})^t_m$, hence $\calO(\ol{n})_{\infty} = \varinjlim_m \calO(\ol{n})_m \subseteq \varinjlim_m \calO(\ol{n})^t_m$. But $t \in \varinjlim_m \calO(\ol{n})^t_m$ already, so $\calO(\ol{n}) = \varinjlim_m \calO(\ol{n})^t_m \cong \varinjlim_m M_{p_m}(K[t^{p_m},t^{-p_m}])$ by Lemma \ref{lemma-charac.An.t}.
\end{proof}

We are now ready to compute $\calR_{\calO(\ol{n})}$.

\begin{theorem}\label{theorem-reg.closure.odometer}
There is a $*$-isomorphism $\calR_{\calO(\ol{n})} \cong \varinjlim_m M_{p_m}(K(t^{p_m}))$, where the maps $M_{p_m}(K(t^{p_m})) \hookrightarrow M_{p_{m+1}}(K(t^{p_{m+1}}))$ are induced by the inclusions of matrices over Laurent polynomial rings.
\end{theorem}
\begin{proof}
We have embeddings $\calO(\ol{n})^t_m \hookrightarrow \calO(\ol{n}) \hookrightarrow \calR_{\calO(\ol{n})} \hookrightarrow \gotR_{\rk}$. By \cite[Lemma 4.4]{AC2}, the rational function field $K(t^{p_m})$ sits inside $\calR_{\calO(\ol{n})}$. Hence there is, for each $m \geq 1$, a commutative diagram
\begin{equation*}
\xymatrix{
	\calO(\ol{n})^t_m \ar@{^{(}->}[r] \ar@{^{(}->}[d] & M_{p_m}(K(t^{p_m})) \ar@{^{(}->}[r] & \calR_{\calO(\ol{n})} \ar@{}[d]|*=0[@]{=} \\
	\calO(\ol{n})^t_{m+1} \ar@{^{(}->}[r] & M_{p_{m+1}}(K(t^{p_{m+1}})) \ar@{^{(}->}[r] & \calR_{\calO(\ol{n})}
}\label{diagram-comm.diag.5}
\end{equation*}
It is straightforward to prove that for any non-zero element $q(t) \in K[t^{p_m},t^{-p_m}]$, the corresponding matrix in $M_{p_{m+1}}(K[t^{p_{m+1}}, t^{-p_{m+1}}])$ becomes invertible in $M_{p_{m+1}}(K(t^{p_{m+1}}))$. Therefore the embedding $\calO(\ol{n})^t_m \hookrightarrow \calO(\ol{n})^t_{m+1}$ extends uniquely to an embedding $M_{p_m}(K(t^{p_m})) \hookrightarrow M_{p_{m+1}}(K(t^{p_{m+1}}))$, and 
the previous commutative diagrams give embeddings
%extend to commutative diagrams
\begin{equation*}\label{equation-comm.diag.6}
%\xymatrix{
%	\calO(\ol{n})_m(t) \ar@{^{(}->}[r] \ar@{^{(}->}[d] & M_{p_m}(K(t^{p_m})) \ar@{^{(}->}[r] \ar@{^{(}->}[d] & \calR_{\calO(\ol{n})} \ar@{}[d]|*=0[@]{=} \\
%	\calO(\ol{n})_{m+1}(t) \ar@{^{(}->}[r] \ar@{^{(}->}[d] & M_{p_{m+1}}(K(t^{p_{m+1}})) \ar@{^{(}->}[r] \ar@{^{(}->}[d] & \calR_{\calO(\ol{n})} \ar@{}[d]|*=0[@]{=} \\
%	\vdots \ar@{^{(}->}[d] & \vdots \ar@{^{(}->}[d] & \vdots \ar@{}[d]|*=0[@]{=} \\
%	\calA \ar@{^{(}->}[r] & \varinjlim_n M_{p_m}(K(t^{p_m})) \ar@{^{(}->}[r] & \calR_{\calO(\ol{n})}
%}\label{diagram-comm.diag.5}
\calO(\ol{n}) \cong \varinjlim_m \calO(\ol{n})^t_m \hookrightarrow \varinjlim_m M_{p_m}(K(t^{p_m})) \hookrightarrow \calR_{\calO(\ol{n})}.
\end{equation*}
But $\varinjlim_m M_{p_m}(K(t^{p_m}))$ is already $*$-regular and contains $\calO(\ol{n})$, so necessarily $\calR_{\calO(\ol{n})} \cong \varinjlim_m M_{p_m}(K(t^{p_m}))$, as required.
\end{proof}

In particular, since $\varinjlim_m M_{p_m}(K(t^{p_m}))$ has a unique Sylvester matrix rank function, it coincides with the one given by Theorem \ref{theorem-rank.function} under the previous $*$-isomorphism.

\subsection{Characterizing the set \texorpdfstring{$\calC'(\calO(\ol{n}))$}{}}\label{subsection-odometer.alg.betti.no}

We characterize the set $\calC'(\calO(\ol{n}))$ of all positive real values that the Sylvester matrix rank function $\rk_{\calO(\ol{n})}$ can achieve, and the subgroup $\calG(\calO(\ol{n}))$ it generates. First, let
$$\calC'(\calR_{\calO(\ol{n})}) := \rk_{\calR_{\calO(\ol{n})}}\Big( \bigcup_{k \geq 1} M_k(\calR_{\calO(\ol{n})}) \Big) \subseteq \R^+$$
and note that $\calC'(\calO(\ol{n})) \subseteq \calC'(\calR_{\calO(\ol{n})})= \calC(\calR_{\calO(\ol{n})}) $ (see Definition \ref{def:semigroupsCandCprime} and Remark \ref{rem:semigroupsCandCprime}). The following definition will be essential.

\begin{definition}\label{definition-supernatural.no}
For each sequence $\ol{n} = (n_1,n_2,...)$ of positive integers $n_i \geq 2$, one may associate to it the \textit{supernatural number}
$$n = \prod_{i \in \N} n_i = \prod_{q \in \P}q^{\varepsilon_q(n)},$$
where $\P$ is the set of prime numbers ordered with respect to the natural ordering, and each $\varepsilon_q(n) \in \{0\} \cup \N \cup \{\infty\}$. In more detail, if $n_i = \prod_{q \in \P}q^{\varepsilon_q(n_i)}$ is the prime decomposition of $n_i$, then, for each $q\in \P$,  $\varepsilon_q (n)$ is defined by
$\varepsilon_q (n) = \sum_{i=1}^{\infty} \varepsilon _q (n_i) \in  \{0\} \cup \N \cup \{\infty\}$, and $n$ is defined as the formal product $\prod_{q \in \P}q^{\varepsilon_q(n)}$.  
\end{definition}

As in \cite[Definition 7.4.2]{LLR}, from any supernatural number $n$ one can construct an additive subgroup of $\Q$ containing $1$, denoted by $\Z(n)$, consisting of those fractions $\frac{a}{b}$ with $a \in \Z$, and $b \in \Z \backslash \{0\}$ being of the form
$$b = \prod_{q \in \P} q^{\varepsilon_q(b)},$$
where $\varepsilon_q(b) \leq \varepsilon_q(n)$ for all $q \in \P$, and $\varepsilon_q(b) = 0$ for all but finitely many $q$'s. If $n$ comes from a sequence $\ol{n} = (n_1,n_2,...)$ as above, $\Z(n)$ is exactly the additive subgroup of $\Q$ consisting of those fractions of the form
$$\frac{a}{n_1 \cdots n_r}, \quad \text{ with } a \in \Z \text{ and } r \geq 1.$$
Each group $\Z (n)$ gives rise to a subgroup $\mathbb G (n)=\Z(n)/\Z$  of the group $\Q/\Z$. Indeed the groups $\mathbb G (n)$ describe all the subgroups of $\Q/\Z$. These groups are called \textit{Pr\"ufer groups}.

\begin{theorem}\label{theorem-betti.numbers.odometer}
We have $\calC'(\calO(\ol{n})) = \calC(\calO(\ol{n})) = \calC'(\calR_{\calO(\ol{n})}) = \Z(n)^+$. In particular, $\calG(\calO(\ol{n})) = \Z(n)$.
\end{theorem}
\begin{proof}
The argument is similar to the one given in the proof of \cite[Proposition 4.1]{AC2}. Since $\calR_{\calO(\ol{n})}$ is a $*$-regular ring with positive definite involution, each matrix algebra $M_k(\calR_{\calO(\ol{n})})$ is also $*$-regular. Hence for each $A \in M_k(\calR_{\calO(\ol{n})})$ there exists a projection $P \in M_k(\calR_{\calO(\ol{n})})$ such that $\rk_{\calR_{\calO(\ol{n})}}(A) = \rk_{\calR_{\calO(\ol{n})}}(P)$. We conclude that $\calC'(\calR_{\calO(\ol{n})})$ equals the set of positive real numbers of the form $\rk_{\calR_{\calO(\ol{n})}}(P)$, where $P$ ranges over matrix projections with coefficients in $\calR_{\calO(\ol{n})}$. Now each such projection $P$ is equivalent to a diagonal projection \cite[Proposition 2.10]{Goo91}, that is of the form $\text{diag}(p_1,p_2,...,p_r)$ for some projections $p_1,...,p_r \in \calR_{\calO(\ol{n})}$, 
%\begin{equation*}
%\begin{pmatrix}
% \diagentry{p_1}\\
% &\diagentry{p_2}\\
% &&\diagentry{\xddots}\\
% &&&\diagentry{p_r}\\
%\end{pmatrix}\quad \text{ for some projections } p_1,...,p_r \in \calR_{\calO(\ol{n})},
%\end{equation*}
so that $\rk_{\calR_{\calO(\ol{n})}}(P) = \rk_{\calR_{\calO(\ol{n})}}(p_1) + \cdots + \rk_{\calR_{\calO(\ol{n})}}(p_r)$. But since $\calR_{\calO(\ol{n})} \cong \varinjlim_m M_{p_m}(K(t^{p_m}))$ by Theorem \ref{theorem-reg.closure.odometer}, the set of ranks of elements in $\calR_{\calO(\ol{n})}$ is contained in $\Z(n)^+ \cap [0,1]$. Therefore $\rk_{\calR_{\calO(\ol{n})}}(P) \in \Z(n)^+$. This proves the inclusion $\calC'(\calR_{\calO(\ol{n})}) \subseteq \Z(n)^+$. The inclusion $\Z(n)^+ \subseteq \calC'(\calO(\ol{n}))$ is straightforward, since for $0 \leq i < p_m$ we have $\rk_{\calO(\ol{n})}(e_{ii}^{(m)}) = \frac{1}{p_m}$. The equality 
$ \calC (\calO(\ol{n}))= \Z (n)^+$ also follows from the above.

The last part of the theorem is immediate.
\end{proof}

In a sense, Theorem \ref{theorem-betti.numbers.odometer} confirms the SAC for a class of crossed product algebras, as follows:

\begin{remark}\label{rem:fourierforX}
Let $n$ be a supernatural number and let $n= \prod_{i\in \N} n_i$ a decomposition of $n$, as above. Consider the presentation of $\mathbb G (n)$ given by generators $\{g_m \}_{m\in \N}$ and relations $g_1^{n_1}= 1$ and $g_{m+1}^{n_{m+1}} = g_m$ for all $m\ge 1$. The topological group $X(n)=\prod_{i\in \N} X_i$ is naturally isomorphic with the Pontryagin dual $\widehat{\mathbb G (n)}$ of the Pr\"ufer group $\mathbb G (n)$. Indeed, let $\xi_{p_m}$ denote the canonical primitive $p_m$-root of unity in $\C$, where $p_m= n_1\cdots n_m$, and observe that $\xi_{p_{m+1}}^{m_{n+1}}= \xi _{p_m}$ for each $m\in \N$. The map $\phi \colon X(n)\to \widehat{\mathbb G (n)}$ defined by
$$\phi_x (g_m)= \xi_{p_m}^{a_1+a_2n_1+\cdots + a_mn_{m-1}}, $$
where $x= (a_1,a_2,\dots )\in X(n)$, is a group isomorphism and a homeomorphism. Therefore if $K$ is a subfield of $\C$ closed under complex conjugation containing all $p_m^{\text{th}}$ roots of unity, Fourier transform defines an isomorphism
$$\mathcal F \colon K[\mathbb G(n)] \longrightarrow  C_K(X(n)) ,$$
and we can pull back the automorphism induced by $T$ on $C_K(X(n))$ to an automorphism $\rho$ on $K[\mathbb G (n)]$ (which is {\it not} induced by an automorphism of $\mathbb G (n)$). Note that $\rho (g_m)= \xi_{p_m} g_m\in K\cdot \mathbb G (n)$ for all $m\ge 1$. Therefore we can interpret Theorem \ref{theorem-betti.numbers.odometer} as giving a positive answer to the SAC for the crossed product $\calA = K[\mathbb G (n)]\rtimes_{\rho}\Z\cong \calO (n)$: the set of $\calA$-Betti numbers is exactly the semigroup $\Z(n)^+$ generated by the inverses of the orders of the elements of $\mathbb G (n)$. Note finally that the groups $X(n)$, $n$ a supernatural, give all the profinite completions of $\Z$. With this view, the dynamical system considered on $X(n)$ is generated by addition by $1$.
\end{remark}

\section*{Acknowledgments}

The authors would like to thank the anonymous referees for their very careful reading of the manuscript and for their many suggestions, which have improved the exposition of the paper.

\section*{Data availability statement}

Data sharing not applicable to this article as no datasets were generated or analysed during the current study.


\begin{thebibliography}{99}

%\bibitem{AaHR} Z. Afsar, A. an Huef, I. Raeburn, \emph{KMS states on $C^*$-algebras associated to a family of $*$-commuting local homeomorphisms},
%    J. Math. Anal. Appl. {\bf 464} (2018), 965--1009.

\bibitem{Ara87} P. Ara, \emph{Matrix rings over *-regular rings and pseudo-rank functions},
		Pacific J. Math. {\bf 129} (1987), 209--241.

%\bibitem{Ara12} P. Ara, \emph{Purely infinite simple reduced $C^*$-algebras of one-relator separated graphs},
%		J. Math. Anal. Appl. {\bf 393} (2012), 493--508.

%\bibitem{AB} P. Ara, M. Brustenga, \emph{The regular algebra of a quiver},
%		J. Algebra {\bf 309} (2007), 207--235.

\bibitem{AC} P. Ara, J. Claramunt, \emph{Sylvester matrix rank functions on crossed products},
		Ergodic Theory Dynam. Systems, \textbf{40} (2020), 2913--2946.

\bibitem{AC2} P. Ara, J. Claramunt, \emph{Approximating the group algebra of the lamplighter by infinite matrix products},
		arXiv:2005.12374v1 [math.RA].

%\bibitem{AC3} P. Ara, J. Claramunt, \emph{Uniqueness of the von Neumann continuous factor},
%		Canad. J. Math. {\bf 70} (2018), 961--982.

%\bibitem{AE} P. Ara, R. Exel, \emph{Dynamical systems associated to separated graphs, graph algebras, and paradoxical decompositions},
%		Adv. Math. {\bf 252} (2014), 748--804.

\bibitem{AG} P. Ara, K. R. Goodearl, \emph{The realization problem for some wild monoids and the Atiyah problem},
		Trans. Amer. Math. Soc. {\bf 369} (2017), 5665--5710.

%\bibitem{AGP} P. Ara, K. R. Goodearl, E. Pardo, \emph{$K_0$ of Purely Infinite Simple Regular Rings},
%		K-Theory, {\bf 26} (2002), 69--100.

%\bibitem{AL} P. Ara, M. Lolk, \emph{Convex subshifts, separated Bratteli diagrams, and ideal structure of tame separated graph algebras},
%		Adv. Math. {\bf 328} (2018), 367--435.

\bibitem{Ati} M. F. Atiyah, \emph{Elliptic operators, discrete groups and von Neumann algebras},
		in \emph{Colloque "Analyse et Topologie" en l'Honneur de Henri Cartan (Orsay, 1974)},
		Soc. Math. France, Paris, {\bf 32--33} (1976), 43--72.

\bibitem{AS} M. F. Atiyah, I. M. Singer, \emph{The index of elliptic operators on compact manifolds},
		Bull. Amer. Math. Soc. {\bf 69} (1963), 422--433.

\bibitem{Aus} T. Austin, \emph{Rational group ring elements with kernels having irrational dimension},
		Proc. London Math. Soc. {\bf 107} (2013), 1424--1448.

\bibitem{Berb} S. K. Berberian, \emph{Baer $*$-rings}.
		Die Grundlehren der mathematischen Wissenschaften, Band 195. Springer-Verlag, New York-Berlin, 1972.

\bibitem{BR} J. Berstel, C. Reutenauer, \emph{Noncommutative rational series with applications},
Encyclopedia of Mathematics and its Applications, 137. Cambridge University Press, Cambridge, 2011.


%\bibitem{BH} B. Blackadar, D. Handelman, \emph{Dimension functions and traces on C*-algebras},
%		J. Funct. Anal. {\bf 45} (1982), 297--340.

%\bibitem{Bra} O. Bratteli, \emph{Inductive Limits of Finite Dimensional $C^*$-Algebras},
%		Trans. Amer. Math. Soc. {\bf 171} (1972), 195--234.

%\bibitem{BO} N. P. Brown, N. Ozawa, \emph{$C^*$-Algebras and Finite-Dimensional Approximations}.
%		Graduate studies in mathematics, Volume 88. American Mathematica Society, 2008.

%\bibitem{BR} J. Berstel, C. Reutenauer, \emph{Rational Series and Their Languages}.
%		Monographs in Theoretical Computer Science, Volume 12. Springer-Verlag, Berlin, 1988.

%\bibitem{Bir} G. Birkhoff, \emph{Combinatorial relations in projective geometries},
%		Ann. of Math. {\bf 36} (1935), 743--748.

%\bibitem{CL} T. M. Carlsen,  N. S. Larsen, \emph{Partial actions and KMS states on relative graph $C^*$-algebras},
%		J. Funct. Anal. {\bf 271} (2016), 2090--2132.

\bibitem{Claramunt} J. Claramunt, \emph{Sylvester matrix rank functions on crossed products and the Atiyah problem}, 
Ph.D. Thesis, Universitat Aut\`{o}noma de Barcelona, 2018. 

%\bibitem{CS} J. Claramunt, A. Sims, \emph{Preferred traces on $C^*$-algebras of self-similar groupoids arising as fixed points},
%		J. Math. Anal. Appl. {\bf 466} (2018), 806--818.

%\bibitem{CEPSS} L. O. Clark, R. Exel, E. Pardo, A. Sims, C. Starling, \emph{Simplicity of algebras associated to non-Hausdorff groupoids},
%		preprint 2018 (arXiv:1806.04362 [math.OA]).

%\bibitem{Cohn85} P. M. Cohn, \emph{Free Rings and Their Relations}.
%		Second Edition, LMS Monographs, Volume 19. Academic Press, London, 1985.

%\bibitem{Cohn91} P. M. Cohn, \emph{Algebra. Vol. 3}.
%		Second edition, John Wiley \& Sons, Ltd., Chichester, 1991.

%\bibitem{Cohn06} P. M. Cohn, \emph{Free Ideal Rings and Localization in General Rings}. 
%		New Mathematical Monographs. Cambridge Univ. Press, cambridge, 2006.

%\bibitem{Con} A. Connes, \emph{Classification of injective factors. Cases II$_1$, II$_{\infty}$, III$_{\lambda}$, $\lambda \ne 1$},
%		Ann. of Math. (2) {\bf 104} (1976), 73--115.

%\bibitem{Con79} A. Connes, \emph{On the equivalence between injectivity and semidiscreteness for operator algebras, in “Algèbres d’opérateurs et leurs application en physique théorique”},
%		Edition CNRS (1979), 107--112.

%\bibitem{CT} J. Christensen, K. Thomsen, \emph{Finite digraphs and KMS states},
%		J. Math. Anal. Appl. {\bf 433} (2016), 1626--1646.

\bibitem{Dav} K. R. Davidson, \emph{$C^*$-Algebras by Example},
		Fields Institute monographs. American Mathematical Soc., 1996.

\bibitem{Davis} M. W. Davis, \emph{The Hopf Conjecture and the Singer Conjecture},
		Enseign. Math. \textbf{54} (2008), 76--78.

\bibitem{DiSc} W. Dicks, T. Schick, \emph{The Spectral Measure of Certain Elements of the Complex Group Ring of a Wreath Product},
		Geometriae Dedicata {\bf 93} (2002), 121--137.

%\bibitem{Dix} J. Dixmier, \emph{$C^*$-algebras},
%		North-Holland mathematical library, Volume 15. North-Holland, 1982.

%\bibitem{Dod} J. Dodziuk, \emph{de Rham-Hodge theory for $L^2$-cohomology of infinite coverings},
%		Topology \textbf{16} (1977), 157--165.

\bibitem{DLMSY} J. Dodziuk, P. Linnell, V. Mathai, T. Schick, S. Yates, \emph{Approximating $L^2$-invariants and the Atiyah conjecture},
		Comm. Pure Appl. Math. \textbf{56} (2003), 839--873.

%\bibitem{DS} N. Dunford, J. T. Schwartz, \emph{Linear Operators, Part I}.
%    Interscience, New York, 1958.

%\bibitem{Elek13} G. Elek, \emph{Connes embeddings and von Neumann regular closures of amenable group algebras},
%		Trans. Amer. Math. Soc. {\bf 365} (2013), 3019--3039.

\bibitem{Elek16} G. Elek, \emph{Lamplighter groups and von Neumann continuous regular rings},
		Proc. Amer. Math. Soc. {\bf 144} (2016), 2871--2883.

%\bibitem{Elek17} G. Elek, \emph{Infinite dimensional representations of finite dimensional algebras and amenablility},
%		Math. Annalen {\bf 369} (2017), 397--439. 

%\bibitem{EFW} M. Enomoto, M. Fujii, Y. Watatani, \emph{KMS states for gauge action on $O_{A}$},
%		Math. Japon. {\bf 29} (1984), 607--619.

%\bibitem{Exel} R. Exel, \emph{Partial Dynamical Systems, Fell Bundles and Applications}.
%		Mathematical Surveys and Monographs, Volume 224. American Mathematical Society, 2017.

%\bibitem{ELQ} R. Exel, M. Laca, J. Quigg, \emph{Partial Dynamical Systems and $C^*$-Algebras Generated by Partial Isometries},
%		J. Op. Theor. {\bf 47} (2002), 169--186.

%\bibitem{FGKP} C. Farsi, E. Gillaspy, S. Kang, J. A. Packer, \emph{Separable representations, KMS states, and wavelets for higher-rank graphs},
%		J. Math. Anal. Appl. {\bf 434} (2016), 241--270.

%\bibitem{Fol84} G. B. Folland, \emph{Real Analysis: Modern Techniques and Their Applications},
%		Pure and Applied Mathematics (New York). Wiley Interscience, New York, 1984.
		
%\bibitem{Fol16} G. B. Folland, \emph{A Course in Abstract Harmonic Analysis}.
%		Textbooks in Mathematics, Volume 29. CRC Press, 2016.

%\bibitem{FR} N. J. Fowler, I. Raeburn, \emph{The Toeplitz Algebra of a Hilbert Bimodule},
%		Indiana Univ. Math. J. {\bf 48} (1999), 155--181.

%\bibitem{GKP} T. Giordano, D. Kerr, N. C. Phillips, A. Toms, \emph{Crossed products of $C^*$-algebras, topological dynamics, and classification},
%		Advanced Courses in Mathematics, CRM Barcelona. Birkhäuser/Springer, Cham, 2018.

%\bibitem{GPS} T. Giordano, I. F. Putnam, C. F. Skau, \emph{Topological orbit equivalence and $C^*$-crossed products},
%		J. Reine Angew. Math. {\bf 469} (1995), 51--111.

%\bibitem{Goo78} K. R. Goodearl, \emph{Centers of regular self-injective rings},
%		Pacific J. Math. {\bf 76} (1978), 381--395.

\bibitem{Goo91} K. R. Goodearl, \emph{Von Neumann Regular Rings}.
		Pitman, London, 1979; Second Edition, Krieger, Malabar, Fl., 1991.

\bibitem{Gra14} L. Grabowski, \emph{On Turing dynamical systems and the Atiyah problem},
		Invent. Math. {\bf 198} (2014), 27--69.

\bibitem{Gra16} L. Grabowski, \emph{Irrational $\ell^2$-invariants	arising from the lamplighter group},
		Groups Geom. Dyn. {\bf 10} (2016), 795--817.

\bibitem{GZ} R. I. Grigorchuk, A. Żuk, \emph{The Lamplighter Group as a Group Generated by a $2$-state Automaton, and its Spectrum},
		Geometriae Dedicata {\bf 87} (2001), 209--244.

\bibitem{GLSZ} R. I. Grigorchuk, P. A. Linnell, T. Schick, A. Żuk, \emph{On a conjecture of Atiyah},
		C.R. Acad. Sci. Paris {\bf 331} (2000), 663-668.

%\bibitem{Halp} I. Halperin, \emph{Von Neumann's manuscript on inductive limits of regular rings},
%		Canad. J. Math. {\bf 20} (1968), 477--483.

%\bibitem{Hand77} D. Handelman, \emph{Completions of rank rings},
%		Canad. Math. Bull. {\bf 20} (1977), 199--205. 

%\bibitem{Hand78} D. Handelman, \emph{Coordinatization applied to finite Baer *-rings},
%		Trans. Amer. Math. Soc. {\bf 235} (1978), 1--34.

%\bibitem{Hand81} D. Handelman, \emph{Rings with involution as partially ordered abelian groups},
%		Rocky Mountain J. Math. {\bf 11} (1981), 337--381.

%\bibitem{HPS92}  R. H. Herman, I. F. Putnam, C. F. Skau, \emph{Ordered Bratteli diagrams, dimension groups and topological dynamics},
%		Internat. J. Math.  {\bf 3} (1992), 827--864.

%\bibitem{aHLRS13} A. an Huef, M. Laca, I. Raeburn, A. Sims, \emph{KMS States on the $C^*$-Algebras of Finite Graphs},
%		J. Math. Anal. Appl. {\bf 405} (2013), 388--399.

%\bibitem{aHLRS15} A. an Huef, M. Laca, I. Raeburn and A. Sims, \emph{KMS states on the $C^*$-algebra of a higher-rank graph and periodicity in the path space},
%		J. Funct. Anal. {\bf 268} (2015), 1840--1875.

%\bibitem{IK} M. Ionescu, A. Kumjian, \emph{Hausdorff Measures and KMS States},
%		Indiana Univ. Math. J. {\bf 62} (2013), 443-463.

\bibitem{Jaik-survey} A. Jaikin-Zapirain, \emph{$L^2$-Betti Numbers and their Analogues in Positive Characteristic},
		Groups St Andrews 2017 in Birmingham, Birmingham (2017), London Math. Soc. Lecture Note Ser., Cambridge University Press \text{455} (2019), 346--405.

\bibitem{Jaik}  A. Jaikin-Zapirain, \emph{The base change in the Atiyah and the L\"uck approximation conjectures},
		Geom. Funct. Anal. {\bf 29} (2019), 464--538.

%\bibitem{JL}  A. Jaikin-Zapirain, D. L\'opez-\'Alvarez, \emph{On the space of Sylvester matrix rank functions},
%		in preparation.

\bibitem{JL2} A. Jaikin-Zapirain, D. L\'opez-\'Alvarez, \emph{The strong Atiyah Conjecture for one-relator groups},
		Math. Ann. {\bf 376} (2020), 1741--1793. 

%\bibitem{KRI} R. V. Kadison, J. R. Ringrose, \emph{Fundamentals of the Theory of Operator Algebras}.
%		Volume $I$ \emph{Elementary Theory}, Academic Press Inc., London, 1983.

%\bibitem{KR} R. V. Kadison, J. R. Ringrose, \emph{Fundamentals of the Theory of Operator Algebras}.
%		Volume $II$ \emph{Advanced Theory}, Academic Press Inc., London, 1986.

%\bibitem{Kak} E. T. A. Kakariadis, \emph{KMS states on Pimsner algebras associated with $C^*$-dynamical systems},
%		J. Funct. Anal. {\bf 269} (2015), 325--354.

%\bibitem{Kap} I. Kaplansky, \emph{Projections in Banach Algebras},
%		Ann. of Math. (2) {\bf 53} (1951), 235--249.

\bibitem{KM} A. S. Kechris, B.D. Miller, \emph{Topics in orbit equivalence}.
		Lecture Notes in Mathematics, 1852. Springer-Verlag, Berlin, 2004.
		
%\bibitem{Kei} K. Keimel, \emph{Algèbres commutatives engendrées par leurs éléments idempotents},
%		Canad. J. Math. {\bf 22} (1970), 1071--1078.

%\bibitem{KL} D. Kerr, H. Li, \emph{Ergodic theory. Independence and dichotomies}.
%		Springer Monographs in Mathematics, Springer, Cham 2016.

\bibitem{KLL} P. Kropholler, P. Linnell, W. Lück, \emph{Groups of small homological dimension and the Atiyah conjecture},
		Proceedings of the Symposium ``Geometry and Cohomology in Group Theory'', Durham (2003), London Math. Soc. Lecture Note Ser., Cambridge University Press \textbf{358} (2009), 272--277.

%\bibitem{LRRW14} M. Laca, I. Raeburn, J. Ramagge, M. F. Whittaker, \emph{Equilibrium States on the Cuntz--Pimsner Algebras of Self-Similar Actions},
%		J. Funct. Anal. {\bf 266} (2014), 6619--6661.

%\bibitem{LRRW16} M. Laca, I. Raeburn, J. Ramagge, M. F. Whittaker, \emph{Equilibrium states on operator algebras associated to self-similar actions of groupoids on graphs},
%		Adv. Math. {\bf 331} (2018), 268--325.

%\bibitem{Lam} T. Y. Lam, \emph{Lectures on Modules and Rings}.
%		Graduate texts in mathematics, Volume 189. Springer-Verlag, New York-Berlin, 1999.

%\bibitem{Law} M. V. Lawson, \emph{Inverse Semigroups: The Theory of Partial Symmetries}.
%		World Scientific Publishing Co. Pte. Ltd., 1999.

\bibitem{LLR} F. Larsen, N. Laustsen, M. R\o rdam, \emph{An Introduction to $K$-Theory for $C^*$-Algebras},
		London Math. Soc. Student Texts. Cambridge Univ. Press, Cambridge, 2000.

\bibitem{Lin91} P. A. Linnell, \emph{Zero Divisors and Group von Neumann Algebras},
		Pacific J. Math. {\bf 149} (1991), 349--363.

\bibitem{Lin93} P. A. Linnell, \emph{Division Rings and Group von Neumann Algebras},
		Forum Math. {\bf 5} (1993), 561--576.

%\bibitem{Lin04} P. A. Linnell, \emph{Noncommutative localization in group rings},
%		London Math. Soc. Lecture Note Ser., 330. Cambridge Univ. Press, Cambridge, 2006.

\bibitem{Lin08} P. A. Linnell, \emph{Embedding Group Algebras into Finite von Neumann Regular Rings},
		Trends Math. \emph{Modules and Comodules} (2008), 295--300.

\bibitem{LLS} P. A. Linnell, W. Lück, T. Schick, \emph{The Ore Condition, Affiliated Operators, and the Lamplighter Group},
		Proceedings of ICTP Trieste conference on High dimensional manifold topology (2001), 315--321.

\bibitem{LS07} P. A. Linnell, T. Schick, \emph{Finite Group Extensions and the Atiyah Conjecture},
		J. Amer. Math. Soc. {\bf 20} (2007), 1003--1051.

\bibitem{LS12} P. A. Linnell, T. Schick, \emph{The Atiyah Conjecture and Artinian Rings},
		Pure and Appl. Math. {\bf 8} (2012), 313--327.

\bibitem{Luck} W. Lück, \emph{$L^2$-Invariants: Theory and Applications to Geometry and $K$-Theory}.
		A Series of Modern Surveys in Matematics, Volume 44. Springer-Verlag, Berlin, 2002.

\bibitem{LL18} W. Lück, P. A. Linnell, \emph{Localization, Whitehead groups and the Atiyah conjecture},
		Annals of $K$-Theory {\bf 3} (2018), 33-53.

\bibitem{Mal} P. Malcolmson, \emph{Determining homomorphisms to skew fields},
		J. Algebra {\bf 64} (1980), 399--413.

%\bibitem{Men} K. Menger, \emph{New foundations of projective and affine geometry},
%		Ann. of Math. {\bf 37} (1936), 456--482

%\bibitem{Murph} G. J. Murphy, \emph{$C^*$-Algebras and Operator Theory}.
%		Academic Press Inc., Boston, MA, 1990.

%\bibitem{MvN36} F. J. Murray, J. von Neumann, \emph{On rings of operators},
%		Ann. of Math. {\bf 37} (1936), 116--229.

%\bibitem{MvN37} F. J. Murray, J. von Neumann, \emph{On rings of operators. II},
%		Trans. Amer. Math. Soc. {\bf 41} (1937), 208--248.

%\bibitem{MvN40} F. J. Murray, J. von Neumann, \emph{On rings of operators. III},
%		Ann. of Math. {\bf 41} (1940), 94--161. 

%\bibitem{MvN43} F. J. Murray, J. von Neumann, \emph{On rings of operators. IV},
%		Ann. of Math. {\bf 44} (1943), 716--808. 

%\bibitem{Nek} V. Nekrashevych, \emph{Self-similar groups}.
%		Mathematical Surveys and Monographs, Volume 12. Illustrated Edition, AMS, 2005.

%\bibitem{vN36} J. von Neumann, \emph{Continuous Geometry},
%		Proc. Natl. Acad. Sci. USA {\bf 22} (1936), 92--100.

%\bibitem{vN36-2} J. von Neumann, \emph{Examples of Continuous Geometry},
%		Proc. Natl. Acad. Sci. USA {\bf 22} (1936), 101--108.

%\bibitem{vN36-3} J. von Neumann, \emph{On Regular Rings},
%		Proc. Natl. Acad. Sci. USA {\bf 22} (1936), 707--713.

%\bibitem{vN37} J. von Neumann, \emph{Algebraic Theory of Continuous Geometries},
%		Proc. Natl. Acad. Sci. USA {\bf 23} (1937), 16--22.

%\bibitem{vN58} J. von Neumann, \emph{The Non-Isomorphism of Certain Continuous Rings},
%		Ann. of Math. {\bf 67} (1958), 485--496.

%\bibitem{vN} J. von Neumann, \emph{Continuous geometry}.
%		Princeton Landmarks in Mathematics. Princeton University Press, New Jersey, 1988.

%\bibitem{Ped} G. K. Pedersen, \emph{$C^*$-Algebras and Their Automorphism Groups}.
%		LMS Monographs, Volume 14. Academic Press, London, 1979.

\bibitem{PSZ} M. Pichot, T. Schick, A. Żuk, \emph{Closed manifolds with transcendental $L^2$-Betti numbers},
		J. London Math. Soc. {\bf 92} (2015), 371--392.

%\bibitem{P89} I. F. Putnam, \emph{The $C^*$-algebras associated with minimal homeomorphisms of the Cantor set},
%		Pacific J. Math. {\bf 136} (1989), 329--353.

%\bibitem{P90} I. F. Putnam, \emph{On the topological stable rank of certain transformation group $C^*$-algebras},
%		Ergodic Theory Dynam. Systems {\bf 10} (1990), 197--207.

%\bibitem{Rei} H. Reich, \emph{Group von Neumann Algebras and Related Algebras},
%		Ph.D. Thesis, Göttingen 1998.

%\bibitem{RS} M. R\o rdam, A. Sierakowski, \emph{Purely infinite $C^*$-algebras arising from crossed products},
%		Erg. Theor. Dynam. Syst. {\bf 32} (2012), 273--293.

\bibitem{Ros} J. Rosenberg, \emph{Algebraic K-theory and its applications}. Graduate Texts in Mathematics, 147. Springer-Verlag, New York, 1994.


%\bibitem{Rud} W. Rudin, \emph{Real and complex analysis},
%		New York: McGraw-Hill, 1987.

\bibitem{Schick}  T. Schick, \emph{Integrality of $L^2$-Betti numbers},
		Math. Ann. {\bf 317} (2000), 727--750.

\bibitem{Sch} A. H. Schofield, \emph{Representation of rings over skew fields}.
		London Math. Soc. Lecture Note Ser., 92. Cambridge Univ. Press, Cambridge, 1985.

%\bibitem{Sen} E. Seneta, \emph{Non-negative Matrices and Markov Chains}.
%		Second Edition, Springer-Verlag, New York, 1981.

%\bibitem{SS} A. M. Sinclair, R. R. Smith, \emph{Finite von Neumann algebras and masas}.
%		London Math. Soc. Lecture Note Ser., 351. Cambridge Univ. Press, Cambridge, 2008.

%\bibitem{Stein} B. Steinberg, \emph{Simplicity, primitivity and semiprimitivity of etale groupoid algebras with applications to inverse semigroup algebras},
%		J. Pure Appl. Alg. {\bf 220} (2016), 1035--1054.

\bibitem{Tan} T. Tanaka, \emph{Transcendence of the values of certain series with Hadamard's gaps},
		Arch. Math. (Basel) {\bf 78} (2002), 202--209.

%\bibitem{Thom} A. Thom, \emph{Sofic Groups and Diophantine Approximation},
%		Comm. Pure Appl. Math. {\bf 61} (2008), 1155--1171.

%\bibitem{Tho} K. Thomsen, \emph{KMS weights on graph $C^*$-algebras},
%		Adv. Math. {\bf 309} (2017), 334--391.

%\bibitem{Utu} Y. Utumi, \emph{On Continuous Rings and Self Injective Rings},
%		Trans. Amer. Math. Soc. {\bf 118} (1965), 158--173.

\end{thebibliography}
\end{document}